\documentclass[10pt, oneside]{article} 
\usepackage{amsmath,amssymb}  
\usepackage{mathrsfs,amsthm}
\usepackage{graphicx}
\usepackage{stmaryrd}
\usepackage[all]{xy}
\usepackage{amscd}
\usepackage{geometry}
\makeatletter
\@addtoreset{equation}{section}

\makeatother

\title{On logarithmic nonabelian Hodge theory of higher level in characteristic $p$}
\author{Sachio Ohkawa}
\date{}

\begin{document}

\maketitle
\begin{abstract}
Given a natural number $m$ and a log smooth integral morphism $X\to S$ of fine log schemes of characteristic $p>0$ with a lifting of its Frobenius pull-back $X'\to S$ modulo $p^{2}$, we use indexed algebras ${\cal A}_{X}^{gp}$, ${\cal B}_{X/S}^{(m+1)}$ of Lorenzon-Montagnon and the sheaf ${\cal D}_{X/S}^{(m)}$ of log differential operators of level $m$ of Berthelot-Montagnon to construct an equivalence between the category of certain indexed ${\cal A}^{gp}_{X}$-modules with ${\cal D}_{X/S}^{(m)}$-action and the category of certain indexed ${\cal B}_{X/S}^{(m+1)}$-modules with Higgs field.
Our result is regarded as a level $m$ version of some results of Ogus-Vologodsky and Schepler.
\end{abstract}
\theoremstyle{plain}
\newtheorem{theo}{Theorem}[section]
\newtheorem{defi}[theo]{Definition}
\newtheorem{lemm}[theo]{Lemma}
\newtheorem{pro}[theo]{Proposition}
\newtheorem{cor}[theo]{Corollary}
\newtheorem{Prod}[theo]{Proposition-Definition}

\theoremstyle{definition}
\newtheorem{rem}[theo]{Remark}
\newtheorem{ex}[theo]{Example}

\renewcommand*\proofname{\upshape{\bfseries{Proof}}}

\section{Introduction}
For a projective smooth complex algebraic variety, Simpson \cite{Si} established a correspondence, which is called the Simpson correspondence nowadays, between local systems and Higgs bundles.
In \cite{OV}, Ogus and Vologodsky studied an analogue of the Simpson correspondence for certain integrable connections or equivalently certain $\cal D$-modules in positive characteristic.
As a natural generalization of their theory,
Schepler \cite{S} studied its log version and Gros, Le Stum and Quir\'os \cite{GLQ} studied its higher level version.
The aim of this article is to establish the log and higher level version of the theory of Ogus-Vologodsky.

Let us recall the Ogus-Vologodsky's analogue of the Simpson correspondence in positive characteristic, which is called the global Cartier transform (see Theorem 2.8 of \cite{OV}).
Let $X\to S$ be a smooth morphism of schemes of characteristic $p>0$.
Let us denote by $X'$ the pull-back of $X\to S$ via the absolute Frobenius endomorphism $F_{S}$ of $S$.
Denote the relative Frobenius morphism $X\to X'$ by $F_{X/S}$.
Let ${\cal T}_{X'/S}$ be the tangent bundle of $X'$ over $S$,
$S^{\cdot}{\cal T}_{X'/S}$ the symmetric algebra of ${\cal T}_{X'/S}$ and
$\cal G$ the nilpotent divided power envelope of the zero section of the cotangent bundle of $X'/S$, so that ${\cal O_{G}}=\hat{\Gamma}.{\cal T}_{X'/S}$.
Assume that we are given a lifting of $X'\to S$ modulo $p^{2}$.
Then there exists an equivalence between the category of ${\cal O}_{X}$-modules $E$ with integrable connection $\nabla$ equipped with a horizontal ${\cal O}_{X}$-linear $\cal G$-Higgs field $\theta:{\cal O}_{\cal G}\to F_{X/S*}{\cal E}nd_{{\cal O}_{X}}(E, \nabla)$ extending the horizontal map $\psi: S^{\cdot}{\cal T}_{X'/S}\to F_{X/S*}{\cal E}nd_{{\cal O}_{X}}(E, \nabla)$ given by the $p$-curvature
and the category of ${\cal O}_{X'}$-modules $E'$ equipped with an ${\cal O}_{X'}$-linear $\cal G$-Higgs field $\theta: {\cal O}_{\cal G}\to {\cal E}nd_{{\cal O}_{X'}}(E')$.
There are two key technical results for the proof of the global Cartier transform.
One is the fact that the sheaf ${\cal D}_{X/S}^{(0)}$ of differential operators of level 0 on $X$ is an Azumaya algebra over its center, which is isomorphic to $S^{\cdot}{\cal T}_{X'/S}$ via the $p$-curvature map.
The other is a construction of the splitting module $\check{\cal K}_{\cal X/S}$ of this Azumaya algebra over the scalar extension  ${\cal O}_{\cal G}$ of $S^{\cdot}{\cal T}_{X'/S}$.
This means an isomorphism of ${\cal O}_{X}$-algebras 
${\cal D}_{X/S}^{(0)}\otimes_{S^{\cdot}{\cal T}_{X'/S}}{\cal O}_{\cal G}\xrightarrow{\cong}{\cal E}nd_{{\cal O}_{\cal G}}\left( \check{\cal K}_{\cal X/S}\right)$.
Then the global Cartier transform can be obtained by the Morita equivalence.
Ogus-Vologodsky also constructed the splitting module over the completion $\hat{S^{\cdot}}{\cal T}_{X'/S}$ of $S^{.}{\cal T}_{X'/S}$ under the assumption of an existence of a mod $p^{2}$ lifting of $F_{X/S}$ and got an analogous equivalence called the local Cartier transform.

As is mentioned in the first paragraph, the theory of Ogus-Vologodsky has been generalized in (at least) two directions.
First, Schepler \cite{S} extended their theory to the case of log schemes.
The difficulty for this generalization is that the Azumaya nature of the sheaf ${\cal D}_{X/S}^{(0)}$ of the log differential operators of level $0$ is no longer true in general.
Schepler overcame this difficulty by using Lorenzon's theory of indexed modules and indexed algebras ${\cal A}_{X}^{gp}$ and ${\cal B}_{X/S}$ associated to a log scheme $X$ and its Frobenius pullback $X'\to S$.
Roughly speaking, ${\cal A}_{X}^{gp}$ and ${\cal B}_{X/S}$ are the suitable scalar extensions of the structure sheaf ${\cal O}_{X}$ and ${\cal O}_{X'}$ respectively in log case.
He used the sheaf $\tilde{{\cal D}}_{X/S}^{(0)}:={\cal A}^{gp}_{X}\otimes_{{\cal O}_{X}}{\cal D}_{X/S}^{(0)}$ in place of ${\cal D}_{X/S}^{(0)}$
and proved the Azumaya nature of $\tilde{{\cal D}}_{X/S}^{(0)}$ over its center.
Schepler also generalized the splitting module $\check{\cal K}_{\cal X/S}$ of Ogus-Vologodsky and got the log global Cartier transform. 
Second, in \cite{GLQ}, Gros, Le Stum and Quir\'os extended some results in \cite{OV} to the case of Berthelot's ring of differential operators of higher level \cite{B2}.
They proved the Azumaya nature of the sheaf ${\cal D}_{X/S}^{(m)}$ of differential operators of level $m$,
constructed a splitting module of ${\cal D}_{X/S}^{(m)}$ over $\hat{S^{\cdot}}{\cal T}_{X'/S}$ (here $X'$ denotes the pull-back of $X\to S$ by the $(m+1)$-st iterate of the absolute Frobenius $F_{S}:S\to S$) under the assumption of an existence of a good lifting of the ($m+1$)-st relative Frobenius morphism $F_{X/S}$ mod $p^{2}$, which they call a strong lifting, and proved the local Cartier transform of higher level.
They also constructed (but informally) a global splitting module by a gluing argument.
But their construction is different from that of Ogus-Vologodsky.
It should be remarked here that the sheaf $\tilde{{\cal D}}_{X/S}^{(0)}$ used by Schepler (or more generally the sheaf  $\tilde{{\cal D}}_{X/S}^{(m)}$ of log differential operators of higher level) is introduced by Montagnon \cite{M}. 
She established there the foundations of log differential operators of higher level and  
especially obtained the log version of  Berthelot's Frobenius descent by using the indexed algebras ${\cal A}^{gp}_{X}$ and ${\cal B}^{(m)}_{X/S}$,
where the latter denotes the higher level version of Lorsntzon's ${\cal B}_{X/S}$.

The purpose of this paper is to generalize Schepler's log global Cartier transform to the case of higher level by using the indexed algebras ${\cal A}_{X}^{gp}$ and ${\cal B}_{X/S}^{(m+1)}$ of Lorenzon and Montagnon.
Our construction is a natural generalization of Ogus-Vologodsky and Schepler, but we also need some log differential calculus of higher level which is based on Montagnon's result. 
We also prove the compatibility of the log global Cartier transform with Montagnon's log Frobenius descent.

Let us describe the content of each section.
We work with a log smooth morphism $X\to S$ of fine log schemes in positive characteristic.
Let $F_{X/S}$ denote the ($m+1$)-st relative Frobenius $X\to X'$.
In the second section, we review the theory of indexed modules.
In the third section, we construct the log version of the higher curvature map $\beta :{\cal T}_{X'/S}\to F_{X/S*}{\cal D}_{X/S}^{(m)}$, which we call the $p^{m+1}$-curvature map, in Definition \ref{defi1}
after reviewing the theory of log differential operators of level $m$.
In the fourth section, after reviewing the construction and some basic results of indexed algebras associated to the log structure, we study the Azumaya nature of $\tilde{{\cal D}}_{X/S}^{(m)}$.
We prove that ${\cal B}^{(m+1)}_{X/S}\otimes_{{\cal O}_{X'}}S^{.}{\cal T}_{X'/S}$ is identified with the center of $\tilde{{\cal D}}_{X/S}^{(m)}$ via the $p^{m+1}$-curvature map (see Theorem \ref{theorem13})
and $\tilde{{\cal D}}_{X/S}^{(m)}$ is an Azumaya algebra over ${\cal B}^{(m+1)}_{X/S}\otimes_{{\cal O}_{X'}}S^{.}{\cal T}_{X'/S}$ (see Corollary \ref{cor1}).
We also prove the log Cartier descent theorem of higher level as an application (see Theorem \ref{theorem40}).
In the fifth section, we construct the splitting module $\check{{\cal K}}_{\cal X/S}^{(m), {\cal A}}$ of $\tilde{{\cal D}}_{X/S}^{(m)}$ over ${\cal B}^{(m+1)}_{X/S}\otimes_{{\cal O}_{X'}}\hat{\Gamma}_{.}{\cal T}_{X'/S}$ 
under the assumption of an existence of a mod $p^{2}$ lifting of $X'\to S$ (see (\ref{AA}))
and get the log global Cartier transform of higher level by using the indexed variant of Morita equivalence (see Theorem \ref{Theorem21}).
In the final section, we consider the compatibility of the log global Cartier transform with  Montagnon's log Frobenius descent. 
Our new ingredient is to prove the behavior of the splitting module $\check{{\cal K}}_{\cal X/S}^{(m), {\cal A}}$ with respect to the Frobenius descent functor of Montagnon (see Theorem \ref{theorem11}).
As a consequence of Theorem \ref{theorem11}, we obtain the expected compatibility (see Theorem \ref{theorem10}).

\subsection{Acknowledgments} 
This paper is a revised version of the author's master thesis in the University of Tokyo.
The author expresses his hearty thanks to Professor Atsushi Shiho for suggesting the topics, helpful discussions, carefully reading the draft of this paper, pointing out a lot of mistakes on it and valuable advice. Without his advice, this paper cannot be realized.
This paper owes its existence to the work of Lorenzon \cite{L}, Montagnon \cite{M}, Ogus-Vologodsky \cite{GLQ}  and Schepler \cite{S}.
He thanks them heartily. 
This work was supported by the Program for Leading Graduate 
Schools, MEXT, Japan.
 
\subsection{Conventions}
Throughout this article, we fix a prime number $p$ and a natural number $m$.
For a natural number $k$, $q_{k}$ denotes the unique natural number which satisfy $k=q_{k}p^{m}+r$ for some $0\leq r<p^{m}$. For natural numbers $k, k'$ and $k''$ such that $k=k'+k''$, we put
\begin{equation*}
\binom{k}{k'}:=\frac{k!}{k'!k''!},\,\, \left\{
\begin{array}{c}
k\\
k'
\end{array}
\right\}:=\frac{q_{k}!}{q_{k'}!q_{k''}!} \text{\,\, and } \left\langle
\begin{array}{c}
k\\
k'
\end{array}
\right\rangle:=\binom{k}{k'}{\left\{
\begin{array}{c}
k\\
k'
\end{array}
\right\}}^{-1}.
\end{equation*}
Next we introduce some notations on multi-indices.
The element $(0, . . . , 1, . . . , 0)\in {\mathbb N}^{r}$, where 1 sits in the $i$-th entry, is denoted by $\underline{\varepsilon}_{i}$. When $\underline{k}$ is an element of ${\mathbb N}^{r}$, we denote its $i$-th entry by $k_{i}$, $k_{1}+\cdots +k_{r}$ by $|\underline{k}|$ and $\prod_{i=1}^{r}q_{k_{i}}$ by $\underline{q}_{\underline{k}}$. If $\underline{j}\leq \underline{i}$, that is, if $j_{k}\leq i_{k}$ for all $1\leq k \leq r$, we put
\begin{equation*}
\binom{\underline{i}}{\underline{j}}:=\prod_{k=1}^{r}\binom{i_{k}}{j_{k}},\,\, \left\{
\begin{array}{c}
\underline{i}\\
\underline{j}
\end{array}
\right\}:=\prod_{k=1}^{r}\left\{
\begin{array}{c}
i_{k}\\
j_{k}
\end{array}
\right\} \text{\,\, and } \left\langle
\begin{array}{c}
\underline{i}\\
\underline{j}
\end{array}
\right\rangle:=\prod_{k=1}^{r} \left\langle
\begin{array}{c}
i_{k}\\
j_{k}
\end{array}
\right\rangle.
\end{equation*}
\section{Indexed Azumaya Algebra}

In this section, we give a review of the theory of indexed modules and indexed Azumaya algebras developed by  (\cite{L}, see also \cite{S}) which we will use to construct the log global Cartier transform of higher level. The general theory of indexed modules can be developed on a ringed topos but, for simplicity, we only consider the case of the ringed topos associated to the \'etale site of a scheme and its structure sheaf. Also, we try to describe several notions more concretely than those given in \cite{L} and \cite{S}.
We fix throughout this section a scheme $X$ and an \'etale sheaf of abelian groups ${\cal I}$.
\subsection{Indexed module}
Let us recall some notions on indexed modules.  
\begin{defi}
{\rm (1)} An $\cal I$-indexed sheaf on $X$ is a sheaf of sets over $\cal I$, namely, a map of sheaves $\cal F \to \cal I$. We denote the map $\cal F \to \cal I$ by $p_{\cal F}$. An $\cal I$-indexed sheaf of abelian groups on $X$ is an $\cal I$-indexed sheaf on $\cal F\to I$ on $X$ equipped with an addition map $\cal F \times_{I}F\to F$ over $\cal I$, a unit map $\cal I\to F$ over $\cal I$ and an inverse map $\cal F\to F$ over $\cal I$ satisfying the usual axioms of abelian groups.

{\rm (2)} An $\cal I$-indexed ${\cal O}_{X}$-module is an $\cal I$-indexed sheaf of abelian groups equipped with  a scalar multiplication map ${\cal O}_{X}\times \cal F\to F$ over $\cal I$ satisfying the usual associativity, distributivity and unitarity conditions, where ${\cal O}_{X}\times \cal F$ is regarded as a sheaf over $\cal I$ via the composite ${\cal O}_{X}\times \cal F\to F\to I$.

{\rm (3)} An $\cal I$-indexed ${\cal O}_{X}$-algebra is an $\cal I$-indexed ${\cal O}_{X}$-module $\cal A$ equipped with an ${\cal O}_{X}$-bilinear multiplication map $\pi : {\cal A \times A\to A} $ over the addition map ${\cal I \times I\to I}$ and a global section $1_{\cal A}$ of $\cal A$ over the zero section $0 : e\to \cal I$ satisfying the usual associativity and unitarity conditions. We say an $\cal I$-indexed ${\cal O}_{X}$-algebra $\cal A$ is commutative if the multiplication map $\pi$ satisfies $\pi \circ \sigma = \pi$ where $\sigma$ is the isomorphism ${\cal A \times A\to A\times A} $ defined by $(a, b)\mapsto (b, a)$.

{\rm (4)} For an $\cal I$-indexed ${\cal O}_X$-algebra $\cal A$, 
an $\cal I$-indexed ${\cal A}$-algebra is an 
$\cal I$-indexed ${\cal O}_X$-algebra $\cal B$ equipped with a morphism 
${\cal A} \rightarrow {\cal B}$ of $\cal I$-indexed ${\cal O}_X$-algebras. 
\end{defi}
\begin{rem}
Let $\cal A$ be an $\cal I$-indexed sheaf on $X$. For an \'etale open $U$ of $X$ and a section $i \in {\cal I}(U)$, we denote by ${\cal A}_{i}$ the pullback $h_{U}\times_{\cal I}{\cal A}$ where $h_{U}$ is an \'etale sheaf on $X$ represented by $U$ and $h_{U} \to \cal I$ is the section $i$.
We call ${\cal A}_{i}$ the fiber of $\cal A\to I$ at $i\in {\cal I}(U)$.
Note that ${\cal A}_{i}$ is naturally considered as an \'etale sheaf on $U$, and moreover, if $\cal A$ is an $\cal I$-indexed ${\cal O}_{X}$-module, then  ${\cal A}_{i}$ has an ${\cal O}_{U}$-module structure naturally induced by the $\cal I$-indexed ${\cal O}_{X}$-module structure on $\cal A$.
If $\cal A$ is an $\cal I$-indexed ${\cal O}_{X}$-algebra, the multiplication map $\pi$ of $\cal A$ is equivalent to the following data: for each \'etale open $U$ of $X$ and sections $i, j\in {\cal I}(U)$, a morphism of  ${\cal O}_{U}$-modules $\pi_{ij}: {\cal A}_{i}\otimes_{{\cal O}_{U}}{\cal A}_{j} \to {\cal A}_{i+j}$ functorial with respect to $i, j$ satisfying the obvious conditions of associativity and unitarity.
\end{rem}
Now we recall the definition of $\cal J$-indexed $\cal A$-modules.
\begin{defi}
Let $\cal A$ be an $\cal I$-indexed ${\cal O}_{X}$-algebra. Let $\cal J$ be an \'etale sheaf of $\cal I$-sets, that is, an \'etale sheaf of sets on $X$ equipped with an $\cal I$-action map ${\cal I\times J\to J}; (i, j)\mapsto i+j$. A  $\cal J$-indexed left $\cal A$-module is a $\cal J$-indexed ${\cal O}_{X}$-module $\cal E$ equipped with an ${\cal O}_{X}$-bilinear map $\rho: {\cal A\times E\to E}$ over the $\cal I$-action map $\cal I\times J\to J$ satisfying the usual associativity and unitarity conditions. We can similarly define the notion of  $\cal J$-indexed right $\cal A$-module.
\end{defi}
\begin{rem}\label{Remark2}
Let $\cal A$ be an $\cal I$-indexed ${\cal O}_{X}$-algebra and  $\cal J$ be an \'etale sheaf of $\cal I$-sets on $X$. Let $\cal E$ be a $\cal J$-indexed left $\cal A$-module.
Then the structure morphism $\rho: {\cal A\times E\to E}$ over $\cal I\times J\to J$ is equivalent to the following data:
for each \'etale open $U$ of $X$, and each section $(i, j)\in {\cal I\times J}$, a morphism of ${\cal O}_{U}$-modules $\rho_{ij}: {\cal A}_{i}\otimes_{{\cal O}_{U}} {\cal E}_{j}\to {\cal E}_{i+j};$ $a\otimes e\mapsto ae$ functorial with respect to $i, j$ satisfying the obvious conditions of associativity and unitarity.
\end{rem}
Next we recall the definition of tensor products and internal hom objects as an indexed module.
\begin{defi}\label{defi3}
Let $\cal A$ be an $\cal I$-indexed ${\cal O}_{X}$-algebra and 
$\cal J$, $\cal K$ be \'etale sheaves of $\cal I$-sets on $X$.

{\rm (1)}  Let $\cal E$ be a $\cal J$-indexed right $\cal A$-module and $\cal F$ a $\cal K$-indexed left $\cal A$-module. Let ${\cal J} \otimes_{\cal I} 
{\cal K}$ be the $\cal I$-set
${\cal J} \times {\cal K}/\!\sim$, where $\sim$ is the equivalence relation 
generated by the relation 
$(i+j,k) \!\sim\! (j,i+k)$ for $i \in {\cal I}, j \in {\cal J}, 
k \in {\cal K}$. Then 
we define a ${\cal J} \otimes_{\cal I} {\cal K}$-indexed sheaf 
of abelian groups $\cal E\otimes_{A} F$ (the tensor product of $\cal E$ and $\cal F$) as the object representing the functor which sends ${\cal J} \otimes_{\cal I}{\cal K}$-indexed sheaf of abelian groups $\cal M$ to the set of biadditive
$\cal A$-balanced morphisms $\cal E\times F\to M$ over the natural projection $\cal J\times K\to J\otimes_{I}K$.
Concretely this is the \'etale sheaf on $X$ associated to the presheaf 
	\begin{equation*}
	U\longmapsto \bigsqcup_{l \in {\cal J} \otimes_{\cal I} {\cal K}(U)}\left(\bigoplus_{(j,k)=l} {\cal E}_{j}(U)\otimes_{{\cal O}_{X}(U)}{\cal F}_{k}(U)\right)\slash R
	\end{equation*}
endowed with natural projection to ${\cal J} \otimes_{\cal I} {\cal K}$,
where $R$ is the ${\cal O}_{X}(U)$-submodule generated by\\
$\left\{ xa\otimes y-x\otimes ay \biggl|
	\begin{array}{c}
	x\in {\cal E}(U),\,\, y\in {\cal F}(U) \,\,\text{and} \,\,a\in {\cal A}(U) \text{ satisfying }\\
	(p_{\cal E}(x) + p_{\cal A}(a), p_{\cal F}(y)) = l
	\end{array}
\right\}$.
When $\cal A$ is commutative, then $\cal E\otimes_{A} F$ naturally forms a 
${\cal J} \otimes_{\cal I} {\cal K}$-indexed $\cal A$-module.

{\rm (2)} For a ${\cal K}$-indexed left $\cal A$-module $\cal F$ and 
$\varphi \in Hom_{\cal I}(\cal J, \cal K)$, 
we define the ${\cal J}$-indexed $\cal A$-module 
${\cal F}(\varphi)$ by the \'etale sheaf 
${\cal F}\times_{{\cal K}, \varphi} {\cal J}$ with the second projection 
and $\cal A$-action via the action on $\cal F$. 

{\rm (3)} Let $\cal E$ be a $\cal J$-indexed left $\cal A$-module and $\cal F$ a $\cal K$-indexed left $\cal A$-module.
We define the internal hom object of $\cal E$ and $\cal F$, which we denote by ${\cal H}om_{\cal A}({\cal E}, {\cal F})$, 
as the \'etale sheaf on $X$
\begin{equation*}
U\longmapsto \bigsqcup_{\varphi\in {\cal H}om_{\cal I}
\left({\cal J}, {\cal K}\right)(U)} 
{\rm Hom}_{\cal A}\left({\cal E}|_{U}, {\cal F}|_{U}(\varphi)\right)
\end{equation*}
endowed with natural projection to ${\cal H}om_{\cal I}(\cal J, \cal K)$,
where $\rm Hom_{\cal A}$ denotes the set of homomorphism of 
$\cal J$-indexed $\cal A$-modules. When $\cal A$ is commutative, then ${\cal H}om_{\cal A}({\cal E}, {\cal F})$ naturally forms a ${\cal H}om_{\cal I}(\cal J, \cal K)$-indexed $\cal A$-module. Also, we denote ${\cal H}om_{\cal A}(\cal E, E)$ simply by ${\cal E}nd_{\cal A}(\cal E)$.
\end{defi}


Finally we recall the local freeness and faithful flatness as an  $\cal I$-indexed $\cal A$-module.

\begin{defi}
Let $\cal A$ be an $\cal I$-indexed ${\cal O}_{X}$-algebra and let $\cal B$ be a $\cal J$-indexed $\cal A$-algebra.

{\rm (1)} We say that an $\cal I$-indexed $\cal A$-module $\cal E$ is locally free of rank $k$ if \'etale localy on $X$ there exist sections $n_{1}, \ldots , n_{k}$ of ${\cal I} = {\cal H}om_{\cal I}({\cal I}, {\cal I})$ such that $\cal E$ is isomorphic to $\bigoplus_{i=1}^{k} {\cal A}(n_{i})$, where ${\cal A}(n_{i})$ are as in Definition \ref{defi3} $(2)$.

{\rm (2)} We say that $\cal B$ is faithfully flat over $\cal A$ if the functor ${\cal E}\mapsto \cal E\otimes_{\cal A}\cal B$ is exact and faithful.
\end{defi}

\subsection{Indexed Azumaya algebra}

The following proposition due to Schepler (see \cite{S}) is an index version of the Morita equivalence.

\begin{pro}\label{Proposition5}
Let $\cal A$ be a commutative $\cal I$-indexed ${\cal O}_{X}$-algebra. Let $\cal J$ be an \'etale sheaf of $\cal I$-sets on $X$ and $M$ be a locally free $\cal I$-indexed $\cal A$-module of finite rank. We denote ${\cal E}nd_{\cal A}(M)$ by $\cal E$ which is an $\cal I$-indexed ${\cal O}_{X}$-algebra in natural way. Then the functor $E \longmapsto M\otimes_{{\cal A}} E$ is an equivalence of categories between the category of  $\cal J$-indexed $\cal A$-modules and the category of  $\cal J$-indexed left $\cal E$-modules.
\end{pro}

\begin{proof}
The quasi-inverse of $E\longmapsto M\otimes_{{\cal A}} E$ is given by $F \longmapsto {\cal H}om_{{\cal A}}(M, F)$. For more details, see Theorem 2.2 of \cite{S}.
\end{proof}

Now let us recall the notion on indexed Azumaya algebra.

\begin{defi}
Let $\cal A$ be a commutative $\cal I$-indexed ${\cal O}_{X}$-algebra and $\cal E$ an $\cal I$-indexed $\cal A$-algebra. Then, for a commutative $\cal I$-indexed $\cal A$-algebra $\cal B$, $\cal E$ splits over $\cal B$ with splitting module $M$ if there exists an $\cal I$-indexed locally free $\cal B$-module $M$ of finite rank such that ${\cal E\otimes_{A}B\simeq E}nd_{\cal B}(M)$.
$\cal E$ is an Azumaya algebra over $\cal A$ of rank $r^{2}$ if there exists some commutative faithfully flat  $\cal I$-indexed $\cal A$-algebra $\cal B$ such that $\cal E$ splits over $\cal B$ with splitting module $M$ of rank $r$.
\end{defi}
If we know that $\cal E$ is an Azumaya algebra over $\cal A$, then we can find the splitting module of $\cal E$ over $\cal A$ in certain case by the following proposition.
\begin{pro}\label{Proposition7}
Let $\cal A$ be a commutative $\cal I$-indexed ${\cal O}_{X}$-algebra and $\cal E$ an Azumaya algebra over $\cal A$ of rank $r^{2}$. If there exists a locally free $\cal I$-indexed $\cal A$-module $M$ of rank $r$ with a structure of $\cal I$-indexed left $\cal E$-module compatible with the given $\cal I$-indexed $\cal A$-module structure, then $\cal E$ splits over $\cal A$ with splitting module $M$. 
\end{pro}
\begin{proof}
See Corollary 2.5 of \cite{S}.
\end{proof}

\section{The $p^{m+1}$-curvature map}
From this section, we are mainly concerned with log schemes. We denote a log scheme by a single letter such as $X$ and the log structure of $X$ by ${\cal M}_{X}$.
Our aim of this section is to construct the $p^{m+1}$-curvature map for a log smooth morphism $X\to S$ of fine log schemes defined over a field of positive characteristic.
We also give the local description of this map.

Let us first briefly recall the log version of Berthelot's theory of differential operators of higher level which is studied by Montagnon.
For more details, see \cite{B2} and \cite{M}.
Throughout this section, we denote by $p$ a fixed prime number and all log schemes are assumed to be defined over $\mathbb{Z}_{(p)}$.
	\subsection{Logarithmic differential operators of higher level} 
	We begin with the definition of the $m$-PD structure.
		\begin{defi}
		Let $m$ be a positive integer. Let $X$ be a log scheme and $I$ a quasi-coherent ideal of ${\cal O}_{X}$. A divided power structure of level $m$ $(m$-PD structure$)$ on $I$ is a divided power ideal 
		$(J, \gamma)$ of ${\cal O}_{X}$ such that 
		\begin{equation*}
I^{(p^{m})}+pI\subset J\subset I
\end{equation*}
and a divided power structure $\gamma$ on $J$ compatible with the unique one on $p{\mathbb Z}_{(p)}$. 
Here $I^{(p^{m})}$ denotes the ideal of ${\cal O}_{X}$ generated by $p^{m}$-th powers of all sections of $I$.
If $(J, \gamma)$ is an $m$-PD structure on $I$, we call $(I, J, \gamma)$ an  $m$-PD ideal of ${\cal O}_{X}$ and call $(X, I, J, \gamma)$ an $m$-PD log scheme. 
\end{defi}
Let $(X, I, J, \gamma)$ be an $m$-PD log scheme. For each natural number $k$, we define the map $I\to {\cal O}_{X}$; $f\mapsto f^{\{k\}(m)}$ by $f^{\{k\}(m)}:=f^{r}\gamma_{q}(f^{p^{m}})$ where $k=p^{m}q+r$ and $0\leq r< p^{m}$. These maps satisfy the following formulas (see p.13 of \cite{M}).
\begin{pro}\label{Proposition9}
Let $(X, I, J, \gamma)$ be an $m$-PD log scheme and $k, l$ be positive integers.

{\rm(1)} For any  $x\in I, x^{\{0\}(m)}=1, x^{\{1\}(m)}=x$,  and  $x^{\{k\}(m)}\in I.$  Moreover if $k\geq p^{m}$, then  $x^{\{k\}(m)}\in J.$ 

{\rm(2)} For any  $x\in I, a\in {\cal O}_{X}$,  $(ax)^{\{k\}(m)}=a^{k}x^{\{k\}(m)}$. 

{\rm(3)} For any $x, y\in I$, $(x+y)^{\{k\}(m)}$=$\sum_{k'+k''=k} \left\langle
	\begin{array}{c}
	k\\
	k'
	\end{array}
\right\rangle x^{\{k'\}(m)} y^{\{k''\}(m)}.$ 

{\rm(4)} For any $x\in I$, $q_{k}!x^{\{k\}(m)}=x^{k}$. 

{\rm (5)} For any $x\in I$, $(x^{\{k\}(m)})^{\{l\}(m)}=\frac{q_{kl}!}{(q_{k}!)^{l}q_{l}!}x^{\{kl\}(m)}$.
\end{pro}
In the following, we sometimes denote an element $f^{\{k\}(m)}$ simply by $f^{\{k\}}$, if there will be no confusions.
For a while, we fix  an $m$-PD fine log scheme $(S, \mathfrak a, \mathfrak b, \gamma)$ on which $p$ is locally nilpotent and a morphism $X\to S$ of fine log schemes.
We assume that $\gamma$ extends to $X$ (for definition, see \cite{B2} D\'efinition 1.3.2(1)).
Note that $\gamma$ always extends to $X$ in the case ${\mathfrak b}=(p)$ (see \cite{B2} D\'efinition 1.3.2(1)), which is the case of our interest.

To recall the sheaf of log differential operators of higher level, we need the log $m$-PD envelope of higher level.
The construction of the log $m$-PD envelope is the same as the classical case of level $0$, which we explain now:
Let $i: X\hookrightarrow Y$ be an immersion of fine log schemes over $S$.
\'Etale locally on $X$, we have a factorization $i=g\circ i'$ with an exact closed immersion $i': X\hookrightarrow Z$ and a log \'etale morphism $g: Z\to Y$.
Let $i'': X\hookrightarrow D$ be the usual $m$-PD envelope of $i'$ (for definition, see \cite{B2}), and endow $D$ with the inverse image log structure of $Z$.
Then, since $i''$ satisfies the obvious universal property, it descents to the exact closed immersion $X\hookrightarrow P_{X, (m)}(Y)$ with the $m$-PD structure globally on $X$.
$P_{X, (m)}(Y)$ is called the log $m$-PD envelope of $i:X\hookrightarrow Y$.

Let us consider the diagonal immersion $X\to X\times_{S}X$. We simply denote its log $m$-PD envelope by $P_{X/S, (m)}$ and the defining ideal of $X\hookrightarrow P_{X/S, (m)}$ by $\bar{I}$.
Then there exists the $m$-PD-adic filtration $\left\{\bar{I}^{\{n\}}\right\}_{n\in {\mathbb N}}$ associated to $\bar{I}$  (for definition, see \cite{B3} D\'efinition A.3) which satisfies the following property.
\begin{equation}
\text{If $x$ is a local section of $\bar{I}^{\{n\}}$, then $x^{\{k\}}$ is in $\bar{I}^{\{nk\}}$}.
\end{equation}
Let ${\cal P}_{X/S, (m)}$ denote the structure sheaf of $P_{X/S, (m)}$. 
For each natural number $n$, we denote ${\cal P}_{X/S, (m)}^{n}$ the quotient sheaf ${\cal P}_{X/S, (m)}\slash \overline{I}^{\{n+1\}}$ and ${P}_{X/S, (m)}^{n}$ the closed subscheme of $P_{X/S, (m)}$ defined by $\overline{I}^{\{n+1\}}$.
We have a sequence of surjective homomorphisms of sheaves

\begin{equation*}
\cdots \to {\cal P}_{X/S, (m)}^{n}\to {\cal P}_{X/S, (m)}^{n-1}\to \cdots \to {\cal P}_{X/S, (m)}^{1}\to {\cal P}_{X/S, (m)}^{0}.
\end{equation*}

Let $p_{0}$ and $p_{1}$ (resp. $p_{0}^{n}$ and $p_{1}^{n}$) denote the first and second projection $P_{X/S, (m)}\to X$ (resp. $P_{X/S, (m)}^{n}\to X$) respectively.
We regard ${\cal P}_{X/S, (m)}$ and ${\cal P}_{X/S, (m)}^{n}$ as an ${\cal O}_{X}$-algebra via the first projection.
	\begin{defi}
	Let $n, m$ be natural numbers.
	The sheaf of differential operators of level $m$ of order $\leq n$ is defined by 
		\begin{equation*}
		{\cal D}_{X/S, n}^{(m)}:={\cal H}om_{{\cal O}_{X}}({{\cal P}}_{X/S, (m)}^{n}, {\cal O}_{X}).
		\end{equation*}
	The sheaf of differential operators of level $m$ is defined by
		\begin{equation*}
		{\cal D}_{X/S}^{(m)}:=\bigcup_{n\in \mathbb{N}} {\cal D}_{X/S, n}^{(m)}.
		\end{equation*}
	\end{defi}
\begin{rem}
Since, for any $m'\geq m$, an $m$-PD ideal can be considered as an $m'$-PD ideal, $\left\{{\cal D}_{X/S}^{(m)}\right\}_{m\geq 0}$ naturally forms an inductive system.
\end{rem}
${\cal D}_{X/S}^{(m)}$ has the (non commutative) ring structure as follows.
By using the universality of $m$-PD envelope, we obtain the canonical homomorphism of ${\cal O}_{X}$-algebras
\begin{equation*}
{\delta_{m}^{n,n'}}:{\cal P}_{X/S, (m)}^{n+n'}\rightarrow {\cal P}_{X/S, (m)}^{n}\otimes_{{\cal O}_{X}}{\cal P}_{X/S, (m)}^{n'}
\end{equation*} 
for each natural number $n, n'$, which is induced by the projection $X\times_{S}X\times_{S}X\to X\times_{S}X$ to the first and the third factors (for precise definition of ${\delta_{m}^{n,n'}}$, see Subsection 2.3.2 of \cite{M}).
For each $\Phi \in{\cal D}_{X/S, n}^{(m)}$ and $\Psi\in {\cal D}_{X/S, n'}^{(m)}$, we define the product $\Phi\cdot \Psi$ in ${\cal D}_{X/S, n+n'}^{(m)}$ by
\begin{equation*}
{\cal P}_{X/S, (m)}^{n+n'}\xrightarrow {\delta_{m}^{n,n'}}{\cal P}_{X/S, (m)}^{n}\otimes_{{\cal O}_{X}}{\cal P}_{X/S, (m)}^{n'}\xrightarrow {{\rm id}\otimes \Psi} {\cal P}_{X/S, (m)}^{n}\xrightarrow {\Phi} {\cal O}_{X}.
\end{equation*}
This is well-defined and ${\cal D}_{X/S}^{(m)}$ forms a sheaf of non commutative  ${\cal O}_{X}$-algebras on $X$.

\begin{rem}\label{Remark8}
Let ${\cal E}$ be an ${\cal O}_{X}$-module.
Then a log $m$-PD stratification on $\cal E$ is a family of ${\cal P}_{X/S, (m)}^{n}$-linear isomorphisms $\varepsilon_{n}: p_{0}^{n*}{\cal E}\xrightarrow{\cong }p_{1}^{n*}{\cal E}$ satisfying the usual cocycle conditions.
As is the same with the classical case, giving a ${\cal D}_{X/S}^{(m)}$-action on ${\cal E}$ extending its ${\cal O}_{X}$-module structure is equivalent to the data of a log $m$-PD stratification on ${\cal E}$. 
\end{rem}

Finally we recall the local description of ${\cal D}_{X/S}^{(m)}$ when $X\to S$ is a log smooth morphism of fine log schemes.
Let $j$ denote the log $m$-PD envelope $X\hookrightarrow P_{X/S,(m)}$ of the diagonal $X\to X\times_{S}X$.
We have an exact sequence 
\begin{equation}\label{exact1}
0\to j^{-1}(1+\bar{I})\xrightarrow{\lambda}j^{-1}{\cal M}_{P_{X/S, (m)}}\xrightarrow{j^{*}}{\cal M}_{X}\to0,
\end{equation}
where $\lambda$ is the restriction of log structure $j^{-1}(\alpha_{P_{X/S, (m)}}^{-1}): j^{-1}({{\cal P}_{X/S, (m)}}^{*})\to j^{-1}({\cal M}_{P_{X/S, (m)}}^{*})$.
For any section $a\in {\cal M}_{X}$, $p_{0}^{*}(a)$ and $p_{1}^{*}(a)$ have the same image in ${\cal M}_{X}$.
Thus, from the exact sequence (\ref{exact1}), there exists the unique section $\mu_{(m)}(a)\in j^{-1}(1+\bar{I})$ such that $p_{1}^{*}(a)=p_{0}^{*}(a)\cdot \lambda\bigl(\mu_{(m)}(a)\bigr)$. Log smoothness of $X\to S$ implies that, \'etale locally on $X$, there is a logarithmic system of coordinates $m_{1}, \ldots, m_{r}\in {\cal M}_{X}^{gp}$, that is, a system of sections such that the set$\left\{d\log m_{1},\ldots, d\log m_{r}\right\}$ forms a basis of the log differential module $\Omega_{X/S}^{1}$ of $X$ over $S$. 
We define the section $\eta_{i}^{(m)}$ by $\eta_{i}^{(m)}:=\mu_{(m)}(m_{i})-1$ and $\underline{\eta}^{\{\underline{k}\}(m)}$ by $\underline{\eta}^{\{\underline{k}\}(m)}=\prod_{i:=1}^{r}\eta_{i}^{\{k_{i}\}(m)}$ for each multi-index $\underline{k}\in {\mathbb N}^{r}$.
\begin{Prod}\label{Prod1}
The set $\left\{ \underline{\eta}^{\{\underline{k}\}(m)} \bigl| |\underline{k}|\leq n\right\}$ forms a local basis of ${\cal P}_{X/S, (m)}^{n}$ over ${\cal O}_{X}$.
We denote the dual basis of $\bigl\{ \underline{\eta}^{\{\underline{k}\}(m)} \bigl| |\underline{k}|\leq n\bigr\}$ by $\bigl\{ \underline{\partial}_{<\underline{k}>(m)} \bigl| |\underline{k}|\leq n\bigr\}$.
We also denote $\underline{\eta}^{\{\underline{k}\}(m)}$ $($resp. $\underline{\partial}_{<\underline{k}>(m)})$ by $\underline{\eta}^{\{\underline{k}\}}$ $($resp. $\underline{\partial}_{<\underline{k}>})$ simply,
if there will be no confusions.
\end{Prod}
\begin{proof}
See Proposition 2.2.1 of \cite{M}.
\end{proof}

\begin{pro}\label{proposition4}
Let $X\to S$ be a log smooth morphism of fine log schemes. Assume that we are given a logarithmic system of coordinates $m_{1}, \ldots, m_{r}\in {\cal M}_{X}^{gp}$.\\
{\rm(1)} ${\cal D}_{X/S}^{(m)}$ is locally generated by $\bigl\{ \underline{\partial}_{<\underline{\varepsilon}_{i}>}, \underline{\partial}_{<p\underline{\varepsilon}_{i}>}, \ldots, \underline{\partial}_{<p^{m}\underline{\varepsilon}_{i}>}\bigl| 1\leq i \leq r\bigr\}$ as an ${\cal O}_{X}$-algebra.\\
{\rm(2)} We have
	\begin{equation*}
	\underline{\partial}_{<\underline{k}'>}\cdot \underline{\partial}_{<\underline{k}''>}=\sum_{\underline{k}=\sup\{\underline{k}', \underline{k}''\}}^{\underline{k}'+\underline{k}''}\frac{\underline{k}!}{(\underline{k}'+\underline{k}''-\underline{k})!(\underline{k}-\underline{k}')!(\underline{k}-\underline{k}'')!}\frac{\underline{q}_{\underline{k}'}!\underline{q}_{\underline{k}''}!}{\underline{q}_{\underline{k}}!}\underline{\partial}_{<\underline{k}>}.
	\end{equation*}
	In particular, $\underline{\partial}_{<\underline{k}>}\cdot \underline{\partial}_{<\underline{k}'>}=\underline{\partial}_{<\underline{k}'>}\cdot \underline{\partial}_{<\underline{k}>}$ holds.\\
{\rm(3)} For any $x\in {\cal O}_{X}$, we have
	\begin{equation*}
	\underline{\partial}_{<\underline{k}>}.x=\sum_{\underline{i}\leq \underline{k}} \left\{
		\begin{array}{c}
		\underline{k}\\
		\underline{i}
		\end{array}
	\right\} \underline{\partial}_{<\underline{k}-\underline{i}>}(x)\underline{\partial}_{<\underline{i}>} \,\,\,\,\text{ in } {\cal D}_{X/S}^{(m)}.
	\end{equation*}
{\rm (4)} The natural map ${\cal D}_{X/S}^{(m)}\to {\cal D}_{X/S}^{(m')}$ sends $\underline{\partial}_{<\underline{k}>(m)}$ to $\frac{\underline{q}!}{{\underline{q'}!}}\underline{\partial}_{<\underline{k}>(m')}$, where $\underline{k}=p^{m}\underline{q}+\underline{r}$, $\underline{k'}=p^{m'}\underline{q'}+\underline{r'}$ with $0\leq \underline{r}<p^{m}$ and $0\leq \underline{r'}<p^{m'}$ .
\end{pro}
\begin{proof}
(1) See Proposition 2.3.1 of \cite{M}.
(2) See Lemme 2.3.4 of \cite{M}.
(3) See (2.5) of \cite{M}.
(4) See (2.6) of \cite{M}.
\end{proof}
We prove the following lemma needed later.
\begin{lemm}\label{lem1}
Let $X\to S$ be a log smooth morphism of fine log schemes defined over ${\mathbb Z}/p{\mathbb Z}$. For any $\underline{k}\in {\mathbb N}^{r}$, $l\in {\mathbb N}$ and $1\leq i\leq r$, we have
\begin{equation*}
\underline{\partial}_{<p^{m+1}\underline{k}>}\cdot \underline{\partial}_{<l\underline{\varepsilon}_{i}>}=\underline{\partial}_{<p^{m+1}\underline{k}+l\underline{\varepsilon}_{i}>}.
\end{equation*}
\end{lemm}
\begin{proof}
We may assume $1\leq l\leq p^{m+1}$.
When $k_{i}=0$, the assertion follows easily from Proposition \ref{proposition4} (2).
Thus we may also assume $k_{i}\geq 1$.
By Proposition \ref{proposition4} (2), we have
	\begin{equation*}
	(\clubsuit)\,\,\,\,\underline{\partial}_{<p^{m+1}\underline{k}>}\cdot \underline{\partial}_{<l\underline{\varepsilon}_{i}>}=\sum_{s=0}^{l}\frac{(p^{m+1}k_{i}+s)!}{(l-s)!s!(p^{m+1}k_{i}-l+s)!}\frac{(pk_{i})!q_{l}!}{(pk_{i}+q_{s})!}\underline{\partial}_{<p^{m
	+1}\underline{k}+s\underline{\varepsilon}_{i}>}.
	\end{equation*}
We put
	\begin{equation*}
	A:=\frac{(p^{m+1}k_{i}+s)!}{(l-s)!s!(p^{m+1}k_{i}-l+s)!}\in {\mathbb Z}\,\,\text{ and }B:=\frac{(pk_{i})!q_{l}!}{(pk_{i}+q_{s})!}\in {\mathbb Q}.
	\end{equation*}
First, we consider the case $s=l$.
Then, we have
	\begin{eqnarray*}
	A\cdot B&=&\frac{(p^{m+1}k_{i}+l)!}{l!(p^{m+1}k_{i})!}\frac{(pk_{i})!q_{l}!}{(pk_{i}+q_{l})!}\\
	&=&\prod_{j=1}^{l}(1+\frac{p^{m+1}k_{i}}{j})\cdot \prod_{j=1}^{q_{l}}(1+\frac{pk_{i}}{j})^{-1}.
	\end{eqnarray*}
Since $1+\frac{p^{m+1}k_{i}}{j}\in 1+p{\mathbb Z}_{(p)}$ if $1\leq j\leq p^{m+1}-1$,$1+\frac{p^{m+1}k_{i}}{p^{m+1}}=1+k_{i}$ if $j=p^{m+1}$, $(1+\frac{pk_{i}}{j})^{-1}\in 1+p{\mathbb Z}_{(p)}$ if $1\leq j \leq p-1$ and $(1+\frac{pk_{i}}{p})^{-1}=(1+k_{i})^{-1}$ if $j=p$, we have $A\cdot B\in 1+p{\mathbb Z}_{(p)}$ and thus $A\cdot B\equiv1$ mod $p$.
Next, we consider the case $0\leq s\leq l-1, l=p^{m+1}$.
Then, we have
	\begin{equation*}
	B=\frac{(pk_{i})!q_{l}!}{(pk_{i}+q_{s})!}=\prod_{j=1}^{q_{s}}\frac{1}{(pk_{i}+j)}\cdot p!\in p{\mathbb Z}_{(p)}.
	\end{equation*}
Hence $A\cdot B\in p{\mathbb Z}_{(p)}$. We thus have $A\cdot B\equiv 0$ mod $p$.
Finally, we consider the case $0\leq s\leq l-1, 0\leq l \leq p^{m+1}-1$. Then, we have 
	\begin{equation*}
	B=\prod_{j=1}^{q_{s}}\frac{1}{(pk_{i}+j)}\cdot q_{l}!\in {\mathbb Z}_{(p)}.
	\end{equation*}
Let $v:{\mathbb Q}^{*}\to {\mathbb Z}$ denote the normalized $p$-adic valuation. 
For any $n\in {\mathbb N}$, it is known that $(p-1)v(n!)=n-\sigma(n)$, where $\sigma(n):=\sum_{j}a_{j}$ if $n=\sum_{j}a_{j}p^{j}$.
Thus, we have
	\begin{eqnarray*}
	(p-1)v(A)&=&\sigma(l-s)+\sigma(s)+\sigma(p^{m+1}k_{i}-l+s)-\sigma(p^{m+1}k_{i}-s)\\
	&=&\sigma(l-s)+\sigma(p^{m+1}k_{i}-l+s)-\sigma(p^{m+1}k_{i})\\
	&>&0.
	\end{eqnarray*}
Hence $A\cdot B\in p{\mathbb Z}_{(p)}$. We thus have $A\cdot B\equiv 0$ mod $p$.
The assertion follows from these calculations and $(\clubsuit)$.
\end{proof}

\subsection{The $p^{m+1}$-curvature map}
Throughout this subsection, all the log schemes are assumed to be defined over $\mathbb{Z}\slash p\mathbb{Z}$. 
Let us introduce some notations. For a log scheme $X$, $F_{X}$ denotes the $(m+1)$-st iterate of its absolute Frobenius endomorphism. 
For a morphism $X\to S$ of fine log schemes, we consider the following commutative diagram:
\[\xymatrix{
X \ar@/_/[dddrr] \ar@{>}[dr]  \ar@/^/[ddrrr]^{F_{X}} \\
& X' \ar[dr] \\
& & {X''} \ar[d] \ar[r] & X \ar[d] \\
& & S \ar[r]^{F_{S}} & S, }
\] 
where $X''$ is the fiber product in the category of fine log schemes, and the morphism $X\to X'$ (denoted by $F_{X/S}$ and called the ($(m+1)$-st) relative Frobenius morphism) is uniquely determined by the requirement that the morphism $F_{X\slash S}$ is purely inseparable and $X'\to X''$ is log \'etale (see Proposition 4.10 of \cite{K}).
We denote the composition $X'\to X''\to X$ by $\pi_{X\slash S}$ or simply by $\pi$.
We also denote $X'$ by $X^{(m+1)}$, if there is a risk of confusion. \\

First we prove the log level $m$ version of Mochizuki's theorem which is used to construct the $p^{m+1}$-curvature map (see also Proposition 3.2 of \cite{GLQ} and Proposition A.7 of \cite{S}).

\begin{theo}\label{Theorem4}
Let $X\to S$ be a log smooth morphism of fine log schemes. Let $P_{X/S, (m)}$ (resp. $Y$) be the log $m$-PD envelope (resp. the log formal neighborhood) of the diagonal immersion $X\to X\times_{S} X$ and $\bar{I}$ (resp. $I$) its defining ideal. 
Let ${\cal P}_{X/S, (m)}$ denote the structure sheaf of $P_{X/S, (m)}$.
Then there is an isomorphism of ${\cal O}_{X}$-modules
\begin{equation*}
\alpha: F_{X\slash S}^{*}\Omega_{X'\slash S}^{1}\to \bar{I}/ \left(\bar{I}^{\{p^{m+1}+1\}}+I{\cal P}_{X/S, (m)}\right) 
\end{equation*}
such that, for any $\xi \in I$ with image $\omega \in I/ I^{2}\cong \Omega_{X\slash S}^{1}$,
\begin{equation}
\alpha(1\otimes \pi^{*}\omega)=\xi^{\{p^{m+1}\}} \label{mochizuki}.
\end{equation}
\end{theo}

\begin{proof}
First we show that the map $\alpha': I\to \bar{I}/ I{\cal P}_{X/S, (m)}$ defined by
\begin{equation*}
\begin{array}{ccc}
 I                     & \longrightarrow &  \bar{I}\slash I{\cal P}_{X/S, (m)}                     \\[-4pt]
 \\[-4pt]
              \xi                     & \longmapsto     & \xi^{\{p^{m+1}\}}
\end{array}
\end{equation*}
is $F_{X}^{*}$-linear and zero on $I^{2}$. If $\xi$ and $\tau$ are local sections of $I$,  by Proposition \ref{Proposition9} (3) we have
\begin{equation*}
(\xi+\tau)^{\{p^{m+1}\}}=\xi^{\{p^{m+1}\}}+\tau^{\{p^{m+1}\}}+\sum_{\begin{subarray}{c}
i+j=p^{m+1}\\
i, j>0
\end{subarray}}\left\langle
\begin{array}{c}
p^{m+1}\\
i
\end{array}
\right\rangle \xi^{\{i\}} \tau^{\{j\}}.
\end{equation*}
Since $0<i, j<p^{m+1}$, we have $q_{i}, q_{j}<p$, where $q_{i}, q_{j}$ are as in Subsection 1.3.
Therefore $q_{i}!$ and $q_{j}!$ are invertible. From Proposition \ref{Proposition9} (4) we have $\xi^{\{i\}} \tau^{\{j\}}=(q_{i}!q_{j}!)^{-1}\xi^{i}\tau^{j}\in I{\cal P}_{X/S, (m)}$.
It follows that the last term in the sum is in $I{\cal P}_{X/S, (m)} $ and we see the additivity of $\alpha'$.
Similarly, the fact that $\alpha'$ is $F_{X}^{*}$-linear and zero on $I^{2}$ follows from Proposition \ref{Proposition9} (2), (5). We thus obtain the ${\cal O}_{X}$-linear map 
\begin{equation*}
\alpha: F_{X\slash S}^{*}\Omega_{X'\slash S}^{1}\cong F_{X}^{*}\Omega_{X\slash S}^{1}\to \bar{I}\slash \left(\bar{I}^{\{p^{m+1}+1\}}+I{\cal P}_{X/S, (m)}\right)
\end{equation*}
which satisfies (\ref{mochizuki}).
Let us show that $\alpha$ is an isomorphism.
Since the assertion is \'etale local on $X$, we may assume that we have a logarithmic system of coordinates $m_{1}, \ldots , m_{r}\in {\cal M}_{X}^{gp}$.
Then the left hand side is isomorphic to $\bigoplus^{r}_{i=1}{\cal O}_{X}\left(1\otimes \pi^{*}d\log m_{i}\right)$.
On the other hand, by Proposition-Definition \ref{Prod1}, $\bar{I}\slash \bar{I}^{\{p^{m+1}+1\}}$ is freely generated by $\left\{\underline{\eta}^{\{\underline{i}\}}| 1\leq |\underline{i}|\leq p^{m+1} \right\}$ as an ${\cal O}_{X}$-module and 
the image of $I{\cal P}_{X/S, (m)}$ under the map $I{\cal P}_{X/S, (m)}\to \bar{I}\slash \bar{I}^{\{p^{m+1}+1\}}$ is generated by $\left\{\eta_{j}\underline{\eta}^{\{\underline{i}\}}| 0\leq |\underline{i}|\leq p^{m+1}-1, 1\leq j\leq r \right\}$ as an ${\cal O}_{X}$-module.
Actually $\bar{I}\slash (\bar{I}^{\{p^{m+1}+1\}}+I{\cal P}_{X/S, (m)})$ is freely generated by $\left\{{\eta}_{i}^{\{p^{m+1}\}}\bigl| 1\leq i\leq r\right\}$ as an ${\cal O}_{X}$-module.
So the right hand side is isomorphic to $\bigoplus^{r}_{i=1}{\cal O}_{X}\eta^{\{p^{m+1}\}}_{i}$ and,  by construction, 
$\alpha$ sends $1\otimes \pi^{*}d\log m_{i}$ to $\eta^{\{p^{m+1}\}}_{i}$.
This completes the proof.
\end{proof}
Let $a$ denote the map defined by the composition of maps
\begin{equation*}
{\cal P}_{X/S, (m)}^{n}\to {\cal O}_{X}\to {\cal P}_{X/S, (m)}^{n},
\end{equation*}
where the first map is the natural projection and the second one is the structural morphism $p^{n*}_{0}$.
Now, we are ready to define the $p^{m+1}$-curvature map.
\begin{defi}\label{defi1}
Let $X\to S$ be a log smooth morphism of fine log schemes. Let ${\cal T}_{X'\slash S}:={\cal H}om_{{\cal O}_{X'}}(\Omega^{1}_{X'/S}, {\cal O}_{X'})$ denote the log tangent bundle on $X'$.
We define the map $\beta :{\cal T}_{X'/S}\to F_{X/S*}{\cal D}_{X/S}^{(m)}$ by sending $D\in {\cal T}_{X'\slash S}$ to the composition of maps
	\begin{equation*}
	{\cal P}_{X/S, (m)}^{p^{m+1}}\xrightarrow{y\mapsto y-a(y)} \bar{I}/\bar{I}^{\{p^{m+1}\}}\to \bar{I}/ \left(\bar{I}^{\{p^{m+1}+1\}}+I{\cal P}_{X, m}\right)\xleftarrow{\cong} F^{*}_{X/S}\Omega^{1}_{X'/S}\xrightarrow{F^{*}_{X/S}D} {\cal O}_{X},
	\end{equation*}
where the second map is the natural projection and the third one is the isomorphism in Theorem \rm{\ref{Theorem4}}. We call it the $p^{m+1}$-curvature map.
\end{defi}
The local description of the $p^{m+1}$-curvature map is the following.
\begin{pro}\label{theorem1}
Let $X\to S$ be a log smooth morphism of fine log schemes. Assume that we are given a logarithmic system of coordinates $m_{1},\ldots ,m_{r}\in {\cal M}^{gp}_{X}$ . Let $\left\{\xi_{i}^{'}\bigl| 1\leq i\leq r\right\}$ denote the dual basis of  $\left\{\pi^{*}d\log m_{i}\bigl| 1\leq i\leq r\right\}$.
Then $\beta$ sends $\xi^{'}_{i}$ to $\underline{\partial}_{<p^{m+1}\underline{\varepsilon}_{i}>}$.
\end{pro}
\begin{proof}
We calculate that $\beta(\xi'_{i})$ sends $\underline{\eta}^{\{\underline{k}\}}$ to $1$ if $\underline{k}=p^{m+1}\underline{\varepsilon}_{i}$ and $0$ otherwise by construction of $\beta$, thereby completing the proof.
\end{proof}
\begin{rem}
When $m$ is equal to $0$, our $p^{m+1}$-curvature map is the usual $p$-curvature map (\cite{OV} Proposition 1.7).
If the log structure of $X$ is trivial, then our $p^{m+1}$-curvature map coincides with the $p^{m}$-curvature map studied in \cite{GLQ} section 3. 
\end{rem}

\section{Azumaya algebra property}
The goal of this section is the Azumaya algebra property of the indexed version of the sheaf of log differential operators $\tilde{{\cal D}}_{X/S}^{(m)}={\cal A}_{X}^{gp}\otimes_{{\cal O}_{X}}{\cal D}_{X/S}^{(m)}$ defined by Montagnon.
We also study the Azumaya nature of $\tilde{{\cal D}}_{X^{(m)}/S}^{(0)}={\cal B}_{X/S}^{(m)}\otimes_{{\cal O}_{X^{(m)}}}{\cal D}_{X^{(m)}/S}^{(0)}$, which is also introduced by Montagnon.
At first we give a review of the canonical indexed algebra ${\cal A}_{X}^{gp}$ associated to the log structure of $X$ and Montagnon's $\tilde{{\cal D}}_{X/S}^{(m)}$ and $\tilde{{\cal D}}_{X^{(m)}/S}^{(0)}$.

\subsection{Indexed algebra associated to a log structure}
\subsubsection{The ${\cal I}_{X}^{gp}$-indexed algebra $\tilde{{\cal D}}_{X/S}^{(m)}$}
First we recall the definition of ${\cal A}_{X}^{gp}$.
Let $X$ be a fine log scheme. We consider the extension of sheaves of abelian groups
\begin{equation*}
0\longrightarrow {\cal O}_{X}^{*}\longrightarrow {\cal M}_{X}^{gp}\xrightarrow{\delta} {\cal I}_{X}^{gp}\longrightarrow 0.
\end{equation*}
Here $ {\cal I}_{X}^{gp}$ is the quotient sheaf ${\cal M}_{X}^{gp}/{\cal O}_{X}^{*}$.
We define ${\cal A}_{X}^{gp}$ as the contracted product ${\cal M}_{X}^{gp}\wedge_{{\cal O}^{*}_{X}} {\cal O}_{X}$ which is the quotient of  ${\cal M}_{X}^{gp}\times{\cal O}_{X}$ by the equivalence relation $(ax, y)\sim (x, ay)$ where $a, x, y$ are local sections of ${\cal O}_{X}^{*}, {\cal M}_{X}^{gp}$ and ${\cal O}_{X}$ respectively.
The projection ${\cal M}_{X}^{gp}\times {\cal O}_{X}\to {\cal M}_{X}^{gp}\to {\cal I}_{X}^{gp}$ induces a map ${\cal A}_{X}^{gp}\to {\cal I}_{X}^{gp}$ which makes ${\cal A}_{X}^{gp}$ an ${\cal I}_{X}^{gp}$-indexed ${\cal O}_{X}$-module.
For a local section $i$ of ${\cal I}_{X}^{gp}$, the fiber ${\cal M}_{X, i}^{gp}$ of ${\cal M}_{X}^{gp}\to {\cal I}_{X}^{gp}$ at $i$ is an ${\cal O}^{*}_{X}$-torsor.
This implies that the fiber ${\cal A}_{X, i}^{gp}$ of ${\cal A}_{X}^{gp}\to {\cal I}_{X}^{gp}$ at $i$ is an invertible ${\cal O}_{X}$-module.
${\cal A}_{X}^{gp}$ has a multiplication map induced by the addition map ${\cal M}_{X}^{gp}\times {\cal M}_{X}^{gp}\to {\cal M}_{X}^{gp}$ over ${\cal I}^{gp}_{X}\times{\cal I}_{X}^{gp}\to {\cal I}^{gp}_{X}$. Hence ${\cal A}_{X}^{gp}$ forms an ${\cal I}_{X}^{gp}$-indexed ${\cal O}_{X}$-algebra. ${\cal A}_{X}^{gp}$ is called the ${\cal I}_{X}^{gp}$-indexed algebra associated to the log structure.\\

Next let us recall the definition of the section $e_{s}$ of ${\cal A}_{X}^{gp}$ associated to a section $s$ of ${\cal M}_{X}^{gp}$. For each \'etale open $U$ of $X$ and a section $s\in {\cal M}_{X}^{gp}(U)$, $s$ trivializes the ${\cal O}_{X}^{*}$-torsor ${\cal M}_{X, \delta(s)}^{gp}$. Thus it gives a basis $e_{s}:=\overline{(s, 1)}$ of the invertible ${\cal O}_{U}$-module 
${\cal A}_{X, \delta(s)}^{gp}$. Then
\begin{equation*}
e_{0}=1,\,\, e_{s}e_{t}=e_{s+t},\,\, ae_{s}=e_{as}
\end{equation*}
for $s, t\in {\cal M}_{X}^{gp}(U)$ and $a\in {\cal O}_{X}^{*}(U)$.\\

The construction of  ${\cal A}_{X}^{gp}$ is functorial in the following sense. For a morphism $f: X\to Y$ of fine log schemes, we have a commutative diagram

\begin{equation*}
\begin{CD}
0 @>>> {\cal O}_{X}^{*} @>>> (f^{*}{\cal M}_{Y})^{gp} @>>> f^{-1}{\cal I}_{Y}^{gp} @>>> 0 \\
@. @VVV @VVV @VVV \\
0 @>>> {\cal O}_{X}^{*} @>>> {\cal M}_{X}^{gp} @>>> {\cal I}_{X}^{gp} @>>> 0,\\
\end{CD}
\end{equation*}
where vertical arrows are isomorphisms if $f$ is strict.
This induces a commutative diagram

\begin{equation*}
\begin{CD}
f^{*}{\cal A}_{Y}^{gp} @> \text{${\cal A}_{f}^{gp}$} >>{\cal A}_{X}^{gp}\\
@VVV @VVV\\
f^{-1}{\cal I}_{Y}^{gp} @>>> {\cal I}_{X}^{gp},\\
\end{CD}
\end{equation*}
where horizontal arrows are isomorphisms if $f$ is strict.
So we get the following proposition.
\begin{pro}
If $f:X\to Y$ is a strict morphism of fine log schemes, then ${\cal A}_{f}^{gp}: f^{*}{\cal A}_{Y}^{gp}\to {\cal A}_{X}^{gp}$ is an isomorphism of ${\cal I}_{X}^{gp}$-indexed algebras.
\end{pro}

Next we recall the definition of ${\cal I}_{X}^{gp}$-indexed left ${\cal D}_{X/S}^{(m)}$-module structure on ${\cal A}_{X}^{gp}$.  Let $\ast$ be a trivial sheaf of abelian groups on $X$. We naturally regard ${\cal D}_{X/S}^{(m)}$ as $\ast$-indexed ${\cal O}_{X}$-algebra and ${\cal I}_{X}^{gp}$ as sheaf of $\ast$-sets.
Since the natural projections $p^{n}_{i}: P^{n}_{X, (m)}\to X$ with $i=0, 1$ are strict by construction, we have the isomorphism of ${\cal I}_{X}^{gp}$-indexed ${\cal O}_{X}$-algebras 
\begin{equation*}
\varepsilon_{n}:p_{1}^{n*}{\cal A}_{X}^{gp}\xrightarrow{\cong} {\cal A}_{P_{X/S, (m)}^{n}}\xleftarrow{\cong} p_{0}^{n*}{\cal A}_{X}^{gp}.
\end{equation*}
These isomorphisms are compatible with respect to $n$ by construction and satisfy the cocycle condition. 
Hence we can define an ${\cal I}_{X}^{gp}$-indexed left ${\cal D}_{X/S}^{(m)}$-module structure on ${\cal A}_{X}^{gp}$ by
\begin{eqnarray*}
{\cal D}_{X/S}^{(m)}\times{\cal A}_{X}^{gp}\hookrightarrow {\cal D}_{X/S}^{(m)}\times p_{1}^{*}{\cal A}_{X}^{gp}\xrightarrow{\rm{id}\times\varepsilon} {\cal D}_{X/S}^{(m)}\times p_{0}^{*}{\cal A}_{X}^{gp}\,\,\,\,\,\,\,\,\,\,\,\,\,\, \\
 \hookrightarrow {\cal H}om_{{\cal O}_{X}}\left({\cal P}_{X/S, (m)}, {\cal O}_{X}\right)\times {\cal A}_{X}^{gp}\otimes_{{\cal O}_{X}} {\cal P}_{X/S, (m)}\to {\cal A}_{X}^{gp} .
\end{eqnarray*}

By calculation with the section $e_{s}$ $\left(s\in {\cal I}^{gp}_{X}\right)$, one can see that the action of ${\cal D}_{X/S}^{(m)}$ on ${\cal A}_{X}^{gp}$ satisfies the Leibnitz type formula (for more details, see Subsection 4.1 of \cite{M}).
Therefore we have the following nontrivial ring structure on $\tilde {{\cal D}}_{X/S}^{(m)}$ which is a central object in this article.
\begin{Prod}\label{proposition10}
Let $X\to S$ be a log smooth morphism of fine log schemes.
Let $\tilde {{\cal D}}_{X/S}^{(m)}$ denote the ${\cal I}_{X}^{gp}$-indexed ${\cal O}_{X}$-module ${\cal A}_{X}^{gp}\otimes_{{\cal O}_{X}}{\cal D}_{X/S}^{(m)}$.
Then there exists a unique ${\cal I}_{X}^{gp}$-indexed ${\cal O}_{X}$-algebra structure on $\tilde {{\cal D}}_{X/S}^{(m)}$ such that the maps ${\cal A}_{X}^{gp}\to \tilde{{\cal D}}_{X/S}^{(m)};\,\,a\mapsto a\otimes1$ and ${\cal D}_{X/S}^{(m)}\to \tilde{{\cal D}}_{X/S}^{(m)};\,\,P\mapsto 1\otimes P$ are homomorphisms and that for any $a\in {\cal A}_{X}^{gp}$, $P\in {\cal D}_{X/S}^{(m)}$ and $\underbar{k}\in {\mathbb{N}}^{r}$, we have the relations $a\otimes P=(a\otimes 1)(1\otimes P)$ and 
\begin{equation*}
(1\otimes \underline{\partial}_{<\underline{k}>})(a\otimes 1)=\sum_{\underline{i}\leq \underline{k}}\left\{
\begin{array}{c}
\underline{k}\\
\underline{i}
\end{array}
\right\}
(\underline{\partial}_{<\underline{k}-\underline{i}>}.a)\otimes \underline{\partial}_{<\underline{i}>}.
\end{equation*}
Here \{$\underline{\partial}_{<\underline{k}>}$\} is as in Proposition-Definition \ref{Prod1}.
\end{Prod}
\begin{proof}
See Subsection 4.1 of \cite{M}.
\end{proof}
\begin{rem}\label{remark6}
In \cite{M}, Montagnon defines $\tilde {{\cal D}}_{X/S}^{(m)}$ by ${\cal A}_{X}^{gp}\otimes_{{\cal O}_{X, {\cal I}}}{\cal D}_{X/S, {\cal I}}^{(m)}$,
where ${\cal O}_{X, {\cal I}}$ $\left(\text{resp. } {\cal D}_{X/S, {\cal I}}^{(m)}\right)$ denotes an ${\cal I}$-indexed ${\cal O}_{X}$-module ${\cal O}_{X}\times {\cal I}$ $\left(\text{resp. } {\cal D}_{X/S}^{(m)}\times {\cal I}\right)$ with the second projection.
This definition is a mistake.
Actually ${\cal D}_{X/S}^{(m)}$ must be trivially indexed so that the ${\cal D}_{X/S}^{(m)}$-action on ${\cal A}^{gp}_{X}$ satisfies the Leibnitz type formula.
\end{rem}
Finally we recall the following formula needed later.
Let $m_{1}, \ldots , m_{r}\in {\cal M}_{X}^{gp}$ be a logarithmic system of coordinates. For $i\in \{1, 2, \ldots, r\}$ and a multi-index $\underline{i}$, we put $\theta_{i}:=e_{m_{i}}\text{ and }\underline{\theta}^{\underline{i}}:=\prod_{j}\theta_{j}^{i_{j}}$.

\begin{pro}\label{proposition1}
For $\underbar{k}\in {\mathbb{N}}^{r}$ and $1\leq i\leq r$,
\begin{equation*}
\underline{\partial}_{<\underline{k}>}.\underline{\theta}^{\underline{j}}=\underline{q}_{\underline{k}}! \binom{\underline{j}}{\underline{k}}\underline{\theta}^{\underline{j}}.
\end{equation*}
\end{pro}
\begin{proof}
See Lemme 4.2.3 of \cite{M}.
\end{proof}

\subsubsection{The ${\cal I}_{X}^{gp}$-indexed algebra $\tilde{{\cal D}}_{X^{(m)}/S}^{(0)}$}
We start with a general theory of log ${\cal D}$-modules of higher level.
For details in more general settings, see Chapitre 3 of \cite{M}.
Let $X\to S$ be a log smooth morphism of fine log schemes defined over $\mathbb{Z}\slash p\mathbb{Z}$.
We consider the following commutative diagram:

\[\xymatrix{
X \ar[d] \ar[r] & P_{X/S, (0)} \ar@<1ex>[r]^{q_{0}}\ar[r]_{q_{1}} & X\ar[d]\\
X^{(m)} \ar[r] & P_{X^{(m)}/S, (0)} \ar@<1ex>[r]^{p_{0}} \ar[r]_{p_{1}} & X^{(m)}. 
}
\]
Here the vertical two arrows are the $m$-th relative Frobenius defined in Subsection 3.2 and 
$q_{0}$ and $q_{1}$ (resp. $p_{0}$ and $p_{1}$) are the first and second projections.
Then, by universal property of log $0$-PD envelopes,
there exists an unique morphism of $0$-PD fine log schemes $F_{\triangle}: P_{X/S,(0)}\to P_{X^{(m)}/S,(0)}$ which fits into the above diagram.

	\begin{pro}\label{proposition11}
	Let $X\to S$ be a log smooth morphism of fine log schemes. 	
	
	{\rm (1)} There exists $\Phi: P_{X/S,(m)}\to P_{X^{(m)}/S,(0)}$ such that $F_{\triangle}$ uniquely factors as $P_{X/S, (0)}\to P_{X/S, (m)}\xrightarrow{\Phi}P_{X^{(m)}/S}$,
	where the first map is the natural homomorphism.
	
	{\rm (2)} Assume that we are given a logarithmic system of coordinates $\{m_{i}\}$ of $X\to S$.
	Let $\left\{\underline{\eta}^{\{\underline{k}\}(m)}\right\}$ (resp. $\left\{\underline{\eta'}^{\{\underline{k}\}(0)}\right\}$) denote the basis of ${\cal P}_{X/S, (m)}$ (resp. ${\cal P}_{X^{(m)}/S, (0)}$) associated to the basis $\{d\log m_{i}\}$ (resp. $\{d\log		\pi^{*}m_{i}\}$),
	where $\pi$ denotes the natural projection $X^{(m)}\to X$ explained in the beginning of Subsection 3.2.
	Then $\Phi^{*}: {\cal P}_{X^{(m)}/S, (0)} \to {\cal P}_{X/S, (m)}$ sends $\underline{\eta'}^{\{\underline{k}\}(0)}$ to $\underline{\eta}^{\{p^{m}\underline{k}\}(m)}$.
	\end{pro}
	\begin{proof}
	(1) see Proposition 3.3.1 of \cite{M}. (2) See (i) of Proposition 3.4.1 of \cite{M}.
	\end{proof}
Let $\cal E$ be a left ${\cal D}_{X^{(m)}/S}^{(0)}$-module and $\{{\varepsilon}_{n}\}$ the log $0$-stratification on $\cal E$ via the equivalence in Remark \ref{Remark8}.
By endowing $F^{*}{\cal E}$ with a left ${\cal D}_{X/S}^{(m)}$-module structure by pulling back $\{{\varepsilon}_{n}\}$ via $\Phi$, we have a functor
\begin{equation}
F^{*}:
	\left(
		\begin{array}{c}
		\text{The category of} \\
		\text{left ${\cal D}_{X^{(m)}/S}^{(0)}$-modules on $X^{(m)}$}
		\end{array}
	\right)
	\to 
	\left(
		\begin{array}{c}
		\text{The category of} \\
		\text{left ${\cal D}_{X/S}^{(m)}$-modules on $X$}
		\end{array}
	\right).
\end{equation}
\begin{rem}\label{Remark20}
Let $\cal E$ be a left ${\cal D}_{X^{(m)}/S}^{(0)}$-module.
Let $\left\{\underline{\eta}^{\{\underline{k}\}(m)}\right\}$ and $\left\{\underline{\eta'}^{\{\underline{k}\}(0)}\right\}$ be as in Proposition \ref{proposition11} (2).
Let $\left\{\underline{\partial}_{<\underline{k}>}\right\}$ (resp. $\left\{\underline{\partial'}_{<\underline{k}>}\right\}$) denote the dual of $\left\{\underline{\eta}^{\{\underline{k}\}(0)}\right\}$ (resp. $\left\{\underline{\eta}'^{\{\underline{k}\}(0)}\right\}$).
Then the ${\cal D}_{X/S}^{(m)}$-action on $F^{*}{\cal E}$ is characterized by the following formula:
\begin{equation}\label{formula3}
\underline{\partial}_{<\underline{k}>(m)}.(1\otimes f)=\biggl\{
	\begin{array}{c}
	1\otimes \underline{\partial}'_{<p^{-m}\underline{k}>(0)}.f \text{\,\,\,\,\,\,\,\,\,\,\,\, if $p^{m}$ devides $\underline{k}$}\\
	0 \text{ \,\,\,\,\,\,\,\,\,\,\,\,\,\,\,\,\,\,\,\,\,\,\,\,\,\,\,\,\,\,\,\,\,\,\,\,\, otherwise}.	
	\end{array}
\end{equation}
For the proof of this formula, see Proposition 3.4.1 of \cite{M}.
\end{rem}
In particular, we can consider $F^{*}{\cal D}_{X^{(m)}/S}^{(0)}$ as a left ${\cal D}^{(m)}_{X/S}$-module.
Now, we define the subalgebra ${\cal B}_{X/S}^{(m)}$ of ${\cal A}_{X}^{gp}$.
\begin{Prod}
Let ${\cal B}_{X/S}^{(m)}$ be the ${\cal I}_{X}^{gp}\cong {\cal H}om_{*}\left(*, {\cal I}_{X}^{gp}\right)$-indexed sheaf of abelian groups ${\cal B}_{X/S}^{(m)}= {\cal H}om_{{\cal D}^{(m)}_{X/S}}\left(F^{*}{\cal D}_{X^{(m)}/S}^{(0)}, {\cal A}_{X}^{gp}\right)$.
We give ${\cal O}_{X^{(m)}}$-action on ${\cal B}_{X/S}^{(m)}$ by right multiplication of ${\cal O}_{X^{(m)}}$ on $F^{*}{\cal D}_{X^{(m)}/S}^{(0)}$ and define the morphism $\psi: {\cal B}_{X}^{(m)}\to {\cal A}_{X}^{gp}$ by $g\mapsto g(1\otimes 1)$. Then $\psi$ is injective and ${\cal B}_{X/S}^{(m)}$ forms an ${\cal I}_{X}^{gp}$-indexed sub ${\cal O}_{X^{(m)}}$-algebra of ${\cal A}_{X}^{gp}$ induced from the multiplication on ${\cal A}_{X}^{gp}$.
\end{Prod}
\begin{proof}
See Lemme 4.2.1 of \cite{M}.
\end{proof}
\begin{rem}
In \cite{M}, Montagnon defines ${\cal B}_{X/S}^{(m)}$ by ${\cal B}_{X/S}^{(m)}= {\cal H}om_{{\cal D}^{(m)}_{X/S, {\cal I}}}\left(F^{*}{\cal D}_{X^{(m)}/S, {\cal I}}^{(0)}, {\cal A}_{X}^{gp}\right)$ but this is a mistake.
See also Remark \ref{remark6}.
\end{rem}
\begin{Prod}
Let $X\to S$ be a log smooth morphism of fine log schemes.
Let $\tilde {{\cal D}}_{X^{(m)}/S}^{(0)}$ denote the ${\cal I}_{X}^{gp}$-indexed ${\cal O}_{X^{(m)}}$-module ${\cal B}_{X/S}^{(m)}\otimes_{{\cal O}_{X^{(m)}}}{\cal D}_{X^{(m)}/S}^{(0)}$.
Then there exists a unique ${\cal I}_{X}^{gp}$-indexed ${\cal O}_{X^{(m)}}$-algebra structure on $\tilde {{\cal D}}_{X^{(m)}/S}^{(0)}$ such that the maps ${\cal B}_{X/S}^{(m)}
\to \tilde{{\cal D}}_{X^{(m)}/S}^{(0)};\,\,a\mapsto a\otimes1$ and ${\cal D}_{X^{m}/S}^{(0)}\to \tilde{{\cal D}}_{X^{(m)}/S}^{(0)};\,\,P\mapsto 1\otimes P$ are homomorphisms and that for any $a\in {\cal B}_{X/S}^{(m)}$, $P\in {\cal D}_{X^{(m)}/S}^{(0)}$ and $\underbar{k}\in {\mathbb{N}}^{r}$, we have the relations $a\otimes P=(a\otimes 1)(1\otimes P)$ and 
\begin{equation*}
(1\otimes \underline{\partial'}_{<\underline{k}>})(a\otimes 1)=\sum_{\underline{i}\leq \underline{k}}
\binom{\underline{k}}{\underline{i}}
(\underline{\partial'}_{<\underline{k}-\underline{i}>}.a)\otimes \underline{\partial'}_{<\underline{i}>}.
\end{equation*}
Here the notation $\left\{\underline{\partial}'_{<\underline{k}>}\right\}$ is as in Remark \ref{Remark20}.
\end{Prod}
\begin{proof}
See Subsection 4.2.1 of \cite{M}.
\end{proof}
\begin{rem}
In general, ${\cal B}_{X/S}^{(m)}$ does not coincide with ${\cal A}^{gp}_{X^{(m)}}$.
For counter-example, see 1.8 of \cite{L}.
\end{rem}\label{Remark01}
Finally, we give the local description of ${\cal B}^{(m)}_{X/S}$.
Let $\left\{ \underline{\partial}_{<\underline{k}>(m)}\right\}$ be as in Remark \ref{Remark20}.
	\begin{pro}\label{proposition3}
	Locally, ${\cal B}_{X/S}^{(m)}$ can be written as follows:
		\begin{equation*}
		\left\{\,a\in {\cal A}_{X}^{gp} \,\bigl|\, \underline{\partial}_{<p^{s}\underline{\varepsilon}_{i}>(m)}a=0 \text{ for all }0\leq s\leq m-1, 1\leq i\leq r\,\right\}.
		\end{equation*}
	\end{pro}
\begin{proof}
See 4.2.1 of \cite{M}.
\end{proof}
	\begin{rem}
	By virtue of Proposition \ref{proposition3} combined with Proposition \ref{proposition4} (4), we can locally write
		\begin{equation*}
		{\cal B}_{X/S}^{(m+1)}=\left\{\,a\in {\cal A}_{X}^{gp} \,\bigl|\, \underline{\partial}_{<p^{s}\underline{\varepsilon}_{i}>(m)}a=0 \text{ for all }0\leq s\leq m, 1\leq i\leq r\,\right\}.
		\end{equation*}	
	\end{rem}
	\begin{pro}\label{proposition2}
	${\cal A}_{X}^{gp}$ is locally free as an ${\cal I}_{X}^{gp}$-indexed ${\cal B}_{X/S}^{(m+1)}$-module with local basis $\{\,\underline{\theta}^{\underline{j}}\,|\,\underline{j}\in [0, p^{m+1}-1]^{r}\,\}$.
	\end{pro}
\begin{proof}
See Corollaire 4.2.1 of \cite{M}.
\end{proof}
	\begin{pro}\label{proposition14}
	${\cal B}_{X/S}^{(m)}$ is locally free as an ${\cal I}_{X}^{gp}$-indexed ${\cal B}_{X/S}^{(m+1)}$-module with local basis $\left\{\,\underline{\theta}^{p^{m}\underline{j}}\,|\,\underline{j}\in \left[0, p-1\right]^{r}\right\}$.	

	\end{pro}
	\begin{proof}
	In this proof we put $M:=\left[0, p^{m+1}-1\right]^{r}$.
	Note that $\underline{\theta}^{p^{m}\underline{j}}\in {\cal B}_{X/S}^{(m)}$ by Proposition \ref{proposition1}.
	Since the assertion is \'etale local on $X$, we may work in a local situation.
	Let $f$ be a local section of ${\cal B}_{X/S}^{(m)}$.
	By Proposition \ref{proposition2}, we can wright $f=\sum_{\underline{j}\in M}a_{\underline{j}}\underline{\theta}^{\underline{j}}$ with $a_{\underline{j}}\in {\cal B}_{X/S}^{(m+1)}$.
	Then we calculate
	\begin{eqnarray*}
	\underline{\partial}_{<p^{s}\underline{\varepsilon}_{i}>}. f=\sum_{\underline{j}\in M}\left(\sum_{\underline{k}\leq {p^{s}\underline{\varepsilon}_{i}}}	\left\{
		\begin{array}{c}
		{p^{s}\underline{\varepsilon}_{i}}\\
		\underline{k}
		\end{array}
		\right\}
		\underline{\partial}_{<{p^{s}\underline{\varepsilon}_{i}}-\underline{k}>}a_{\underline{j}}\underline{\partial}_{<\underline{k}>}\underline{\theta}^{\underline{j}} \right).
	\end{eqnarray*}
	From the proof of Proposition 2.3.1 in \cite{M}, $\underline{\partial}_{<{p^{s}\underline{\varepsilon}_{i}}-\underline{k}>}$ is belongs to the ${\cal O}_{X}$-subalgebra of ${\cal D}^{(m)}_{X/S}$ generated by \{$\underline{\partial}_{<{p^{s}\underline{\varepsilon}_{i}}>}| 0\leq i\leq m-1$\} .
	So $\underline{\partial}_{<{p^{s}\underline{\varepsilon}_{i}}-\underline{k}>}a_{\underline{j}}=0$ if $\underline{k}\neq p^{s}\underline{\varepsilon}_{i}$.
	Hence, by Proposition \ref{proposition1}, we have 
	\begin{eqnarray*}
	\underline{\partial}_{<p^{s}\underline{\varepsilon}_{i}>}. f=\sum_{\underline{j}\in M}a_{\underline{j}}\binom{j_{i}}{p^{s}}	
	\underline{\theta}^{\underline{j}}=0 \text{ \,\,\,\,\,for any $1\leq s\leq m-1$, and $1\leq i\leq r$.}
	\end{eqnarray*}	
	Let us calculate the $p$-adic valuation of $\binom{j_{i}}{p^{s}}$.
	Let $v$ be the normalized $p$-adic valuation.
	In the case that $p^{m}$ divides $j_{i}$, we easily calculate $v(\binom{j_{i}}{p^{s}})>0$ for any $1\leq s\leq m-1$.
	In the case that $p^{m}$ does not divide $j_{i}$, if we write $j_{i}=\sum b_{l}p^{l}$, then there exists $1\leq s\leq m-1$ such that $b_{s}>0$.
	We obtain $v(\binom{j_{i}}{p^{s}})=0$ for this $s$.
	Since $\{\,\underline{\theta}^{\underline{j}}\,|\,\underline{j}\in M\,\}$ is linearly independent, we conclude that $a_{\underline{j}}\neq 0$ if and only if $p^{m}$ divides $\underline{j}$.
	We finish the proof.
	\end{proof}

\subsection{The Azumaya algebra property}
We first prove the Azumaya algebra property of $\tilde {{\cal D}}_{X/S}^{(m)}$ over its center. To do this, first we calculate the center of $\tilde {{\cal D}}_{X/S}^{(m)}$. Throughout this subsection we assume that all log schemes are defined over ${\mathbb{Z}}/p\mathbb{Z}$.\\

Let $X\to S$ be a log smooth morphism of fine log schemes and $F_{X/S}:X\to X'$ the relative Frobenius morphism of $X$ defined in Subsection 3.2.
We have constructed the $p^{m+1}$-curvature map $\beta : {\cal T}_{X'\slash S}\to F_{X\slash S*}{\cal D}^{(m)}_{X}$ (Definition \ref{defi1}). 
\begin{lemm}\label{Lem4}
The image of ${\cal T}_{X'\slash S}$ under the $p^{m+1}$-curvature map $\beta$ is contained in the center of ${\cal D}_{X/S}^{(m)}$.
\end{lemm}
\begin{proof}
Since the assertion is \'etale local on $X$, we may work with the local description of $\beta $ (Proposition \ref{theorem1}).
Since ${\cal D}_{X/S}^{(m)}$ is generated by $\bigl\{ \underline{\partial}_{<p^{s}\underline{\varepsilon}_{i}>}\bigl| 1\leq i\leq r, 1\leq s\leq m \bigr\}$ as ${\cal O}_{X}$-algebra, we need to calculate $\underline{\partial}_{<p^{m+1}\underline{\varepsilon}_{j}>}\cdot a$ and $\underline{\partial}_{<p^{m+1}\underline{\varepsilon}_{j}>}\cdot \underline{\partial}_{<p^{s}\underline{\varepsilon}_{i}>}$ for any $a\in {\cal O}_{X}$, $1\leq i, j\leq r$ and $1\leq s\leq m$.
By Proposition \ref{proposition4} (2), we obtain $\underline{\partial}_{<p^{m+1}\underline{\varepsilon}_{j}>}\cdot \underline{\partial}_{<p^{s}\underline{\varepsilon}_{i}>}=\underline{\partial}_{<p^{s}\underline{\varepsilon}_{i}>}\cdot\underline{\partial}_{<p^{m+1}\underline{\varepsilon}_{j}>}$.
Note that, by Proposition \ref{proposition4} (4), we have $\underline{\partial}_{<p^{m+1}\underline{\varepsilon}_{j}>(m)}(a)=p!\underline{\partial}_{<p^{m+1}\underline{\varepsilon}_{j}>(m+1)}(a)=0$.
Therefore, by Proposition \ref{proposition4} (3), 
we also obtain 
	\begin{eqnarray*}
	\underline{\partial}_{<p^{m+1}\underline{\varepsilon}_{j}>}\cdot a&=&\sum_{i\leq p^{m+1}}\left\{
		\begin{array}{c}
		p^{m+1}\\
		i
		\end{array}
	\right\}
	\underline{\partial}_{<(p^{m+1}-i)\underline{\varepsilon}_{j}>}(a) \underline{\partial}_{<i\underline{\varepsilon}_{j}>}\\
	&=&a\cdot \underline{\partial}_{<p^{m+1}\underline{\varepsilon}_{j}>}.
	\end{eqnarray*} 
We finish the proof.
\end{proof}
Hence we obtain the morphism of rings
\begin{equation*}
 S^{\cdot}{\cal T}_{X'\slash S}\to F_{X\slash S*}{\cal D}^{(m)}_{X/S}.
\end{equation*}
We also denote this map by $\beta$.
Let us put $\mathfrak Z:=S^{\cdot}{\cal T}_{X'\slash S}$ and regard $\mathfrak Z$ as subring of ${\cal D}^{(m)}_{X/S}$ via the map $\beta$.
The following theorem is due to Schepler when $m$ is equal to zero (see Theorem 2.14 of \cite{S}).

\begin{theo}\label{theorem13}
Let $X\to S$ be a log smooth morphism of fine log schemes. Let $\tilde{\mathfrak Z}$ denote the center of $\tilde {{\cal D}}_{X/S}^{(m)}$. Then $\tilde{\mathfrak Z}$ is isomorphic to ${\cal B}_{X/S}^{(m+1)}\otimes_{{\cal O}_{X'}}\mathfrak Z$ as indexed subalgebra of $\tilde {{\cal D}}_{X/S}^{(m)}$.
\end{theo}

\begin{proof}
We may consider \'etale locally on $X$. Let $m_{1}, \ldots , m_{r}\in {\cal M}_{X}$ be a logarithmic system of coordinates. If $f\in {\cal B}_{X/S}^{(m+1)}$, then, since $\underline{\partial}_{<p^{s}\underline{\varepsilon}_{i}>}$.$f=0$ for any $0\leq s\leq m$, $1\leq i\leq r$ (Remark \ref{proposition3}), we have
\begin{eqnarray*}
(1\otimes \underline{\partial}_{<p^{s}\underline{\varepsilon}_{i}>})(f\otimes 1)&=&\sum_{j\leq p^{s}}\left\{
\begin{array}{c}
p^{s}\\
j
\end{array}
\right\}
(\underline{\partial}_{<(p^{s}-j)\underline{\varepsilon}_{i}>}.f)\otimes \underline{\partial}_{<j\underline{\varepsilon}_{i}>}\\
&=&f\otimes \underline{\partial}_{<p^{s}\underline{\varepsilon}_{i}>}\\
&=&(f\otimes 1)(1\otimes \underline{\partial}_{<p^{s}\underline{\varepsilon}_{i}>})
\end{eqnarray*}
for any $0\leq s\leq m$, $1\leq i\leq r$. Since $\{\underline{\partial}_{<p^{s}\underline{\varepsilon}_{i}>}| 0\leq s\leq m, 1\leq i\leq r\}$ generate ${\cal D}_{X/S}^{(m)}$ (Proposition \ref{proposition4}),
we see that $f$ commutes with ${\cal D}_{X/S}^{(m)}$ and it obviously commutes with ${\cal A}_{X}^{gp}$. Hence we have $f\in \tilde{\mathfrak Z}$ and see ${\cal B}_{X/S}^{(m+1)}\otimes_{{\cal O}_{X'}}\mathfrak Z \subset \tilde{\mathfrak Z}$. Conversely, since ${\cal A}_{X}^{gp}$ is locally free as an ${\cal I}_{X}^{gp}$-indexed ${\cal B}_{X/S}^{(m+1)}$-module with local basis $\{\underline{\theta}^{\underline{i}} | \underline{i}\in [0, p^{m+1}-1]^{r}\}$(see Proposition \ref{proposition2}), $\tilde {{\cal D}}_{X/S}^{(m)}$ is generated as a ${\cal B}_{X/S}^{(m+1)}$-algebra by $\{\theta_{i}, \underline{\partial}_{<p^{s}\underline{\varepsilon}_{i}>} |0\leq s\leq m, 1\leq i\leq r\}$.
First, for a local section $\varphi=\sum_{\underline{k}}f_{\underline{k}}\otimes \underline{\partial}_{<\underline{k}>}$ of $\tilde {{\cal D}}_{X/S}^{(m)}$, we now calculate  commutators $[\varphi, \theta_{i}]$ and $[\varphi, \underline{\partial}_{<p^{s}\underline{\varepsilon}_{i}>}]$.
\begin{eqnarray*}
[\varphi, \theta_{i}]&=&\sum_{\underline{k}}(f_{\underline{k}} \otimes \underline{\partial}_{<\underline{k}>})\cdot \theta_{i} - \sum_{\underline{k}}\theta_{i}f_{\underline{k}}\otimes \underline{\partial}_{<\underline{k}>}\\
&=&\sum_{\underline{k}}f_{\underline{k}}\left(\sum_{\underline{j}<{\underline{k}}}\left\{
\begin{array}{c}
\underline{k}\\
\underline{j}
\end{array}
\right\}
(\underline{\partial}_{<\underline{k}-\underline{j}>}.\theta_{i})\otimes \underline{\partial}_{<\underline{j}>}\right).
\end{eqnarray*}
By Proposition \ref{proposition1}, we have
\begin{equation}
[\varphi, \theta_{i}]=\sum_{\underline{k}}\left\{
\begin{array}{c}
k_{i}\\
k_{i}-1
\end{array}
\right\}
f_{\underline{k}}\theta_{i}\otimes \underline{\partial}_{<\underline{k}>}.\label{equation1}
\end{equation}
Similarly, 
\begin{eqnarray}
[\varphi, \underline{\partial}_{<p^{s}\underline{\varepsilon}_{i}>}]&=&\sum_{\underline{k}}f_{\underline{k}} \otimes \underline{\partial}_{<\underline{k}>} \underline{\partial}_{<p^{s}\underline{\varepsilon}_{i}>}- \sum_{\underline{k}}\sum_{0\leq l\leq p^{s}}\left\{
\begin{array}{c}
p^{s}\underline{\varepsilon}_{i}\nonumber \\
l\underline{\varepsilon}_{i}
\end{array}
\right\}
\underline{\partial}_{<(p^{s}-l)\underline{\varepsilon}_{i}>}.f_{\underline{k}}\otimes \underline{\partial}_{<l\underline{\varepsilon}_{i}>} \underline{\partial}_{<\underline{k}>}\\
&=&-\sum_{\underline{k}}\sum_{0\leq l<p^{s}}\left\{
\begin{array}{c}
p^{s}\underline{\varepsilon}_{i}\\
l\underline{\varepsilon}_{i}
\end{array}
\right\}
\underline{\partial}_{<(p^{s}-l)\underline{\varepsilon}_{i}>}.f_{\underline{k}}\otimes \underline{\partial}_{<l\underline{\varepsilon}_{i}>} \underline{\partial}_{<\underline{k}>}\nonumber \\
&=&-\sum_{\underline{k}}\sum_{0\leq l<p^{s}}
\underline{\partial}_{<(p^{s}-l)\underline{\varepsilon}_{i}>}.f_{\underline{k}}\otimes \underline{\partial}_{<l\underline{\varepsilon}_{i}>}\underline{\partial}_{<\underline{k}>}.\label{equation2}
\end{eqnarray}
Note that $\left\{
\begin{array}{c}
k_{i}\\
k_{i}-1
\end{array}
\right\}=0$ in characteristic $p$ if and only if $p^{m+1}| k_{i}$.
By (\ref{equation1}), if $\varphi \in \tilde{\mathfrak Z}$ then $f_{\underline{k}}\neq 0$ only for $\underline{k}$ with $p^{m+1}| \underline{k}$.
Hence, by (\ref{equation2}) and Lemma \ref{lem1}, we have
$\underline{\partial}_{<p^{s}\underline{\varepsilon}_{i}>}.f_{\underline{k}}=0$ for any $0<s \leq m, 1\leq i\leq r$, that is, $f_{\underline{k}}\in {\cal B}_{X/S}^{(m+1)}$.
Since the $p^{m+1}$-curvture map sends $\xi_{i}'$ to $\underline{\partial}_{<p^{m+1}\underline{\varepsilon}_{i}>}$, its image is generated as an ${\cal O}_{X'}$-algebra by $\{\underline{\partial}_{<p^{m+1}\underline{k}>}|\underline{k}\in {\mathbb N}^{r}\}$ by Lemma \ref{lem1}.
Therefore $\varphi \in \tilde{\mathfrak Z}$ implies $\varphi \in {\cal B}_{X/S}^{(m+1)}\otimes_{{\cal O}_{X'}}\mathfrak Z$. This completes the proof.
\end{proof}
We also calculate the centralizer of ${\cal A}_{X}^{gp}$ in $\tilde {{\cal D}}_{X/S}^{(m)}$ in a similar manner. By the calculation (\ref{equation1}), we also see the following.
\begin{pro}
Let ${\cal C}_{X}$ denotes the centralizer of ${\cal A}_{X}^{gp}$ in $\tilde {{\cal D}}_{X/S}^{(m)}$. Then ${\cal C}_{X}={\cal A}_{X}^{gp}\otimes_{{\cal O}_{X}}{\mathfrak Z}$ as an indexed subalgebra of $\tilde {{\cal D}}_{X/S}^{(m)}$.
\end{pro}
Let us consider $\tilde {{\cal D}}_{X/S}^{(m)}$ as a right ${\cal C}_{X}$-module by multiplication on the right.
\begin{theo}\label{Theorem5}
Let $X\to S$ be a log smooth morphism of fine log schemes. Then there is an isomorphism of ${\cal I}_{X}^{gp}$-indexed ${\cal C}_{X}$-algebras
	\begin{equation}
		\begin{array}{ccc}
		\tilde {{\cal D}}_{X/S}^{(m)}\otimes_{\tilde{\mathfrak Z}}{\cal C}_{X}& \longrightarrow &  {\cal E}nd_{{\cal C}_{X}}\left(\tilde{{\cal D}}_{X/S}^{(m)}\right)                    \\[-4pt]
 		\\[-4pt]
		\varphi \otimes c                     & \longmapsto     & \{x\mapsto \varphi \cdot x\cdot c\}.
		\end{array}\label{equation4}
	\end{equation}
\end{theo}
\begin{rem}
Theorem \ref{Theorem5} is due to Schepler in the case $m=0$ {\rm(see Theorem 2.15 of \cite{S})} and is due to Gros, Le Stum and Quir\'os when $X\to S$ is a smooth morphism of schemes \rm{(}see Theorem 3.7 of \cite{GLQ}$)$.
\end{rem}
\begin{proof}
Since the assertion is \'etale local on $X$, we may assume that there is a logarithmic system of coordinates $m_{1},\ldots m_{r}\in {\cal M}_{X}$. Then $\tilde {{\cal D}}_{X/S}^{(m)}\otimes_{\tilde{\mathfrak Z}}{\cal C}_{X}$ has a basis $\left\{1\otimes \underline{\theta}^{\underline{i}}|\underline{i}\in \left[0, p^{m+1}-1\right]^{r}\right\}$ as a left $\tilde {{\cal D}}_{X/S}^{(m)}$-module. We set
	\begin{equation*}
	\beta_{i}=\theta_{i}^{-1}\otimes \theta_{i}-1\otimes 1.
	\end{equation*}
Note that this sum is well-defined since $\theta_{i}^{-1}\otimes \theta_{i}$ and $1\otimes 1$ are both in the fiber of $\tilde {{\cal D}}_{X/S}^{(m)}\otimes_{\tilde{\mathfrak Z}}{\cal C}_{X}$ at $0$. We put $\underline{\beta}^{\underline{j}}:=\prod_{i=1}^{r}\beta_{i}^{j_{i}}$ for $\underline{j}\in [0, p^{m+1}-1]^{r}$. By the binomial theorem, we have
\begin{equation*}
\underline{\beta}^{\underline{j}}=\sum_{\underline{i}\leq \underline{j}}\left((-1)^{|\underline{j}-\underline{i}|}\binom{\underline{j}}{\underline{i}}\theta^{-\underline{i}}\right)(1\otimes \theta^{\underline{i}}).
\end{equation*}
Thus, if we endow $M$ with some suitable order compatible with the product partial order then the transition matrix from the basis $\left\{1\otimes \underline{\theta}^{\underline{i}}|\underline{i}\in \left[0, p^{m+1}-1\right]^{r}\right\}$ to the set
$\left\{\underline{\beta}^{\underline{j}}|\underline{j}\in \left[0, p^{m+1}-1\right]^{r}\right\}$ of $\tilde {{\cal D}}_{X/S}^{(m)}\otimes_{\tilde{\mathfrak Z}}{\cal C}_{X}$ is upper triangular, with units on the diagonal.
Therefore $\left\{\underline{\beta}^{\underline{j}}|\underline{j}\in M\right\}$ is also a basis for $\tilde {{\cal D}}_{X/S}^{(m)}\otimes_{\tilde{\mathfrak Z}}{\cal C}_{X}$.
On the other hand, let us consider the image of $\left\{\underline{\beta}^{\underline{j}}|\underline{j}\in \left[0, p^{m+1}-1\right]^{r}\right\}$ by the map (\ref{equation4}).
For $\underline{i}\in M$, let $\alpha_{\underline{i}}$ be the unique homomorphism which sends $\underline{\partial}_{<\underline{j}>}$ to $\delta_{\underline{i}\underline{j}}$.
Then, since $\tilde{\cal D}_{X/S}^{(m)}$ has a basis $\left\{\underline{\partial}_{<\underline{i}>}|\underline{i}\in \left[0, p^{m+1}-1\right]^{r}\right\}$ as a right ${\cal C}_{X}$-module, $\left\{\alpha_{\underline{i}}|\underline{i}\in \left[0, p^{m+1}-1\right]^{r}\right\}$ forms a basis for ${\cal E}nd_{{\cal C}_{X}}(\tilde {{\cal D}}_{X/S}^{(m)})$ as a left $\tilde {{\cal D}}_{X/S}^{(m)}$-module.
By Proposition \ref{proposition1}, we now calculate that  $\beta_{i}$ acts on $\underline{\partial}_{<\underline{k}>}$ by 
\begin{eqnarray*}
\beta_{i}\cdot  \underline{\partial}_{<\underline{k}>}&=&\theta_{i}^{-1}\cdot \underline{\partial}_{<\underline{k}>}\cdot \theta_{i}-\underline{\partial}_{<\underline{k}>}\\
&=&\theta_{i}^{-1}\left(\theta_{i}\otimes \underline{\partial}_{<\underline{k}>}+\left\{
\begin{array}{c}
k_{i}\\
k_{i}-1
\end{array}
\right\} \theta_{i}\otimes \underline{\partial}_{<\underline{k}-\underline{\varepsilon}_{i}>}\right)-\underline{\partial}_{<\underline{k}>}\\
&=&\left\{
\begin{array}{c}
k_{i}\\
k_{i}-1
\end{array}
\right\} \underline{\partial}_{<\underline{k}-\underline{\varepsilon}_{i}>}
\end{eqnarray*}
Hence $\underline{\beta}^{\underline{j}}$ acts on  $\tilde {{\cal D}}_{X/S}^{(m)}$ by sending  $\underline{\partial}_{<\underline{k}>}$ to $\frac{\underline{q}_{\underline{k}}!}{\underline{q}_{\underline{k}-\underline{j}}!} \underline{\partial}_{<\underline{k}- \underline{j}>}$ if $\underline{k}\leq \underline{j}$ and $0$ otherwise. We thus obtain
\begin{equation*}
\underline{\beta}^{\underline{j}}=\sum_{\underline{i}\leq \underline{j}, \underline{i}\in M}\left(\frac{\underline{q}_{\underline{i}}!}{\underline{q}_{\underline{i}-\underline{j}}!} \underline{\partial}_{<\underline{i}- \underline{j}>}\right)\alpha_{\underline{i}}
\end{equation*}
in ${\cal E}nd_{{\cal C}_{X}}(\tilde{{\cal D}}_{X}^{(m)})$ and  see the the transition matrix from the basis $\left\{\alpha_{\underline{i}}|\underline{i}\in \left[0, p^{m+1}-1\right]^{r}\right\}$ to the set $\left\{\underline{\beta}^{\underline{j}}|\underline{j}\in \left[0, p^{m+1}-1\right]^{r}\right\}$ of ${\cal E}nd_{{\cal C}_{X}}(\tilde{{\cal D}}_{X}^{(m)})$ is lower triangular with units on the diagonal.
We see that $\left\{\underline{\beta}^{\underline{j}}|\underline{j}\in \left[0, p^{m+1}-1\right]^{r}\right\}$ is also a basis for ${\cal E}nd_{{\cal C}_{X}}(\tilde{{\cal D}}_{X}^{(m)})$.
\end{proof}

\begin{cor}\label{cor1}
The ${\cal I}_{X}^{gp}$-indexed ${\cal O}_{X}$-algebra $\tilde{{\cal D}}_{X/S}^{(m)}$ is an Azumaya algebra over its center $\tilde{\mathfrak Z}$ of rank $p^{2(m+1)r}$ where $r$ is the rank of $\Omega^{1}_{X/S}$.
\end{cor}
\begin{proof}
From Proposition \ref{proposition2}, ${\cal A}_{X}^{gp}$ is locally free as an ${\cal I}_{X}^{gp}$-indexed ${\cal B}_{X/S}^{(m+1)}$-module. Thus ${\cal C}_{X}\cong {\cal A}_{X}^{gp}\otimes_{{\cal B}_{X/S}^{(m+1)}} \tilde{\mathfrak Z}$ is a faithfully flat extension of $\tilde{\mathfrak Z}$. By Theorem \ref{Theorem5}, $\tilde{{\cal D}}_{X/S}^{(m)}$ splits over ${\cal C}_{X}$ with splitting module $\tilde{{\cal D}}_{X/S}^{(m)}$.
\end{proof}
The proof of the Azumaya nature of $\tilde{{\cal D}}_{X^{(m)}/S}^{(0)}$ is the same as that of $\tilde{{\cal D}}_{X/S}^{(m)}$ except for obvious modifications.
\begin{theo}\label{theorem30}
Let $X\to S$ be a log smooth morphism of fine log schemes.

{\rm (1)}
Let $\tilde{\mathfrak Z'}$ denote the center of $\tilde {{\cal D}}_{X^{(m)}/S}^{(0)}$.
Then the $p$-curvature map ${\mathfrak Z}\to {\cal D}_{X^{(m)}/S}^{(0)}$ induces an isomorphism between ${\cal B}_{X/S}^{(m+1)}\otimes_{{\cal O}_{X'}}\mathfrak Z$ and $\tilde{\mathfrak Z'}$ as an indexed subalgebra of $\tilde {{\cal D}}_{X^{(m)}/S}^{(0)}$.

{\rm (2)}
The ${\cal I}_{X}^{gp}$-indexed ${\cal O}_{X^{(m)}}$-algebra $\tilde{{\cal D}}_{X^{(m)}/S}^{(0)}$ is an Azumaya algebra over its center $\tilde{\mathfrak Z'}$ of rank $p^{2r}$ where $r$ is the rank of $\Omega^{1}_{X/S}$.
\end{theo}

\begin{proof}
We shall only sketch a proof.
Take $\varphi \in \tilde {{\cal D}}_{X^{(m)}/S}^{(0)}$.
As in the proof of Theorem \ref{theorem13}, we consider the necessary and sufficient condition for $\varphi \in \tilde{\mathfrak Z'}$.
Recall that ${\cal B}_{X/S}^{(m)}$ is locally free as an ${\cal I}_{X}^{gp}$-indexed ${\cal B}_{X/S}^{(m+1)}$-module with a local basis $\left\{\,\underline{\theta}^{p^{m}\underline{j}}\,|\,\underline{j}\in \left[0, p-1\right]^{r}\,\right\}$
(see Proposition \ref{proposition14}).
Thus $\tilde {{\cal D}}_{X^{(m)}/S}^{(0)}$ is generated by 
$\left\{\theta_{i}^{p^{m}}, \underline{\partial}_{<\underline{\varepsilon}_{i}>(0)} | 1\leq i\leq r\right\}$
as an ${\cal I}_{X}^{gp}$-indexed ${\cal B}_{X/S}^{(m+1)}$-algebra.
Let us describe $\varphi=\sum_{\underline{k}}f_{\underline{k}}\otimes \underline{\partial}_{<\underline{k}>(0)}$ with $f_{\underline{k}}\in {\cal B}_{X/S}^{(m+1)}$.
Then commutators $[\varphi, \theta^{p^{m}}_{i}]$ and $[\varphi, \underline{\partial}_{<\underline{\varepsilon}_{i}>}]$ can be calculated as follows.
\begin{equation}\label{formula1}
[\varphi, \theta^{p^{m}}_{i}]=\sum_{\underline{k}} k_{i}f_{\underline{k}} \theta^{p^{m}}_{i} \otimes \underline{\partial}_{<\underline{k}-\underline{\varepsilon}_{i}>(0)}.
\end{equation}
\begin{equation}\label{formula2}
[\varphi, \underline{\partial}_{<\underline{\varepsilon}_{i}>}]=-\sum_{\underline{k}} \underline{\partial}_{<\underline{\varepsilon}_{i}>(0)} f_{\underline{k}}\otimes \underline{\partial}_{<\underline{k}>}.
\end{equation}
By (\ref{formula1}), if $\phi \in \tilde{\mathfrak Z'}$ then $f_{\underline{k}}\neq 0$ only for $p | \underline{k}$.
By (\ref{formula2}) and ({\ref{formula3}}), we have $\underline{\partial}_{<p^{m}\underline{\varepsilon}_{i}>(m)} f_{\underline{k}}=\underline{\partial}_{<\underline{\varepsilon}_{i}>(0)} f_{\underline{k}}=0$ for any $1\leq i\leq r$, that is, $f_{\underline{k}}\in {\cal B}_{X/S}^{(m+1)}$.
This finishes the proof of (1).
For proving (2), let ${\cal C}_{X^{(m)}}$ be the centralizer of ${\cal B}_{X/S}^{(m)}$ in $\tilde {{\cal D}}_{X^{(m)}/S}^{(0)}$ which is isomorphic to ${\cal B}_{X/S}^{(m)}\otimes_{{\cal O}_{X'}}{\mathfrak Z'}$.
Then, by an analogous argument to the proof of Theorem \ref{Theorem5}, one can obtain the isomorphism
	\begin{equation}
		\begin{array}{ccc}
		\tilde {{\cal D}}_{X^{(m)}/S}^{(0)}\otimes_{\tilde{\mathfrak Z'}}{\cal C}_{X^{(m)}}& \xrightarrow{\cong} &  {\cal E}nd_{{\cal C}_{X^{(m)}}}\left(\tilde{{\cal D}}_{X^{(m)}/S}^{(0)}\right)                    \\[-4pt]
 		\\[-4pt]
		\varphi \otimes c                     & \longmapsto     & \left\{x\mapsto \varphi \cdot x\cdot c\right\}.
		\end{array}
	\end{equation}
Therefore (2) follows from Proposition \ref{proposition14}.
\end{proof}
\begin{rem}\label{remark31}
Theorem \ref{theorem30} is a variant of Theorem 2.14 and Theorem 2.15 in \cite{S}.
\end{rem}

\subsection{The log Cartier descent theorem}

As a first application of Corollary \ref{cor1}, we prove the log Cartier descent theorem of higher level. (See Corollary 3.2.4 of \cite{LQ} for the case without log structure.)\\

First we generalize the notion of admissible connection to the case of higher level.  
\begin{defi}
Let $\cal J$ be a sheaf of ${\cal I}_{X}^{gp}$-sets on $X$ and let $\cal E$ be a $\cal J$-indexed ${\cal A}_{X}^{gp}$-module with an $\cal J$-indexed left ${\cal D}_{X/S}^{(m)}$-module structure.
Then the action of ${\cal D}_{X/S}^{(m)}$ on $\cal E$ is admissible if the $\cal J$-indexed ${\cal A}_{X}^{gp}$-module structure and the $\cal J$-indexed left ${\cal D}_{X/S}^{(m)}$-module structure on $\cal E$ extend to the $\cal J$-indexed  $\tilde {{\cal D}}_{X/S}^{(m)}$-module structure on $\cal E$.
We also define the notion of admissibleness for $\cal J$-indexed ${\cal B}_{X/S}^{(m)}$-module with an $\cal J$-indexed left ${\cal D}_{X^{(m)}/S}^{(0)}$-module structure in a similar manner.

\end{defi}
We give an equivalent condition for the action of ${\cal D}_{X/S}^{(m)}$ on a $\cal J$-indexed ${\cal A}_{X}^{gp}$-module to be admissible.
Let $\cal E$ be $\cal J$-indexed ${\cal A}_{X}^{gp}$-module.
Recall from Remark \ref{Remark2} that the action of ${\cal A}_{X}^{gp}$ on $\cal E$ is equivalent to the family of morphisms $\rho_{ij}:{\cal A}_{X, i}^{gp}\otimes_{{\cal O}_{U}}{\cal E}_{j}\to{\cal E}_{i+j}$ satisfying the suitable conditions.
\begin{pro}
Let $\cal J$ be a sheaf of ${\cal I}_{X}^{gp}$-sets on $X$ and let $\cal E$ be a $\cal J$-indexed ${\cal A}_{X}^{gp}$-module with an $\cal J$-indexed left ${\cal D}_{X/S}^{(m)}$-module structure.
Then the action of ${\cal D}_{X/S}^{(m)}$ on $\cal E$ is admissible if and only if the corresponding structural morphism $\rho_{ij} :{\cal A}_{X, i}^{gp}\otimes_{{\cal O}_{U}} {\cal E}_{j}\to {\cal E}_{i+j}$;  $a\otimes e\mapsto ae$ is a ${\cal D}_{X/S}^{(m)}$-homomorphism for any \'etale open $U$ of $X$ and  any section $(i, j)\in {\cal I}_{X}^{gp}\times {\cal J}$. Here ${\cal A}_{X, i}^{gp}\otimes_{{\cal O}_{U}} {\cal E}_{j}$ is a tensor product as a left ${\cal D}_{X/S}^{(m)}$-module.
\end{pro}

\begin{proof}
We show the if part. Let us show that the action of $\tilde {{\cal D}}_{X/S}^{(m)}$ on $\cal E$ defined by $\tilde {{\cal D}}_{X, i}^{(m)}\otimes _{{\cal O}_{U}}{\cal E}_{j}\to {\cal E}_{i+j}$ $(a\otimes P)\otimes e\mapsto a. P. e$ is well-defined. We may work locally and by assumption, we have 
\begin{equation*}
\underline{\partial}_{<\underline{k}>}.(a.e)=\sum_{\underline{i}\leq \underline{k}}\left\{
\begin{array}{c}
\underline{k}\\
\underline{i}
\end{array}
\right\}\underline{\partial}_{<\underline{k}-\underline{i}>}.a\underline{\partial}_{<\underline{i}>}.e.
\end{equation*} 
On the other hand, we have
\begin{eqnarray*}
(\underline{\partial}_{<\underline{k}>}.a).e&=&\left(\sum_{\underline{i}\leq \underline{k}}\left\{
\begin{array}{c}
\underline{k}\\
\underline{i}
\end{array}
\right\} \underline{\partial}_{<\underline{k}-\underline{i}>}.a\otimes \underline{\partial}_{<\underline{i}>}\right).e\\
&=&\sum_{\underline{i}\leq \underline{k}} \left\{
\begin{array}{c}
\underline{k}\\
\underline{i}
\end{array}
\right\}\underline{\partial}_{<\underline{k}-\underline{i}>}.a\underline{\partial}_{<\underline{i}>}.e.
\end{eqnarray*}
This completes the proof.
\end{proof}
\begin{rem}\label{lemma25}
We can also give a condition for the action of ${\cal D}_{X/S}^{(m)}$ on a $\cal J$-indexed ${\cal A}_{X}^{gp}$-module to be admissible by using the notion of log $m$-stratification.
Let $p_{i}^{n}$ (with $i=0, 1$) be the natural projection $P_{X/S, (m)}^{n}\to X$ and
\{$\varepsilon_{{\cal A}, n}$\} (resp. \{$\varepsilon_{{\cal E}, n}$\}) the log $m$-stratification on ${\cal A}_{X}^{gp}$ (resp. on ${\cal E}$) associated to the ${\cal D}_{X/S}^{(m)}$-action.
Then a $\cal J$-indexed left ${\cal D}_{X/S}^{(m)}$-module structure on $\cal J$-indexed ${\cal A}_{X}^{gp}$-module $\cal E$ is admissible if and only if the following diagram is commutative for any positive integer $n$: 
	\[\xymatrix{
	{p_{1}^{n*}\bigl({\cal A}_{X}^{gp}\otimes_{{\cal O}_{X}}{\cal E}}\bigr) \ar[r] \ar[d] & {p_{0}^{n*}\bigl({\cal A}_{X}^{gp}\otimes_{{\cal O}_{X}}{\cal E}}\bigr) \ar[d]\\
	{p_{1}^{n*}{\cal E}} \ar[r]	
	 & {p_{0}^{n*}{\cal E}}
	,}\]
	where the upper horizontal arrows is the tensor product $\varepsilon_{{\cal A}, n}\otimes \varepsilon_{{\cal E}, n}$,
	the under horizontal arrow is $\varepsilon_{{\cal E}, n}$
	and the vertical arrows are induced from the ${\cal A}_{X}^{gp}$-action on $\cal E$.

\end{rem}
\begin{rem}\label{lemma26}
Let $\cal E$ be a ${\cal D}_{X/S}^{(m)}$-module.
Then the ${\cal D}_{X/S}^{(m)}$-action on the tensor product ${\cal A}_{X}^{gp}\otimes_{{\cal O}_{X}}{\cal E}$ as a ${\cal D}_{X/S}^{(m)}$-module is admissible.
\end{rem}

For a $\cal J$-indexed $\tilde {{\cal D}}_{X/S}^{(m)}$-module $\cal E$, we put ${\cal E}^{\nabla}:={\cal H}om_{{\cal D}_{X/S}^{(m)}}({\cal O}_{X}, {\cal E})$.

\begin{theo}\label{theorem40}
Let $X\to S$ be a log smooth morphism of fine log schemes. Let $\cal J$ be a sheaf of ${\cal I}_{X}^{gp}$-sets on $X$. Then the functor $\cal E\mapsto {\cal E}^{\nabla}$ give an equivalence between the category of $\cal J$-indexed ${\cal A}_{X}^{gp}$-modules with an  admissible $\cal J$-indexed left ${\cal D}_{X/S}^{(m)}$-module structure and zero $p^{m+1}$-curvature and the category of $\cal J$-indexed ${\cal B}_{X/S}^{(m+1)}$-modules.
\end{theo}
\begin{proof}
We use Proposition \ref{Proposition5} to construct the equivalence.
We consider ${\cal B}_{X/S}^{(m+1)}$ as a $\tilde{\mathfrak Z}$-algebra via the composite $S^{\cdot}({\cal T}_{X'/S})\to S^{\cdot}({\cal T}_{X'/S})/S^{+}({\cal T}_{X'/S})\xrightarrow{\cong} {\cal O}_{X'}\to {\cal B}_{X/S}^{(m+1)}$.
By Corollary \ref{cor1}, $\tilde{\cal D}_{0}:=\tilde{{\cal D}}_{X}^{(m)}\otimes_{\tilde{\mathfrak Z}}{\cal B}_{X/S}^{(m+1)}$ is an Azumaya algebra of rank $p^{2(m+1)r}$. Let us find the splitting module of $\tilde{\cal D}_{0}
$. Since ${\cal A}_{X}^{gp}$ is a locally free ${\cal B}_{X/S}^{(m+1)}$-module of rank $p^{(m+1)r}$ which has a structure of left $\tilde{\cal D}_{0}$-module, $\tilde{\cal D}_{0}$ splits over ${\cal B}_{X/S}^{(m+1)}$ with splitting module ${\cal A}_{X}^{gp}$ by Proposition \ref{Proposition7}.
Hence, we can apply Proposition \ref{Proposition5} and get the equivalence of categories ${\cal E}\mapsto {\cal H}om_{\tilde{\cal D}_{0}}({\cal A}_{X}^{gp}, {\cal E})$ between the category of $\cal J$-indexed left $\tilde{\cal D}_{0}$-modules and the category of $\cal J$-indexed ${\cal B}_{X/S}^{(m+1)}$-modules. Now the notion of $\cal J$-indexed left $\tilde{\cal D}_{0}$-module is equivalent to that of $\cal J$-indexed ${\cal A}_{X}^{gp}$-module with an  admissible $\cal J$-indexed left ${\cal D}_{X/S}^{(m)}$-module structure and zero $p^{m+1}$-curvature, and there is a natural isomorphism ${\cal H}om_{\tilde{\cal D}_{0}}({\cal A}_{X}^{gp}, {\cal E})\cong {\cal E}^{\nabla}$. This completes the proof.
\end{proof}

\section{The global Cartier transform}
The goal of this section is to construct the log global Cartier transform of higher level.
First we recall a few notions on the log crystalline theory of higher level needed later.
It should be remarked here that Miyatani studied the foundations of log crystalline theory of higher level in his unpublished master thesis in the University of Tokyo \cite{Mi}.
\subsection{Logarithmic crystalline site of level $m$}
We fix throughout this subsection an $m$-PD fine log scheme $(S, \mathfrak a,\mathfrak b, \gamma)$ and a fine log scheme $X$ over $S$. We assume that the $m$-PD structure $\gamma$ extends to $X$.
\begin{defi}
Let $U$ be a fine log scheme over $X$. A log $m$-PD thickening $(U, T, J, \delta)$ of $U$ over $(S, \mathfrak a,\mathfrak b, \gamma)$ is a data which consists of a fine log scheme $T$ over $S$, an exact closed immersion $U\hookrightarrow T$ over $S$ and an $m$-PD structure $(J, \delta)$ on the defining ideal of $U\hookrightarrow T$ compatible with $({\mathfrak b}, \gamma)$.
A morphism of log $m$-PD thickenings over $(S, \mathfrak a,\mathfrak b, \gamma)$ can be defined in an obvious way.
\end{defi}
\begin{defi}
{\rm(1)} The log $m$-crystalline site ${\rm{Cris}}^{(m)}(X/S)$ is the category of log $m$-PD thickenings $(U, T, J, \delta)$ of an \'etale open $U$ of $X$ over $(S, \mathfrak a,\mathfrak b, \gamma)$ endowed with the topology induced by the \'etale topology on $T$. Its associated topos $(X/S)^{(m)}_{\rm cris}$ is called the log $m$-crystalline topos.

{\rm(2)} The sheaf of rings defined by
\begin{equation*}
(U, T, J, \delta)\longmapsto \Gamma(T, {\cal O}_{T})
\end{equation*}
in the topos $(X/S)^{(m)}_{\rm cris}$ is called the structure sheaf of the site ${\rm{Cris}}^{(m)}(X/S)$ and we denote it by ${\cal O}_{X/S}^{(m)}$.
\end{defi}
\begin{rem}
A sheaf $E$ on ${\rm{Cris}}^{(m)}(X/S)$ is equivalent to the following data: For every log $m$-PD thickening $(U, T, J, \delta)$, an \'etale sheaf $E_{T}$ on $T$, and for every morphism $u:T_{1}\to T_{2}$ in ${\rm{Cris}}^{(m)}(X/S)$, a map $\rho_{u}:u^{-1}(E_{T_{2}})\to E_{T_{1}}$, satisfying the cocycle condition such that $\rho_{u}$ is an isomorphism if $u$ is \'etale.
\end{rem}
Next we define the notion of log $m$-crystal.
\begin{defi}
Let $E$ be an ${\cal O}_{X/S}^{(m)}$-module in $(X/S)^{(m)}_{\rm cris}$. Then, $E$ is called a log $m$-crystal in ${\cal O}_{X/S}^{(m)}$-modules if, for all morphism $f: (U, T, J, \delta)\to (U', T', J', \delta')$ of $\rm{Cris}^{(m)}(X/S)$, the canonical morphism
\begin{equation*}
f^{*}(E_{(U', T', J', \delta')})\longrightarrow E_{(U, T, J, \delta)}
\end{equation*}
is an isomorphism.
\end{defi}
\begin{defi}
Let $M$ be an ${\cal O}_{X}$-module. Then a log hyper $m$-PD stratification on $M$ is a ${\cal P}_{X/S, (m)}$-linear isomorphism
\begin{equation*}
\varepsilon: {\cal P}_{X/S, (m)}\otimes M \to M\otimes {\cal P}_{X/S, (m)}
\end{equation*}
which is reduced to the identity map on $M$ modulo the kernel of ${\cal P}_{X/S, (m)}\to {\cal O}_{X}$ and satisfies the usual cocycle condition.
\end{defi}

\begin{pro}\label{Proposition8}
Let $({\mathfrak a_{0}, \mathfrak b_{0}}, \gamma_{0})$ be a quasi-coherent $m$-PD subideal of $\mathfrak a$, let $S_{0}\hookrightarrow S$ denote the exact closed immersion defined by $\mathfrak a_{0}$, and $X_{0}\hookrightarrow X$ its base change by $X\to S$. We assume that $X$ is log smooth and flat over $S$. Then the following categories are equivalent
\begin{enumerate}
\item[\rm (1)] The category of  log $m$-crystals in ${\cal O}_{X_{0}/S}^{(m)}$-modules
\item[\rm (2)] The category of ${\cal O}_{X}$-modules equipped with a log hyper $m$-PD stratification. 
\end{enumerate}
\end{pro}
\begin{proof}
The proof is the same as the classical case. 
It suffices to see that the log $m$-PD envelope of $X_0$ in 
$X \times_S X, X \times_S X \times_S X$ is equal to 
the log $m$-PD envelope of $X$ in 
$X \times_S X, X \times_S X \times_S X$, respectively. 
This follows from the fact that the $m$-PD structure of the latter 
is compatible with the $m$-PD structure 
$({\mathfrak b_{0}}, \gamma_{0})$. 
\end{proof}
Finally, for technical reason, we introduce a variant of a big crystalline site.
\begin{defi}
We define a site ${\rm CRIS}_{\rm Int}^{(m)}(X/S)$ as the category of log $m$-PD thickenings $(U, T, J, \delta)$ of $U$ over $S$ such that $U$ is any fine log scheme over $X$ and that $T\to S$ is integral endowed with the topology induced from the \'etale topology on $T$.
\end{defi}
\begin{rem} \label {Remark7}
As in the classical case, we obtain an equivalence
\begin{equation*}
 \left(
\begin{array}{c}
\text{The category of} \\
\text{log $m$-crystals of }{\cal O}^{(m)}_{X/S} \text{-modules}\\
\text{on }{\rm Cris}^{(m)}(X/S)
\end{array}
\right)
\to 
\left(
\begin{array}{c}
\text{The category of} \\
\text{${\cal O}_{X}$-modules with a log $m$-PD stratification}
\end{array}
\right)
\end{equation*}
by the following way.
Let $\cal E$ be a log $m$-crystal of ${\cal O}^{(m)}_{X/S}$-module.
Then, for each natural number $n$, the natural morphism  $(X\hookrightarrow P^{n}_{X/S, (m)}) \to (X \xrightarrow{\rm id} X)$ in ${\rm Cris}^{(m)}(X/S)$ defines an isomorphism of
${\cal P}^{n}_{X/S, (m)}$-modules ${\varepsilon}_{n}: p_{0}^{n*}{\cal E}_{X}\xrightarrow{\cong} {\cal E}_{P^{n}_{X/S, (m)}}\xleftarrow{\cong}p_{1}^{n*}{\cal E}_{X}$.
These isomorphisms define a log $m$-stratification on an ${\cal O}_{X}$-module ${\cal E}:={\cal E}_{X}$.
If $X\to S$ is log smooth and integral, which is of our interest in the sequel, then log $m$-PD envelopes of $X$ in $X \times_S X$ and $X \times_S X \times_S X$ are integral over $S$.
So we also obtain an equivalence
\begin{equation*}
 \left(
\begin{array}{c}
\text{The category of} \\
\text{log $m$-crystals of }{\cal O}^{(m)}_{X/S} \text{-modules}\\
\text{on }{\rm CRIS}_{\rm Int}^{(m)}(X/S)
\end{array}
\right)
\to 
\left(
\begin{array}{c}
\text{The category of} \\
\text{${\cal O}_{X}$-modules with a log $m$-PD stratification}
\end{array}
\right)
\end{equation*}
in a similar manner.
Hence the category of log $m$-crystals on ${\rm CRIS}_{\rm Int}^{(m)}(X/S)$ is equivalent to that of ${\rm{Cris}}^{(m)}(X/S)$ in this case.
\end{rem}
\begin{rem}
Let $(U, T, J, \delta)$ be an object in ${\rm{CRIS}}_{\rm Int}^{(m)}(X/S)$. 
Because $U$ and $T$ are integral over $S$ by definition, the underlying scheme of $U''=U\times_{S}S$ (resp. $T''=T\times_{S}S$) coincide with the fiber product of the diagram $U\to S\leftarrow S$ (resp. $T\to S\leftarrow S$) in the category of schemes, where $S\to S$ is the ($m+1$)-st iterate of its absolute Frobenius endomorphism of $S$. See also the Subsection 3.2.
\end{rem}
\subsection{The global Cartier transform}

Let us set some notation.
Let $X\to S$ be a log smooth morphism of fine log schemes defined over ${\mathbb Z}/p{\mathbb Z}$.
Fix a sheaf of ${\cal I}_{X}^{gp}$-sets $\cal J$ on $X$.
$\cal G$ denotes the nilpotent divided power envelope of the zero section of the cotangent bundle of $X'/S$, so that ${\cal O_{G}}=\hat{\Gamma}.{\cal T}_{X'/S}$. We put ${\cal O_{G}^{B}}:={\cal B}_{X/S}^{(m+1)}\otimes_{{\cal O}_{X'}}\cal O_{G}$, ${\cal O_{G}^{A}}:={\cal A}_{X}^{gp}\otimes_{{\cal O}_{X'}}\cal O_{G}$, and $\tilde {{\cal D}}_{X/S}^{(m), \gamma}:=\tilde {{\cal D}}_{X/S}^{(m)}\otimes_{\tilde{\mathfrak Z}}{\cal O_{G}^{B}}$.
We denote by ${\rm HIG}^{\cal B, J}_{\rm PD}(X'/S)$ the category of $\cal J$-indexed ${\cal O_{G}^{B}}$-modules, and by ${\rm MIC}^{\cal A, J}_{\rm PD}(X/S)$ the category of $\cal J$-indexed $\tilde {{\cal D}}_{X/S}^{(m), \gamma}$-modules.
Note that an object of ${\rm HIG}^{\cal B, J}_{\rm PD}(X'/S)$ is equivalent to a $\cal J$-indexed ${\cal B}_{X/S}^{(m+1)}$-module $E'$ equipped with a homomorphism of ${\cal O}_{X'}$-algebras 
\begin{equation*}
\theta:{\cal O_{G}}\to {\cal E}nd_{{\cal B}_{X/S}^{(m+1)}}(E') 
\end{equation*}
(called a ${\cal B}_{X/S}^{(m+1)}$-linear $\cal G$-Higgs field). 
Similarly, an object of ${\rm MIC}^{\cal A, J}_{\rm PD}(X/S)$ is equivalent to a $\cal J$-indexed ${\cal A}_{X}^{gp}$-module $E$ with an admissible ${\cal D}_{X/S}^{(m)}$-action endowed with 
a homomorphism of algebras 
\begin{equation*}
\theta:{\cal O_{G}}\to F_{X/S*}{\cal E}nd_{\tilde {{\cal D}}_{X/S}^{(m)}}(E)
\end{equation*}
extending the map 
\begin{equation*}
\psi:S^{\cdot}{\cal T}_{X'/S}\to F_{X/S*}{\cal E}nd_{\tilde {{\cal D}}_{X/S}^{(m)}}(E)
\end{equation*}
given by the $p^{m+1}$-curvature 
(called a horizontal ${\cal A}_{X}^{gp}$-linear $\cal G$-Higgs field).
The global Cartier transform is formulated as an equivalence of categories between ${\rm MIC}^{\cal A, J}_{\rm PD}(X/S)$ and ${\rm HIG}^{\cal B, J}_{\rm PD}(X'/S)$.
As in \cite{OV} and \cite{S}, first we study the lifting torsor of the $(m+1)$-st relative Frobenius morphism in the context of crystals.
\begin{defi}\label{Definition1}
Let $f:Y\to Z$ be a morphism of fine log schemes defined over ${\mathbb Z}/p{\mathbb Z}$. Then a lifting of $f$ modulo $p^{n}$ is a morphism $\tilde{f}:\tilde{Y}\to \tilde{Z}$ of fine log schemes flat over ${\mathbb Z}/p^{n}{\mathbb Z}$ which fits into a cartesian square in the category of fine log schemes
	\begin{equation*}
		\begin{CD}
			Y @>>>\tilde{Y}\\
			@VVV @VVV\\
			Z @>>> \tilde{Z},\\
		\end{CD}
	\end{equation*}
where $Z\to \tilde{Z'}$ is the exact closed immersion defined by $p$.
\end{defi}
If $f$ is log smooth, respectively log \'etale, resp. integral, resp. exact, so is $\tilde{f}$. From now on, we are mainly concerned with liftings modulo $p^{2}$.\\

For the rest of this paper, we consider an integral log smooth morphism $f:X\to S$ of fine log schemes defined over ${\mathbb Z}/p{\mathbb Z}$ equipped with a lifting $\tilde{X'}\to \tilde{S}$ of $X'\to S$ modulo $p^{2}$,
and we regard $S$, $\tilde{S}$ as $m$-PD fine log schemes with the canonical log structure on $(p)$.
We denote the data $(X\to S, \tilde{X'}\to \tilde{S})$ by $\cal X/S$. 
\begin{lemm}\label{Lemma2}
Let $(U, T, J, \delta)$ be an object of ${\rm{CRIS}}_{\rm Int}^{(m)}(X/S)$. Then there exists a canonical morphism $T\to U'$ such that the following diagram commutes
\[\xymatrix{
 {U} \ar[rd] \ar[rr] & & U'\\
 & T\ar[ru] &}
 ,
\]
where the symbol ${}'$ is as in the beginning of Subsection 3.2.
\end{lemm}
\begin{proof}
We use the symbol ${}''$ as in the beginning of Subsection 3.2.
Since $J$ is an $m$-PD ideal, we have $a^{p^{m+1}}=p!a^{\{p^{m+1}\}}=0$ in characteristic $p>0$ for any $a\in J$. 
Therefore $U\to T$ is a homeomorphism.
Since $a^{p^{m+1}}=0$ for any $a\in J$, the morphism ${\cal O}_{S}\otimes_{{\cal O}_{S}}{\cal O}_{T}\to {\cal O}_{T};\,  a\otimes b\mapsto a\cdot b^{p^{m+1}}$ induces a natural morphism ${\cal O}_{U''}\to {\cal O}_{T}$.
Since $U''\to T''$ is strict, we also have a natural morphism ${\cal M}_{U''}\simeq {\cal M}_{T''}\to {\cal M}_{T}$.
Therefore the morphism $T\to T''$ factors through $U''$.
We thus obtain solid arrows in the diagram
\[\xymatrix{
{U} \ar[d] \ar[r] & U' \ar[d] \\
T \ar[r]\ar@{-->}[ru] & U'' .}
\] 
Then, since the morphism $U'\to U''$ is log \'etale and $U\to T$ is an exact nilimmersion, there exists a unique morphism $T\to U'$ (the dotted arrow) making the above diagram commute.
\end{proof}
\begin{defi}
Let $(U, T, J, \delta)$ be an object in ${\rm{CRIS}}_{\rm Int}^{(m)}(X/S)$.
We define the morphism $T\to X'$ by the composition of the morphism in Lemma \ref{Lemma2} and $U'\to X'$, and denote it by $f_{T/S}$.
\end{defi}
If $g:T_{1}\to T_{2}$ be a morphism of ${\rm{CRIS}}_{\rm Int}^{(m)}(X/S)$, then $f_{T_{2}/S}\circ g=f_{T_{1}/S}$. Hence if $E'$ is an ${\cal O}_{X'}$-module, there exists a natural isomorphism
\begin{equation*}
\theta_{g}:g^{*}f^{*}_{T_{2}/S}E'\xrightarrow{\cong} f^{*}_{T_{1}/S}E'.
\end{equation*}
Thus we have the following lemma.
\begin{lemm}\label{Lem3}
Let $E'$ be an ${\cal O}_{X'}$-module.
The collection $\bigl\{ f^{*}_{T/S}E', \theta_{g}\bigr\}$ defines a log $m$-crystal of ${\cal O}^{(m)}_{X/S}$-module on ${\rm CRIS}_{\rm Int}^{(m)}(X/S)$. 
We denote it by $F_{X/S}^{*}E'$, by abuse of notation.
\end{lemm}
\begin{defi}\label{definition20}
Let $\tilde{T}$ be an object of ${\rm{CRIS}}_{\rm Int}^{(m)}(X/\tilde{S})$ which is flat over $\tilde{S}$, and $T$ the closed subscheme defined by $p$.
A lifting of $f_{T/S}$ to $\tilde{T}$ is a lifting $\tilde{F}:\tilde{T}\to \tilde{X'}$ over $\tilde{S}$ modulo $p^{2}$. ${\cal L}^{(m)}_{\cal X/S}(\tilde{T})$ denotes the set of all such liftings, and ${\cal L}^{(m)}_{{\cal X/S}, \tilde{T}}$ denotes the \'etale sheaf of sets on $\tilde{T}$ of local liftings of $f_{T/S}$.
For a morphism $\tilde{g}:\tilde{T_{1}}\to \tilde{T_{2}}$ of ${\rm{CRIS}}_{\rm Int}^{(m)}(X/S)$, we define the map ${\cal L}^{(m)}_{\cal X/S}(\tilde{g}):{\cal L_{X/S}}(\tilde{T_{2}})\to {\cal L_{X/S}}(\tilde{T_{1}})$ by $\tilde{F}\mapsto \tilde{F}\circ \tilde{g}$.
\end{defi}
Let ${\rm{CRIS}}_{{\rm Int},f}^{(m)}(X/\tilde{S})$ denote the full subsite of ${\rm{CRIS}}_{\rm Int}^{(m)}(X/\tilde{S})$
consisting of those objects which are flat over $\tilde{S}$.
The family $\Bigl\{\, {\cal L}^{(m)}_{{\cal X/S}, \tilde{T}}\, \Bigl|\, \tilde{T}\in {\rm{CRIS}}_{{\rm Int},f}^{(m)}(X/\tilde{S})\,\Bigr\}$ together with the family of transition maps ${\cal L}^{(m)}_{\cal X/S}(\tilde{g})$ defines a sheaf of sets on ${\rm{CRIS}}_{{\rm Int},f}^{(m)}(\tilde{X}/S)$.
\begin{lemm}\label{Lem8}
Let $(\tilde{U}, \tilde{T}, \tilde{J}, \tilde{\delta})$ be an object of ${\rm{CRIS}}_{{\rm Int},f}^{(m)}(X/\tilde{S})$.
Let $T$ denote the closed subscheme of $\tilde{T}$ defined by $p$. Then the sheaf ${\cal L}^{(m)}_{{\cal X/S}, \tilde{T}}$ forms a torsor over ${\cal H}om_{{\cal O}_{T}}(f^{*}	_{T/S}\Omega^{1}_{X'/S}, {\cal O}_{T})\cong f^{*}_{T/S}{\cal T}_{X'/S}$.
\end{lemm}
\begin{proof}
Let us consider the following diagram
	\[\xymatrix{
	T \ar[d] \ar[r]^{f_{T/S}} & X' \ar[r] & \tilde{X'}\ar[d]\\
 	\tilde{T} \ar[rr] & & \tilde{S}. }
	\] 
Since $\tilde{X'}\to \tilde{S}$ is log smooth, a lifting of $f_{T/S}$ to $\tilde{T}$ exist locally on $T$, so ${\cal L}_{{\cal X/S}, \tilde{T}}$ has nonempty stalks.
If we define, for $g_{1}, g_{2}\in {\cal L}_{{\cal X/S}, \tilde{T}}$, 
the subtraction $g_{1}-g_{2}\in {\cal H}om_{{\cal O}_{T}}(f^{*}_{T/S}\Omega^{1}_{X'/S}, {\cal O}_{T})$ by
\begin{eqnarray}
f^{*}_{T/S}(dx)\mapsto a \text{ where }p\tilde{a}=g_{2}^{*}(\tilde{x})-g_{1}^{*}(\tilde{x})\\
f^{*}_{T/S}(d\log m)\mapsto b \text{ where }g_{1}^{*}(\tilde{m})(1+p\tilde{b})=g_{2}^{*}(\tilde{m}),
\end{eqnarray}
where $\tilde{a}\in {\cal O}_{\tilde{T}}$ (resp. $\tilde{a}\in {\cal O}_{\tilde{T}}$) is a lift of $a\in {\cal O}_{T}$ (resp. $b\in {\cal O}_{T}$) and $g^{*}_{i}$ is the underlying morphism of the structure sheaf or the log structure,
then we can check that ${\cal L}_{{\cal X/S}, \tilde{T}}$ is a torsor over $f^{*}_{T/S}{\cal T}_{X'/S}$ by this subtraction.
\end{proof}
Therefore the family $\{{\cal L}^{(m)}_{{\cal X/S}, \tilde{T}}\}$ defines a log $m$-crystal of torsor over $F_{X/S}^{*}{\cal T}_{X'/S}$.
The following lemma shows that this crystal of torsor on ${\rm{CRIS}}_{{\rm Int},f}^{(m)}(X/\tilde{S})$ defines a log $m$ crystal of torsor on ${\rm{CRIS}}_{{\rm Int}}^{(m)}(X/S)$.
\begin{lemm}\label{Lem7}
The natural inclusion of sites ${\rm{CRIS}}_{{\rm Int},f}^{(m)}(X/\tilde{S})\to {\rm{CRIS}}_{\rm Int}^{(m)}(X/\tilde{S})$ induces an equivalence of categories between the respective categories of log $m$-crystals of ${\cal O}^{(m)}_{X/\tilde{S}}$-modules.
The natural functor from the category of $p$-torsion log $m$-crystals ${\cal O}^{(m)}_{X/\tilde{S}}$-modules on ${\rm{CRIS}}_{\rm Int}^{(m)}(X/\tilde{S})$ to the category of log $m$-crystals ${\cal O}^{(m)}_{X/S}$-modules on ${\rm{CRIS}}_{\rm Int}^{(m)}(X/S)$ is also an equivalence of categories.
Furthermore the category of $p$-torsion log $m$-crystals of torsor over $F_{X/S}^{*}{\cal T}_{X'/S}$ on ${\rm{CRIS}}_{\rm Int}^{(m)}(X/\tilde{S})$ is also equivalent to the category of log $m$-crystals of torsor over $F_{X/S}^{*}{\cal T}_{X'/S}$
on ${\rm{CRIS}}_{\rm Int}^{(m)}(X/S)$.
\end{lemm}
\begin{proof}
Since the question is \'etale local on $X$, we may assume that there exists a lifting $\tilde{X}/\tilde{S}$. Then, by Proposition \ref{Proposition8}, both categories can be identified with the category of ${\cal O}_{\tilde{X}}$-modules equipped with a log hyper $m$-PD stratification.
\end{proof}

Second, we construct the ${\cal O}_{X}$-module ${\cal K}^{(m)}_{\cal X/S}$ with natural actions of $\hat{\Gamma}.(F_{X/S}^{*}{\cal T}_{X'/S})$ and ${\cal D}^{(m)}_{X/S}$.
Let $(\tilde{U}, \tilde{T}, \tilde{J}, \tilde{\delta})$ be an object of ${\rm{CRIS}}_{{\rm Int},f}^{(m)}(X/\tilde{S})$. Let $T$ be a closed subscheme of $\tilde{T}$ defined by $p$. 
In Lemma \ref{Lem8}, we saw that ${\cal L}^{(m)}_{{\cal X/S}, \tilde{T}}$ forms a torsor over $f^{*}_{T/S}{\cal T}_{X'/S}$.
For a local section $a\in {\cal L}^{(m)}_{{\cal X/S},\tilde{T}}$ and $\varphi \in {\cal H}om({\cal L}^{(m)}_{{\cal X/S}, \tilde{T}}, {\cal O}_{T})$, we define the map $\varphi_{a}:f^{*}_{T/S}{\cal T}_{X'/S}\to {\cal O}_{T}$ by $D\mapsto \varphi(a+D)-\varphi(a)$.
Let ${\cal E}_{{\cal X/S}, \tilde{T}}^{(m)}$ denote the subsheaf of ${\cal H}om({\cal L}^{(m)}_{{\cal X/S}, \tilde{T}}, {\cal O}_{T})$ consisting of morphisms $\varphi:{\cal L}^{(m)}_{{\cal X/S}, \tilde{T}}\to {\cal O}_{T}$ such that, for any local section $a$ of ${\cal L}^{(m)}_{\cal X/S}$, the map $\varphi_{a}$ is ${\cal O}_{T}$-linear.
Note that the map $\varphi_{a}$ is independent of the choice of $a$.
We put $\omega_{\varphi}:=\varphi_{a}$.
Then, we have a diagram
	\begin{equation*}
	(\sharp)_{\tilde{T}}\,\,\,\,\,0\longrightarrow {\cal O}_{T}\longrightarrow {\cal E}_{{\cal X/S}, \tilde{T}}^{(m)}\xrightarrow[\varphi\mapsto \omega_{\varphi}] \ f^{*}_{T/S}{\Omega}^{1}_{X'/S}\longrightarrow 0,
	\end{equation*}
	where the map ${\cal O}_{T} \to {\cal E}_{{\cal X/S}, \tilde{T}}^{(m)}$ sends $b\in {\cal O}_{T}$ to the constant function.
This is a locally split exact sequence.
In fact, given a local section $a\in {\cal L}^{(m)}_{{\cal X/S}, \tilde{T}}$, we define the map $\sigma_{a}: f^{*}_{T/S}{\Omega}^{1}_{X'/S}\to {\cal E}_{{\cal X/S}, \tilde{T}}^{(m)}$ by $\omega\mapsto [b\mapsto \langle \omega,b-a\rangle]$.
Then this map gives the section of ${\cal E}_{{\cal X/S}, \tilde{T}}^{(m)}\to f^{*}_{T/S}{\Omega}^{1}_{X'/S}$.
The injection ${\cal O}_{T}\hookrightarrow {\cal E}_{{\cal X/S}, \tilde{T}}^{(m)}$ induces an injection $S^{n}({\cal E}_{{\cal X/S}, \tilde{T}}^{(m)})\hookrightarrow S^{n+1}({\cal E}_{{\cal X/S}, \tilde{T}}^{(m)})$ for each natural number $n$.
We define the ${\cal O}_{T}$-algebra ${\cal K}^{(m)}_{{\cal X/S}, \tilde{T}}$ by the inductive limit $\displaystyle \lim_{\to} S^{n}({\cal E}_{{\cal X/S}, \tilde{T}}^{(m)})$.
Next, let us define an action of  $f^{*}_{T/S}{\cal T}_{X'/S}$ on ${\cal K}^{(m)}_{{\cal X/S}, \tilde{T}}$.
When the exact sequence $(\sharp)_{\tilde{T}}$ splits, we have ${\cal E}_{{\cal X/S}, \tilde{T}}^{(m)}\simeq {\cal O}_{T}\times f^{*}_{T/S}{\Omega}^{1}_{X'/S}$.
This isomorphism induces an isomorphism of ${\cal O}_{T}$-algebras ${\cal K}^{(m)}_{{\cal X/S}, \tilde{T}}\simeq S^{\cdot}(f^{*}_{T/S}{\Omega}^{1}_{X'/S})$.
Now we define an action of $D\in f^{*}_{T/S}{\cal T}_{X'/S}$ on ${\cal K}^{(m)}_{{\cal X/S}, \tilde{T}}$ by the composition
\begin{equation*}
{\cal K}^{(m)}_{{\cal X/S}, \tilde{T}}\simeq S^{\cdot}(f^{*}_{T/S}{\Omega}^{1}_{X'/S})\xrightarrow{D} S^{\cdot}(f^{*}_{T/S}{\Omega}^{1}_{X'/S})\simeq{\cal K}^{(m)}_{{\cal X/S}, \tilde{T}},
\end{equation*}
where the map $D$ is defined as derivation.
Furthermore this action induces the action of $\hat{\Gamma}.(f^{*}_{T/S}{\cal T}_{X'/S})$ on ${\cal K}^{(m)}_{{\cal X/S}, \tilde{T}}$.
Since $D\in f^{*}_{T/S}{\cal T}_{X'/S}$ acts on ${\cal K}^{(m)}_{{\cal X/S}, \tilde{T}}$ as a derivation, the action of $D\in f^{*}_{T/S}{\cal T}_{X'/S}$ on ${\cal O}_{T}$ is zero.
Thus the action is independent of the choice of a splitting ${\cal E}_{{\cal X/S}, \tilde{T}}^{(m)}\simeq {\cal O}_{T}\times f^{*}_{T/S}{\Omega}^{1}_{X'/S}$ of the exact sequence $(\sharp)_{\tilde{T}}$.
Now the family $\left\{f^{*}_{T/S}{\cal T}_{X'/S}\right\}$ (resp. $\left\{f^{*}_{T/S}{\Omega}^{1}_{X'/S}\right\}$, $\left\{{\cal E}_{{\cal X/S}, \tilde{T}}^{(m)} \right\}$ and $\left\{{\cal K}^{(m)}_{{\cal X/S}, \tilde{T}}\right\}$) defines a $p$-torsion log $m$-crystal of ${\cal O}_{X/\tilde{S}}^{(m)}$-modules on ${\rm{CRIS}}^{(m)}_{{\rm Int}, f}(X/\tilde{S})$.
Hence, by Lemma \ref{Lem8}, we obtain corresponding log $m$-crystals of ${\cal O}_{X/S}^{(m)}$-modules $F_{X/S}^{*}{\cal T}_{X'/S}$, $F_{X/S}^{*}\Omega^{1}_{X'/S}$, ${\cal E}_{\cal X/S}^{(m)}$ and ${\cal K}^{(m)}_{\cal X/S}$ respectively.
Each of the log $m$-crystals defines the sheaf on $X$. We denote it by the same symbol.
The family of exact sequences $\left\{ (\sharp)_{\tilde{T}}\right\}$ induces an exact sequence of ${\cal O}_{X}$-modules
\begin{equation*}
(\sharp)\,\,\,\,\,\,0\longrightarrow {\cal O}_{X}\longrightarrow {\cal E}_{\cal X/S}^{(m)}\longrightarrow F^{*}_{X/S}\Omega^{1}_{X'/S}\longrightarrow0.
\end{equation*}
Let us assume that  there exists a lifting $\tilde{X}\to \tilde{X'}$ of $X\to X'$ modulo $p^{2}$.
Then the exact sequence $(\sharp)$ coincides with $(\sharp)_{\tilde{X}}$ and splits.
The splitting defines an action of $\hat{\Gamma}_{.}(F_{X/S}^{*}{\cal T}_{X'/S})$ on ${\cal K}^{(m)}_{\cal X/S}$.
Since this action is independent of the choice of a splitting, we have an action $({\rm A})$ of $\hat{\Gamma}_{.}(F_{X/S}^{*}{\cal T}_{X'/S})$ on ${\cal K}^{(m)}_{\cal X/S}$ globally on $X$.
On the other hand, the structure of log $m$-crystal of ${\cal O}_{X/S}^{(m)}$-modules on ${\cal K}^{(m)}_{\cal X/S}$ gives the action of ${\cal D}_{X/S}^{(m)}$ on ${\cal K}^{(m)}_{\cal X/S}$.
Thus we also have an action $({\rm B})$ of ${\cal T}_{X'/S}$ on ${\cal K}^{(m)}_{\cal X/S}$ via the $p^{m+1}$-curvature map $\beta:{\cal T}_{X'/S}\to {\cal D}_{X/S}^{(m)}$.
Let us show the following lemma.
\begin{lemm}\label{lem5}
The two ${\cal T}_{X'/S}$-actions ${\rm(A)}$ and ${\rm(B)}$ on ${\cal K}^{(m)}_{\cal X/S}$ are equal.
\end{lemm}
\begin{proof}
Since the assertion is \'etale local, we may assume that there exists a lifting $\tilde{F}:\tilde{X}\to \tilde{X'}$ of $X\to X'$ modulo $p^{2}$ and a logarithmic system of coordinates $\{\tilde{m}_{i}\}_{i}$ of $\tilde{X}\to \tilde{S}$.
Denote by $\{m_{i}\}_{i}$ the image of $\{\tilde{m}_{i}\}_{i}$ in ${\cal M}_{X}^{gp}$.
Then $\left\{d\log \pi^{*}(m_{i})\right\}_{i}$ forms a basis for $\Omega^{1}_{X'/S}$.
We also denote by $\tilde{P}$ the log $m$-PD envelope of the diagonal $\tilde{X}\rightarrow \tilde{X}\times_{\tilde{S}}\tilde{X}$ and by $P$ the log $m$-PD envelope of $X\rightarrow X\times_{S}X$.
Let $\left\{\underline{\tilde{\eta}}^{\{\underline{k}\}}\right\}$ denote a basis of ${\cal O}_{\tilde{P}}$ induced by $\{\tilde{m}_{i}\}_{i}$.
We denote the image of $\left\{\underline{\tilde{\eta}}^{\{\underline{k}\}}\right\}$ in ${\cal O}_{P}$ by $\left\{\underline{\eta}^{\{\underline{k}\}}\right\}$.
Let $\{\xi^{'}_{i}\}_{i}$ denote the dual basis of $\{d\log \pi^{*}(m_{i})\}_{i}$.
By the construction of ${\cal K}^{(m)}_{\cal X/S}$  (and the actions of ${\cal T}_{X'/S}$ on it), it suffices to show that the two actions ${\rm(A)}$ and ${\rm(B)}$ agree on ${\cal E}_{\cal X/S}^{(m)}$.
By the existence of a lifting $\tilde{F}:\tilde{X}\to \tilde{X'}$, the exact sequence $(\sharp)$ splits via $\sigma_{\tilde{F}}:F^{*}_{X/S}\Omega^{1}_{X'/S}\hookrightarrow {\cal E}_{\cal X/S}^{(m)}$.
Hence we have ${\cal E}_{\cal X/S}^{(m)}\simeq {\cal O}_{X}\times F^{*}_{X/S}\Omega^{1}_{X'/S}$.
The two actions ${\rm(A)}$ and ${\rm(B)}$ are zero on ${\cal O}_{X}$.
Thus, we need to calculate the action on $F^{*}_{X/S}\Omega^{1}_{X'/S}$.
First, the action ${\rm(A)}$ of $\xi^{'}_{i} \in {\cal T}_{X'/S}$ is given by $F^{*}_{X/S}\Omega^{1}_{X'/S}\to {\cal O}_{X};\,d\log \pi^{*}(m_{j})\mapsto \delta_{ij}$.
On the other hand, the action ${\rm (B)}$ of $\xi^{'}_{i} \in {\cal T}_{X'/S}$ is given by
\begin{eqnarray*}
F^{*}_{X/S}\Omega^{1}_{X'/S}\xrightarrow{\sigma_{\tilde{F}}} {\cal E}_{\cal X/S}^{(m)} \xrightarrow{p^{*}_{1}} {\cal O}_{\tilde{P}}\otimes {\cal E}_{\cal X/S}^{(m)}\xrightarrow{\cong} {\cal E}_{\cal X/S}^{(m)}\otimes {\cal O}_{\tilde{P}}\,\,\,\,\,\,\,\,\,\,\,\,\,\,\,\,\\ 
\,\,\,\,\,\,=\bigoplus_{\underline{k}} {\cal E}_{\cal X/S}^{(m)}\underline{\eta}^{\{\underline{k}\}}\to {\cal E}_{\cal X/S}^{(m)}\eta_{i}^{\{p^{m+1}\}}={\cal E}_{\cal X/S}^{(m)},
\end{eqnarray*}
where the isomorphism ${\cal O}_{\tilde{P}}\otimes {\cal E}_{\cal X/S}^{(m)}\xrightarrow{\cong} {\cal E}_{\cal X/S}^{(m)}\otimes {\cal O}_{\tilde{P}}$ is the HPD-stratification associated to the log $m$-crystal structure on ${\cal E}_{\cal X/S}^{(m)}$, and the last map is the natural projection.
Let us calculate this action explicitly.
First, $\sigma_{\tilde{F}}\left(d\log \pi^{*}\left(m_{i}\right)\right)\in {\cal E}_{\cal X/S}^{(m)}$ sends $\tilde{F}'\in {\cal L}^{(m)}_{{\cal X/S}, \tilde{X}}$ to $a$,
where $a$ is a the section of ${\cal O}_{X}$ satisfying $\tilde{F}^{*}(\pi^{*}(m_{j})^{\sim})(1+p\tilde{a})=\tilde{F'}(\pi^{*}(m_{j})^{\sim})$.
(We denote by $\sim$ a lifting of section.)
Next, if we take the pullback of $\sigma_{\tilde{F}}(d\log \pi^{*}(m_{i}))$ by $p_{1}^{*}$, we
obtain the map $\sigma_{p_{1}\circ \tilde{F}}(d\log \pi^{*}(m_{i})): {\cal L}^{(m)}_{{\cal X/S}, \tilde{P}}\to {\cal O}_{P}$.
The image of $\tilde{F}\in {\cal L}^{(m)}_{{\cal X/S}, \tilde{X}}$ by
\begin{equation*}
{\cal L}^{(m)}_{{\cal X/S}, \tilde{X}}\xrightarrow{-\circ p_{0}} {\cal L}^{(m)}_{{\cal X/S}, \tilde{P}}\xrightarrow{\sigma_{p_{1}\circ \tilde{F}}(d\log \pi^{*}(m_{i}))}{\cal O}_{P} 
\end{equation*}
is the element $b\in {\cal O}_{P}=\bigoplus {\cal O}_{X}\underline{\eta}^{\{\underline{k}\}}$ satisfying
$(\tilde{F}\circ p_{1})^{*}(\pi^{*}(m_{j})^{\sim})(1+p\tilde{b})=(\tilde{F'}\circ p_{0})^{*}(\pi^{*}(m_{j})^{\sim})$.
We have to show that the $\eta_{i}^{\{p^{m+1}\}}$-component of $b$ is $\delta_{ij}$.
We take $c, d\in {\cal O}_{P}$ satisfying
\begin{eqnarray}
(\tilde{F}\circ p_{1})^{*}(\pi^{*}(m_{j})^{\sim})(1+p\tilde{c})=(\tilde{F}\circ p_{0})^{*}(\pi^{*}(m_{j})^{\sim}),\label{equ5}\\
(\tilde{F}\circ p_{0})^{*}(\pi^{*}(m_{j})^{\sim})(1+p\tilde{d})=(\tilde{F'}\circ p_{0})^{*}(\pi^{*}(m_{j})^{\sim}).\nonumber
\end{eqnarray}
Then, we have $b=c+d$.
By the definition of $d$, we may assume that $\tilde{d}$ is of the form $\tilde{d}=p_{0}^{*}(\tilde{d}')$.
Then $d$ is in ${\cal O}_{X}\cdot 1\subset{\cal O}_{P}$, and we see that the $\eta_{i}^{\{p^{m+1}\}}$-component of $d$ is zero.
On the other hand, From (\ref{equ5}), we have 
\begin{equation}\label{equ6}
p_{1}^{*}\left(\tilde{F}^{*}(\pi^{*}(m_{j})^{\sim})\right)(1+p\tilde{c})=p_{0}^{*}(\tilde{F}^{*}(\pi^{*}(m_{j})^{\sim})
\end{equation}
Since $\tilde{F}^{*}(\pi^{*}(m_{j})^{\sim})$ is a lift of $F^{*}\pi^{*}(m_{j})=m_{j}^{p^{m+1}}$, there exists $\tilde{e}\in {\cal O}_{\tilde{X}}$ satisfying 
\begin{equation}\label{equ7}
\tilde{F}^{*}\left(\pi^{*}(m_{j})^{\sim}\right)=\tilde{m_{j}}^{p^{m+1}}\cdot (1+p\tilde{e}).
\end{equation}
From (\ref{equ6}) and (\ref{equ7}), we have
\begin{equation*}
\left(1+p\left(p_{1}^{*}(\tilde{e})-p_{0}^{*}(\tilde{e})+\tilde{c}\right)\right)(1+\tilde{\eta}_{j})^{p^{m+1}}=1
\end{equation*}
Now, since $(1+\tilde{\eta}_{j})^{p^{m+1}}=1+p\left(\text{the terms of }\tilde{\eta}_{j} \text{ of degree}<p^{m+1}\right)+p!\tilde{\eta}_{j}^{\{p^{m+1}\}}$,
we see that
\begin{equation*} 
\text{the }{\eta}_{i}^{\{p^{m+1}\}}\text{-component of }(p_{1}^{*}(e)-p_{0}^{*}(e)+c)=\delta_{ij}.
\end{equation*}
Finally, since $p_{1}^{*}(e)-p_{0}^{*}(e)=\sum_{\underline{k}>0}\underline{\partial}_{<\underline{k}>}(e)\underline{\eta}^{\{\underline{k}\}}$,
\begin{equation*}
\text{the }{\eta}_{j}^{\{p^{m+1}\}}\text{-component of }(p_{1}^{*}(e)-p_{0}^{*}(e))=\underline{\partial}_{<p^{m+1}\underline{\varepsilon_{j}}>}(e)=0.
\end{equation*}
Hence the ${\eta}_{j}^{\{p^{m+1}\}}$-component of $c$ is $\delta_{ij}$. We finish the proof.
\end{proof}
Now, we are ready to construct the log global Cartier transform of higher level.
Let $\check{{\cal K}}_{\cal X/S}^{(m)}$ be the dual of ${\cal K}_{\cal X/S}^{(m)}$ as a left ${\cal D}_{X/S}^{(m)}$-module.
We consider the tensor product $\check{{\cal K}}_{\cal X/S}^{(m), {\cal A}}:={\cal A}_{X}^{gp}\otimes_{{\cal O}_{X}}\check{{\cal K}}_{\cal X/S}^{(m)}$ as a left ${\cal D}_{X/S}^{(m)}$-module.
Let us show that $\check{{\cal K}}_{\cal X/S}^{(m), {\cal A}}$ is a splitting module of $\tilde {{\cal D}}_{X/S}^{(m)}$ over ${\cal O_{G}^{B}}$.
Since the action of ${\cal D}_{X/S}^{(m)}$ on $\check{{\cal K}}_{\cal X/S}^{(m), {\cal A}}$ is admissible by Remark \ref{lemma26},
the action of ${\cal D}_{X/S}^{(m)}$ extends to an $\cal J$-indexed $\tilde {{\cal D}}_{X/S}^{(m)}$-module structure on $\check{{\cal K}}_{\cal X/S}^{(m), {\cal A}}$.
On the other hand, we also have the ${\cal O}_{\cal G}$-action on $\check{{\cal K}}_{\cal X/S}^{(m)}$ induced from the action $({\rm A})$.
By Lemma \ref{lem5}, these two actions extend to the structure of left ${\cal A}_{X}^{gp}\otimes_{{\cal O}_{X}}{\cal D}_{X/S}^{(m)}\otimes_{S^{\cdot}{\cal T}_{X'/S}}{\cal O}_{\cal G}\simeq \tilde {{\cal D}}_{X}^{(m), \gamma}$-module on $\check{{\cal K}}_{\cal X/S}^{(m), {\cal A}}$.
Since we locally have an isomorphism
\begin{equation*}
\check{{\cal K}}_{\cal X/S}^{(m)}:={\cal H}om_{{\cal O}_{X}}({\cal K}_{\cal X/S}^{(m)}, {\cal O}_{X})\simeq \hat{\Gamma.}(F^{*}_{X/S}{\cal T}_{X'/S})\simeq F^{*}_{X/S}{\cal O_{G}},
\end{equation*}
$\check{{\cal K}}_{\cal X/S}^{(m), {\cal A}}$ is a locally free ${\cal O_{G}^{A}}$-module of rank $1$.
Moreover, by Proposition \ref{proposition2}, ${\cal O_{G}^{A}}\simeq {\cal O_{G}^{B}}\otimes_{{\cal B}_{X/S}^{(m+1)}}{\cal A}_{X}^{gp}$ is a locally free module of rank $p^{(m+1)r}$ over ${\cal O_{G}^{B}}$.
Hence $\check{{\cal K}}_{\cal X/S}^{(m), {\cal A}}$ is a splitting module for $\tilde {{\cal D}}_{X/S}^{(m)}$ over ${\cal O_{G}^{B}}$ by Proposition \ref{Proposition7}.
We thus obtain the following isomorphism of ${\cal O_{G}^{B}}$-algebras:
\begin{equation}\label{AA}
\tilde{\cal D}_{X/S}^{(m)}\otimes_{\tilde{\mathfrak Z}} {\cal O_{G}^{B}}\xrightarrow{\cong} {\cal E}nd_{\cal O_{G}^{B}}(\check{{\cal K}}_{\cal X/S}^{(m), {\cal A}}).
\end{equation}
\begin{rem}
In the case without log structure, Gros, Le Stum and Quiros obtained a similar isomorphism (\ref{AA}) in a different way.
See Subsection 6.4 of \cite{GLQ}.
\end{rem}
Using Proposition \ref{Proposition5}, we obtain the following theorem, which is the central result of this paper.
\begin{theo}\label{Theorem21}
 Let ${\cal X/S}=(X\to S, \tilde{X'}\to \tilde{S})$ be a log smooth integral morphism with a lifting $\tilde{X}\to \tilde{S}$ modulo $p^{2}$. Then the functor
 	\begin{equation*}
	C_{\cal X/S}: {{\rm MIC}^{\cal A, J}_{\rm PD}(X/S)}\to {{\rm HIG}^{\cal B, J}_{\rm PD}(X'/S)}, \,\,\,E\mapsto {\cal H}om_{\tilde {{\cal D}}_{X}^{(m), \gamma}}(\check{{\cal K}}_{\cal X/S}^{(m), {\cal A}}, E)
	\end{equation*}
is an equivalence of categories. The quasi-inverse of $C_{\cal X/S}$ is given by
	\begin{equation*}
	C_{\cal X/S}^{-1}:{{\rm HIG}^{\cal B, J}_{\rm PD}(X'/S)}\to {{\rm MIC}^{\cal A, J}_{\rm PD}(X/S)}, \,\,\,E'\mapsto \check{{\cal K}}_{\cal X/S}^{(m), {\cal A}}\otimes_{\cal O_{G}^{B}}E'.
	\end{equation*}
\end{theo}
Let us give two corollaries.
In the case ${\cal J}={\cal I}_{X}^{gp}$ with the standard action, the category of $\cal J$-indexed ${\cal A}_{X}^{gp}$-modules is equivalent to the category of ${\cal O}_{X}$-modules 
with quasi-inverse $E\mapsto {\cal A}_{X}^{gp}\otimes_{{\cal O}_{X}}E$ (see p.22 of \cite{S}).

We put ${\cal D}_{X/S}^{(m), \gamma}:={\cal D}_{X/S}^{(m)}\otimes_{\mathfrak Z} {\cal O_{G}}$, and denote by ${\rm MIC}_{\rm PD}(X/S)$ the category of ${\cal D}_{X/S}^{(m), \gamma}$-modules.
We thus obtain the following corollary.
\begin{cor}
The functor
 	\begin{equation*}
	C_{\cal X/S}: {\rm MIC}_{\rm PD} (X/S)\to {\rm HIG}^{{\cal B, I}^{gp}_{X}}_{\rm PD}(X'/S), \,\,\,E\mapsto {\cal H}om_{\tilde {{\cal D}}_{X}^{(m), \gamma}}(\check{{\cal K}}_{\cal X/S}^{(m), {\cal A}}, {\cal A}_{X}^{gp}\otimes_{{\cal O}_{X}}E)
	\end{equation*}
is an equivalence of categories.
\end{cor}
Next, we consider the case without log structure. Let $X\to S$ be a smooth morphism of schemes.
We denote by ${\rm HIG}_{\rm PD}(X'/S)$ the category of ${\cal O_{G}}$-modules.
\begin{cor}
Let ${\cal X/S}=(X\to S, \tilde{X'}\to \tilde{S})$ be a smooth morphism of schemes with a lifting $\tilde{X}\to \tilde{S}$ modulo $p^{2}$. Then the functor
 	\begin{equation*}
	C_{\cal X/S}: {{\rm MIC}_{\rm PD}(X/S)}\to {{\rm HIG}_{\rm PD}(X'/S)}, \,\,\,E\mapsto {\cal H}om_{{\cal D}_{X/S}^{(m), \gamma}}(\check{{\cal K}}_{\cal X/S}^{(m)}, E)
	\end{equation*}
is an equivalence of categories. 
\end{cor}

\section{Compatibility}
In this section, we discuss the compatibility between the log global Cartier transform and the log Frobenius descent.
Throughout this section, we fix a sheaf of ${\cal I}^{gp}_{X/S}$-sets $\cal J$ and denote by $F$ (resp. $F_{X/S}$) the $m$-th relative Frobenius morphism (resp. the $(m+1)$-st relative Frobenius morphism) by abuse of notation.

\subsection{Frobenius descent}
Let us briefly recall Montagnon's log Frobenius descent.
Let $\cal E$ be a $\cal J$-indexed left $\tilde {{\cal D}}_{X/S}^{(m)}$-module.
We consider ${\cal J}\xrightarrow{\cong} {\cal H}om_{*}(*, {\cal J})$-indexed sheaf defined by
\begin{equation}
{\mathbb F}_{\cal J}({\cal E}):={\cal H}om_{{\cal D}^{(m)}_{X/S}}(F^{*}{\cal D}_{X^{(m)}/S}^{(0)}, {\cal E}).
\end{equation}
We give a ${\cal B}_{X/S}^{(m)}$-action on ${\mathbb F}_{\cal J}({\cal E})$ by the scalar restriction of the ${\cal A}_{X}^{gp}$-action on $\cal E$ and 
give a left ${\cal D}_{X^{(m)}/S}^{(0)}$-action on ${\mathbb F}_{\cal J}(\cal E)$ by the right multiplication of ${\cal D}_{X^{(m)}/S}^{(0)}$ on $F^{*}{\cal D}_{X^{(m)}/S}^{(0)}$.
Then one can see that these two actions naturally extend to a $\cal J$-indexed left $\tilde{{\cal D}}_{X^{(m)}/S}^{(0)}$-module structure on ${\mathbb F}_{\cal J}(\cal E)$ (for the proof, see Subsection 4.2.1 of \cite{M}).
We thus obtain a functor 
\begin{equation*}
{\mathbb F}_{\cal J}:
 \left(
\begin{array}{c}
\text{The category of} \\
{\cal J}\text{-indexed left } \tilde {{\cal D}}_{X/S}^{(m)} \text{-modules}
\end{array}
\right)
\to 
\left(
\begin{array}{c}
\text{The category of} \\
{\cal J}\text{-indexed left } \tilde {{\cal D}}_{X^{(m)}/S}^{(0)} \text{-modules}
\end{array}
\right).
\end{equation*}
Conversely, let $\cal F$ be a $\cal J$-indexed left $\tilde {{\cal D}}_{X^{(m)}/S}^{(0)}$-module.
We consider ${\cal J}\xrightarrow{\cong} {\cal I}_{X}^{gp}\otimes_{{\cal I}_{X}^{gp}} {\cal J}$-indexed sheaf defined by
	\begin{equation*}
	{\mathbb G}_{\cal J}({\cal F}):={\cal A}_{X}^{gp}\otimes_{F^{*}{\cal B}_{X/S}^{(m)}} F^{*}{\cal F} .
	\end{equation*}
Here $F^{*}{\cal B}_{X/S}^{(m)}$ ia the ${\cal I}_{X}^{gp}\xrightarrow{\cong} {*}\otimes_{{*}}{\cal I}_{X}^{gp}$-indexed ${\cal O}_{X}$-algebra ${\cal O}_{X}\otimes_{F^{-1}{\cal O}_{X^{(m)}}}F^{-1}{\cal B}_{X/S}^{(m)}$ and
$F^{*}{\cal F}$ is defined in a similar manner.
${\mathbb G}_{\cal J}({\cal F})$ has an ${\cal A}_{X}^{gp}$-action induced from the first component.
Let us endow ${\mathbb G}_{\cal J}({\cal F})$ with a left ${\cal D}_{X/S}^{(m)}$-action in the following way.
Let $\{\varepsilon_{{\cal A}, n}\}$ be the log $m$-stratification of ${\cal A}_{X}^{gp}$ associated to its ${\cal D}_{X/S}^{(m)}$-action.
Let $\{\varepsilon_{{\cal F}, n}\}$ (resp. $\{\varepsilon_{{\cal B}, n}\}$) be the log $0$-stratification associated to the ${\cal D}^{(0)}_{X^{(m)}/S}$-action on $\cal F$ (resp. ${\cal B}_{X/S}^{(m)}$).
Let $\Psi$ be the composition
\begin{equation*}
p_{1}^{n*}{\cal A}_{X}^{gp}\times p_{1}^{n*}F^{*}{\cal F}\xrightarrow{\varepsilon_{{\cal A}, n}\times {\Phi}^{*}\varepsilon_{{\cal F}, n}} p_{0}^{n*}{\cal A}_{X}^{gp}\times p_{0}^{n*}F^{*}{\cal F}\to
p_{0}^{n*}{\cal A}_{X}^{gp}\otimes_{p_{0}^{n*}F^{*}{\cal B}_{X/S}^{(m)}} p_{0}^{n*}F^{*}{\cal F},
\end{equation*}
where $p_{0}^{n}$ and $p_{1}^{n}$ denote the first and second projection $P_{X/S}^{n}\to X$, $\Phi$ is defined in the beginning of Subsection 4.1.2 and the second map is the natural projection.

\begin{lemm}\label{lemma11}
$\Psi$ is a biadditive $p_{1}^{n*}F^{*}{\cal B}_{X/S}^{(m)}$-balanced map.
\end{lemm}
\begin{proof}
We prove $\Psi(ab, f)=\Psi(a, bf)$ for any local section $a\in p_{1}^{n*}{\cal A}_{X}^{gp}$, $b\in p_{1}^{n*}F^{*}{\cal B}_{X/S}^{(m)}$ and $f\in p_{1}^{n*}F^{*}{\cal F}$.
Since $\varepsilon_{{\cal A}, n}$ is a morphism of ${\cal I}_{X}^{gp}$-indexed ${\cal P}_{X/S, (m)}^{n}$-algebras, $\Psi(ab, f)$ is equal to
$\varepsilon_{{\cal A}, n}(a) \varepsilon_{{\cal A}, n}(b)\otimes {\Phi}^{*} \varepsilon_{{\cal F}, n}(f)$.
On the other hand, since ${\cal D}_{X^{(m)}/S}^{(0)}$-module structure on $\cal F$ is admissible (see Remark \ref{lemma26}), $\Psi(a, bf)$ is equal to
$\varepsilon_{{\cal A}, n}(a)\otimes {\Phi}^{*} \varepsilon_{{\cal B}, n}(b){\Phi}^{*} \varepsilon_{{\cal F}, n}(f)$.
So it suffices to show that $\varepsilon_{{\cal A}, n}(b)={\Phi}^{*} \varepsilon_{{\cal B}, n}(b)$.
This claim follows from Lemma \ref{lemma22} below.
\end{proof}
\begin{lemm}\label{lemma22}
The natural homomorphism $F^{*}{\cal B}_{X/S}^{(m)}\to {\cal A}_{X}^{gp}$ of ${\cal I}_{X}^{gp}$-indexed ${\cal O}_{X}$-algebras  is a morphism of ${\cal I}_{X}^{gp}$-indexed ${\cal D}_{X/S}^{(m)}$-modules.
\end{lemm}
\begin{proof}
We may work \'etale locally on $X$.
Let \{$\underline{\partial}_{<\underline{k}>}$\} and \{$\underline{\partial}'_{<\underline{k}>}$\} be as in Remark \ref{Remark20}.
Let $\psi$ denote the natural homomorphism of ${\cal I}_{X}^{gp}$-indexed ${\cal O}_{X}$-algebras defined by
\begin{equation*}
F^{*}{\cal B}_{X/S}^{(m)}={\cal O}_{X}\otimes_{{\cal O}_{X^{(m)}}}{\cal H}om_{{\cal D}_{X/S}^{(m)}}(F^{*}{\cal D}_{X^{(m)}/S}^{(0)}, {\cal A}_{X}^{gp}) \to {\cal A}_{X}^{gp},\,\,\,\, a\otimes f \mapsto af(1\otimes 1).
\end{equation*}
We shall show the equality $\psi\bigl(\underline{\partial}_{<\underline{k}>}.(a\otimes f)\bigr)=\underline{\partial}_{<\underline{k}>}.\psi\bigl(a\otimes f\bigr)$.
By the formulas in Remark \ref{Remark20}, the left hand side is calculated by  
\begin{eqnarray*}
\psi\left(\underline{\partial}_{<\underline{k}>}.\left(a\otimes f\right)\right)&=&
\psi\left(\sum_{\underline{i}\leq \underline{k}}\left\{
\begin{array}{c}
\underline{k}\\
\underline{i}
\end{array}
\right\} \underline{\partial}_{<\underline{k}-\underline{i}>}\left(a\right)\underline{\partial}_{<\underline{i}>}.\left(1\otimes f\right)\right)\\
&=&
\psi\Bigl(\sum_{p^{m}\underline{j}\leq \underline{k}}\left\{
\begin{array}{c}
\underline{k}\\
p^{m}\underline{j}
\end{array}
\right\} \underline{\partial}_{<\underline{k}-p^{m}\underline{j}>}(a)\otimes \underline{\partial}'_{<\underline{j}>}.f\Bigr)\\
&=&
\sum_{p^{m}\underline{j}\leq \underline{k}}\left\{
\begin{array}{c}
\underline{k}\\
p^{m}\underline{j}
\end{array}
\right\} \underline{\partial}_{<\underline{k}-p^{m}\underline{j}>}(a)f(1\otimes \underline{\partial}'_{<\underline{j}>}).
\end{eqnarray*}
Similarly, by the formulas in Remark \ref{Remark20}, the right hand side is calculated by
\begin{eqnarray*}
\underline{\partial}_{<\underline{k}>}. af(1\otimes 1)&=&
\sum_{\underline{i}\leq \underline{k}}\left\{
\begin{array}{c}
\underline{k}\\
\underline{i}
\end{array}
\right\} \underline{\partial}_{<\underline{k}-\underline{i}>}(a)\underline{\partial}_{<\underline{i}>}.f(1\otimes 1)\\
&=&
\sum_{\underline{i}\leq \underline{k}}\left\{
\begin{array}{c}
\underline{k}\\
\underline{i}
\end{array}
\right\} \underline{\partial}_{<\underline{k}-\underline{i}>}(a)f\left(\underline{\partial}_{<\underline{i}>}.\left(1\otimes 1\right)\right)\\
&=&
\sum_{p^{m}\underline{j}\leq \underline{k}}\left\{
\begin{array}{c}
\underline{k}\\
p^{m}\underline{j}
\end{array}
\right\} \underline{\partial}_{<\underline{k}-p^{m}\underline{j}>}(a)f(1\otimes \underline{\partial}'_{<\underline{j}>}).
\end{eqnarray*}
This finishes the proof.

\end{proof}
Thanks to Lemma \ref{lemma11} and the universal mapping property of a tensor product, we have 
\begin{equation*}
\varepsilon_{n}: p_{1}^{n*}{\cal A}_{X}^{gp}\otimes_{p_{1}^{n*}F^{*}{\cal B}_{X/S}^{(m)}} p_{1}^{n*}F^{*}{\cal F} \to p_{0}^{n*}{\cal A}_{X}^{gp}\otimes_{p_{0}^{n*}F^{*}{\cal B}_{X/S}^{(m)}} p_{0}^{n*}F^{*}{\cal F}.
\end{equation*}
One can obtain the inverse of $\varepsilon_{n}$ in a similar way.
So these maps \{$\varepsilon_{n}$\} are isomorphisms of $\cal J$-indexed ${\cal P}_{X/S, (m)}^{n}$-algebras. 
Furthermore \{$\varepsilon_{n}$\} satisfy the obvious cocycle conditions.
Therefore \{$\varepsilon_{n}$\} forms a log $m$-stratification on ${\cal A}_{X}^{gp}\otimes_{F^{*}{\cal B}_{X/S}^{(m)}} F^{*}{\cal F}$ and defines a $\cal J$-indexed ${\cal D}_{X/S}^{(m)}$-module structure.
\begin{lemm}\label{lemma23}
The $\cal J$-indexed ${\cal D}_{X/S}^{(m)}$-module structure on ${\cal A}_{X}^{gp}\otimes_{F^{*}{\cal B}_{X/S}^{(m)}} F^{*}{\cal F}$ is admissible.
\end{lemm} 
	\begin{proof}
	It suffices to show that the following diagram is commutative:
	\[\xymatrix{
	{p_{1}^{n*}({\cal A}_{X}^{gp}\otimes_{{\cal O}_{X}}{\cal A}_{X}^{gp}\otimes_{F^{*}{\cal B}_{X/S}^{(m)}} F^{*}{\cal F})} \ar[r] \ar[d] & {p_{0}^{n*}({\cal A}_{X}^{gp}\otimes_{{\cal O}_{X}}{\cal A}_{X}^{gp}\otimes_{F^{*}{\cal B}_{X/S}^{(m)}} F^{*}{\cal F})} \ar[d]\\
	{p_{1}^{n*}({\cal A}_{X}^{gp}\otimes_{F^{*}{\cal B}_{X/S}^{(m)}} F^{*}{\cal F})} \ar[r]^{\varepsilon_{n}} & {p_{0}^{n*}({\cal A}_{X}^{gp}\otimes_{F^{*}{\cal B}_{X/S}^{(m)}} F^{*}{\cal F})},}
	\]
	where the upper horizontal arrow is the  the log $m$-stratification $\varepsilon_{{\cal A}, n}\otimes \varepsilon_{n}$ and the vertical arrows are induced from the multiplication on ${\cal A}_{X}^{gp}$.
	This follows from the fact that the log $m$-stratification $\varepsilon_{{\cal A}, n}$ is a morphism of ${\cal I}_{X}^{gp}$-indexed ${\cal P}_{X/S, (m)}^{n}$-algebras.
	\end{proof}
By Lemma \ref{lemma23} and Remark \ref{lemma25} , we have a ${\cal J}$-indexed $\tilde{\cal D}_{X/S}^{(m)}$-module structure on ${\cal A}_{X}^{gp}\otimes_{F^{*}{\cal B}_{X/S}^{(m)}} F^{*}{\cal F}$ and obtain a functor
\begin{equation*}
{\mathbb G}_{\cal J}:
 \left(
\begin{array}{c}
\text{The category of} \\
{\cal J}\text{-indexed left } \tilde {{\cal D}}_{X^{(m)}/S}^{(0)} \text{-modules}

\end{array}
\right)
\to 
\left(
\begin{array}{c}
\text{The category of} \\
{\cal J}\text{-indexed left } \tilde {{\cal D}}_{X/S}^{(m)} \text{-modules}
\end{array}
\right).
\end{equation*}

	\begin{theo}
	${\mathbb F}_{\cal J}$ is an equivalence of categories with a quasi-inverse ${\mathbb G}_{\cal J}$.
	\end{theo}
\begin{proof}
See Th\'eor\`eme 4.2.1 of \cite{M}.
\end{proof}
\begin{rem}
In \cite{M}, Montagnon only treats the case of ${\cal J}={\cal I}^{gp}_{X}$, but the same argument as that in \cite{M} works for any sheaf of ${\cal I}^{gp}_{X}$-sets ${\cal J}$.
As is the same with the case about the definition of ${\cal B}^{(m)}_{X/S}$, she defines the Frobenius descent functor by ${\mathbb F}({\cal E}):={\cal H}om_{{\cal D}^{(m)}_{X/S, {\cal I}}}(F^{*}{\cal D}_{X^{(m)}/S, {\cal I}}^{(0)}, {\cal E})$ but this is a mistake.
\end{rem}

\subsection{Main Theorem}
Let us start with stating the main theorem.
Let ${\cal X/S}=(X\to S, \tilde{X'}\to \tilde{S})$ be a log smooth integral morphism of fine log schemes with a lifting $\tilde{X}\to \tilde{S}$ of $X'\to S$ modulo $p^{2}$.
We also denote the data $(X^{(m)}\to S, \tilde{X'}\to \tilde{S})$ by $\cal X/S$ by abuse of notation.
As in the Subsection 5.1, $\cal G$ denotes the nilpotent divided power envelope of the zero section of the cotangent bundle of $X'/S$, ${\cal O_{G}^{B}}$ denotes ${\cal B}_{X/S}^{(m+1)}\otimes_{{\cal O}_{X'}}\cal O_{G}$ and ${\rm HIG}^{\cal B, J}_{\rm PD}(X'/S)$ denotes the category of $\cal J$-indexed ${\cal O_{G}^{B}}$-modules.
We put $\tilde {{\cal D}}_{X^{(m)}/S}^{(0), \gamma} =\tilde{{\cal D}}_{X^{(m)}/S}^{(0)} \otimes_{\mathfrak Z'}{\cal O_{G}^{B}}$,
and denote by ${\rm MIC}^{\cal B, J}_{\rm PD}(X^{(m)}/S)$ the category of $\cal J$-indexed $\tilde {{\cal D}}_{X^{(m)}/S}^{(0), \gamma}$-modules.
By Theorem \ref{theorem30}, Proposition \ref{proposition14} and Lemma \ref{lem5}, we see that $\check{{\cal K}}_{\cal X/S}^{(0), {\cal B}}:={\cal B}_{X/S}^{(m)}\otimes_{{\cal O}_{X^{(m)}/S}}\check{{\cal K}}_{\cal X/S}^{(0)}$ is a splitting module of $\tilde{{\cal D}}_{X^{(m)}/S}^{(0)}$ over ${\cal O_{G}^{B}}$ and obtain an equivalence of categories
\begin{equation*}
C': {\rm MIC}^{{\cal B, J}}_{\rm PD} (X^{(m)}/S)\to {\rm HIG}^{\cal B, J}_{\rm PD}(X'/S), E\mapsto {\cal H}om_{\tilde {{\cal D}}_{X^{(m)}/S}^{(0), \gamma}}(\check{{\cal K}}_{\cal X/S}^{(0), {\cal B}}, E),
\end{equation*}
which is a variant of Theorem \ref{Theorem21} and Theorem 4.2 in \cite{S}.
Note that the equivalence $C'$ does not coincide with Theorem 4.2 in \cite{S} (see Remark \ref{Remark01}).

On the other hand, the Frobenius descent induces the following equivalence of categories.

	\begin{lemm}
	The Frobenius descent functor ${\mathbb F}_{\cal J}$ induces an equivalence of categories between ${\rm MIC}^{\cal A, J}_{\rm PD}(X/S)$ and ${\rm MIC}^{\cal B, J}_{\rm PD}(X^{(m)}/S)$.
	\end{lemm}
	\begin{proof}
	Let $\cal E$ be an object in ${\rm MIC}^{\cal A, J}_{\rm PD}(X/S)$.
	Then the $\hat{\Gamma}. {\cal T}_{X'/S}$-action on ${\mathbb F}_{\cal J}(\cal E)$ is naturally induced from $\hat{\Gamma}. {\cal T}_{X'/S}$-action on ${\cal E}$.
	We shall show that the following diagram is commutative:
	\[\xymatrix{
	{S^{.}{\cal T}_{X'/S}} \ar[r] \ar[d] & {\hat{\Gamma}.{\cal T}_{X'/S}} \ar[d]\\
	{{\cal D}_{X^{(m)}/S}^{(0)}} \ar[r] & {{\cal E}nd({\mathbb F}_{\cal J}({\cal E}))}.}
	\]
	Here the left vertical arrow is the $p$-curvature map and the right vertical arrow (resp. the lower horizontal arrow) is induced from the $\hat{\Gamma}.{\cal T}_{X'/S}$-action (resp. the ${\cal D}_{X^{(m)}/S}^{(0)}$-action) on
	${\mathbb F}_{\cal J}({\cal E})$.
	Note that, by assumption, the composite $S^{.}{\cal T}_{X'/S} \to \hat{\Gamma}.{\cal T}_{X'/S} \to {\cal E}nd({\mathbb F}_{\cal J}({\cal E}))$ is equals to the composite $S^{.}{\cal T}_{X'/S} \to 
	{\cal D}_{X/S}^{(m)} \to {\cal E}nd({\mathbb F}_{\cal J}({\cal E}))$, where the first arrow $S^{.}{\cal T}_{X'/S} \to {\cal D}_{X/S}^{(m)}$ is the $p^{m+1}$-curvature map.
	We may work in a local situation.
	Take $g\in {\mathbb F}_{\cal J}({\cal E})$.
	Then, by the formulas defined in Remark \ref{Remark20}, ${S^{.}{\cal T}_{X'/S}}$-action on ${\mathbb F}_{\cal J}({\cal E})	$ defined by the composite $S^{.}{\cal T}_{X'/S} \to {\cal D}_{X^{(m)}/S}^{(0)} \to {\cal E}nd({\mathbb F}_{\cal J}({\cal E}))$	is calculated by 
	\begin{eqnarray*}
	\underline{\partial}'_{<p\underline{\epsilon}_{i}>(0)}.g\left(1\otimes 1\right)=g(1\otimes \underline{\partial}'_{<p\underline{\epsilon}_{i}>(0)})=g\left(\underline{\partial}_{<p^{m+1}\underline{\epsilon}_{i}>(m)}.\left(1\otimes 1\right)\right)=\underline{\partial}_{<p^{m+1}\underline{\epsilon}_{i}>(m)}.g(1\otimes 1).
	\end{eqnarray*}
	So this is same as the action defined by the composite $S^{.}{\cal T}_{X'/S} \to {\cal D}_{X/S}^{(m)} \to {\cal E}nd({\mathbb F}_{\cal J}({\cal E}))$.
	Therefore these two ${S^{.}{\cal T}_{X'/S}}$-action are equal and
	${\mathbb F}_{\cal J}({\cal E})$ defines an object in ${\rm MIC}^{\cal B, J}_{\rm PD}(X^{(m)}/S)$ if ${\cal E}$ is an object in ${\rm MIC}^{\cal A, J}_{\rm PD}(X/S)$.
	One can check that a quasi-inverse ${\mathbb G}_{\cal J}$ of ${\mathbb F}_{\cal J}$ induces the functor ${\rm MIC}^{\cal B, J}_{\rm PD}(X^{(m)}/S)\to {\rm MIC}^{\cal A, J}_{\rm PD}(X/S)$ in a similar manner.
	 \end{proof}
Now, we are ready to state the main theorem in this section.
	\begin{theo}\label{theorem10}
	The following diagram of categories commutes:
	\[\xymatrix{
	{{\rm MIC}^{\cal A, J}_{\rm PD}(X/S)} \ar[rd]_{{\mathbb F}_{\cal J}} \ar[rr]^{C_{\cal X/S}} & & {\rm HIG}^{\cal B, J}_{\rm PD}(X'/S)\\
	& {\rm MIC}^{{\cal B, J}}_{\rm PD} (X^{(m)}/S)\ar[ru]_{C'} &}.
	\]
	\end{theo}
	\begin{rem}
	Theorem \ref{theorem10} can be regarded as the log global version of the result stated in Subsection 6.6 of \cite{GLQ}.  
	\end{rem}	
	
The key ingredient is the following. 
	\begin{theo}\label{theorem11}
	The image of $\check{{\cal K}}_{\cal X/S}^{(m), {\cal A}}$ under the Frobenius descent functor is naturally isomorphic to $\check{{\cal K}}_{\cal X/S}^{(0), {\cal B}}$,
	that is, 
	\begin{equation*}
	{\mathbb F}_{{\cal I}_{X}^{gp}}(\check{{\cal K}}_{\cal X/S}^{(m), {\cal A}})=\check{{\cal K}}_{\cal X/S}^{(0), {\cal B}}	.
	\end{equation*}
	\end{theo}
We shall give a proof of Theorem \ref{theorem11} in the next subsection.
Here, let us prove Theorem \ref{theorem10} from Theorem \ref{theorem11}.
We need the following easy lemma.
\begin{lemm}\label{lemma12}
Let ${\cal A}$ be an ${\cal I}$-indexed ${\cal O}_{X}$-algebra.
Let $\cal E$ be a ${\cal I}$-indexed left $\cal A$-module and $\cal F$ a $\cal J$-indexed left $\cal A$-module.
Then, for each $j\in {\cal J}(X)=Hom_{\cal I}(\cal I, J)$, ${\cal H}om_{\cal A}({\cal E, F})(j)$ is isomorphic to ${\cal H}om_{\cal A}({\cal E, F}(j))$.
\end{lemm}
\begin{proof}
There exists a natural morphism ${\cal H}om_{\cal A}({\cal E, F})(j) \to {\cal H}om_{\cal A}({\cal E, F}(j))$.
It is suffices to show that this morphism is an isomorphism at every fiber.
Let $i$ be a section of $\cal I$ defined on $U$. 
The left hand side is a sheaf defined by 
\begin{eqnarray*}
V\longmapsto \bigsqcup_{i'\in {\cal I}(V)}
{\rm Hom}_{{\cal A}|_{V}}({\cal E}|_{V}, {\cal F}(j)(i'))=\bigsqcup_{i'\in {\cal I}(V)}{\rm Hom}_{{\cal A}|_{V}}({\cal E}|_{V}, {\cal F}(i'+j)).
\end{eqnarray*}
Therefore the fiber at $i$ is isomorphic to ${\cal H}om_{\cal A}({\cal E, F}(i+j))$.
On the other hand, the fiber of  left hand side at $i$ is isomorphic to
${\cal H}om_{\cal A}({\cal E, F})_{i+j} \cong {\cal H}om_{\cal A}({\cal E, F}(i+j))$.
\end{proof}
Now, we prove Theorem \ref{theorem10}.
Let ${\cal E}$ be an object in ${\rm MIC}^{\cal A, J}_{\rm PD}(X/S)$ and consider the image
	\begin{equation*}
	C'({\mathbb F}_{\cal J}({\cal E}))={\cal H}om_{\tilde {{\cal D}}_{X^{(m)}/S}^{(0), \gamma}}(\check{{\cal K}}_{\cal X/S}^{(0), {\cal B}}, {\mathbb F}_{\cal J}({\cal E})).
	\end{equation*}
It is a correspondence
\begin{eqnarray*}
U\longmapsto \bigsqcup_{j\in {\cal J}(U)}
{\rm Hom}_{\tilde {{\cal D}}_{X^{(m)}/S}^{(0), \gamma}|_{U}}(\check{{\cal K}}_{\cal X/S}^{(0), {\cal B}}|_{U}, {\mathbb F}_{\cal J}({\cal E})(j))
\end{eqnarray*}
as a presheaf on $X$.
It is suffices to show that ${\rm Hom}_{\tilde {{\cal D}}_{X^{(m)}/S}^{(0), \gamma}|_{U}}(\check{{\cal K}}_{\cal X/S}^{(0), {\cal B}}|_{U}, {\mathbb F}_{\cal J}({\cal E})(j))$ is canonically identified with 
${\rm Hom}_{\tilde {{\cal D}}_{X/S}^{(m), \gamma}|_{U}}(\check{{\cal K}}_{\cal X/S}^{(m), {\cal A}}|_{U}, {\cal E}(j))$ for every $U$ and $j\in {\cal J}(U)$.
We may assume $U=X$.
Then, by Theorem \ref{theorem11} and Lemma \ref{lemma12}, we have 
\begin{eqnarray*}
{\rm Hom}_{\tilde {{\cal D}}_{X^{(m)}/S}^{(0), \gamma}}(\check{{\cal K}}_{\cal X/S}^{(0), {\cal B}}, {\mathbb F}_{\cal J}({\cal E})(j))
={\rm Hom}_{\tilde {{\cal D}}_{X^{(m)}/S}^{(0), \gamma}}( {\mathbb F}_{{\cal I}_{X}^{gp}}(\check{{\cal K}}_{\cal X/S}^{(m), {\cal A}}), {\mathbb F}_{{\cal I}_{X}^{gp}}({\cal E}(j))).
\end{eqnarray*}
The assertion follows from the fact that the Frobenius descent is an equivalence of categories.

\subsection{Proof}
In this subsection, we give a proof of Theorem \ref{theorem11}.
First, we introduce a description of ${\mathbb G}_{\cal J}$ in terms of crystals which is essentially due to Berthelot \cite{B1}.

Let $\cal E$ be a log $0$-crystal on ${\rm{CRIS}}_{\rm Int}^{(0)}(X^{(m)}/\tilde{S})$ and $(U, T, J, \delta)$ log $m$-PD thickening in ${\rm{CRIS}}_{\rm Int}^{(m)}(X/\tilde{S})$.
We consider a closed subscheme $T_{0}$ of $T$ defined by $J+p{\cal O}_{T}$ 
and endow $T_{0}$ with the inverse image log structure from $T$.
We have a log $0$-PD thickenings $T_{0}\hookrightarrow T$ and the following commutative diagram:
\[\xymatrix{
{X} \ar[d]  & U \ar[d]\ar[l]\ar[r] & T_{0}\ar[d] \\
X^{(m)}  & U^{(m)} \ar[l]\ar[r] & T_{0}^{(m)}.}
\]

Here all vertical arrows are the $m$-th relative Frobenius morphisms.
\begin{lemm}\label{lemma10}
There exists a unique morphism $T_{0}\to U^{(m)}$ such that the following diagram commutes
\[\xymatrix{
{U} \ar[d] \ar[r] & T_{0} \ar[d]\ar[ld] \\
U^{(m)} \ar[r] & T_{0}^{(m)} .}
\] 
\end{lemm}
\begin{proof}
Note that all fine log schemes in the diagram are homeomorphic.
Let $I$ denote the defining ideal of $U\hookrightarrow T$.
Since $(U, T, J, \delta)\in {\rm{CRIS}}_{\rm Int}^{(m)}(X/\tilde{S})$ (So $U$ and $T$ are integral over $\tilde{S}$), 
${\cal O}_{U^{(m)}}={\cal O}_{\tilde{S}}\otimes_{{\cal O}_{\tilde{S}}}{\cal O}_{T}/I{\cal O}_{\tilde{S}}\otimes_{{\cal O}_{\tilde{S}}}{\cal O}_{T}$, ${\cal O}_{T_{0}}={\cal O}_{T}/J+p{\cal O}_{T}$.
Consider the composition ${\cal O}_{T^{(m)}}\to {\cal O}_{T}\to {\cal O}_{T_{0}}$, where the first map is the $m$-th relative Frobenius and the second one is the natural projection.
Then, since the image of this map is contained in $I^{(p^{m})}+pI$, this one is zero on $I{\cal O}_{\tilde{S}}\otimes_{{\cal O}_{\tilde{S}}}{\cal O}_{T}$.
Thus ${\cal O}_{T^{(m)}}\to {\cal O}_{T_{0}}$ uniquely factors as ${\cal O}_{T^{(m)}}\to {\cal O}_{U^{(m)}}\to {\cal O}_{T_{0}}$.
\end{proof}
By Lemma \ref{lemma10}, $T_{0}\hookrightarrow T$ can be considered as an object in ${\rm{CRIS}}_{\rm Int}^{(0)}(X^{(m)}/\tilde{S})$.
We define a functor 
	\begin{equation*}
	{\mathbb G'}:
	\left(
		\begin{array}{c}
		\text{The category of} \\
		\text{log $0$-crystals on ${\rm{CRIS}}_{\rm Int}^{(0)}(X^{(m)}/\tilde{S})$}
		\end{array}
	\right)
	\to 
	\left(
		\begin{array}{c}
		\text{The category of} \\
		\text{log $m$-crystals on ${\rm{CRIS}}_{\rm Int}^{(m)}(X/\tilde{S})$}
		\end{array}
	\right)
	\end{equation*}
	by ${\mathbb G'}({\cal E})_{(U, T, J. \delta)}:={\cal E}_{(T_{0}, T, J+p{\cal O}_{T}, \delta)}$ for each $(U, T, J, \delta)\in {\rm{CRIS}}_{\rm Int}^{(m)}(X/\tilde{S})$.
	We also obtain a functor 
	\begin{equation*}
	\left(
		\begin{array}{c}
		\text{The category of} \\
		\text{log $0$-crystals on ${\rm{CRIS}}_{\rm Int}^{(0)}(X^{(m)}/S)$}
		\end{array}
	\right)
	\to 
	\left(
		\begin{array}{c}
		\text{The category of} \\
		\text{log $m$-crystals on ${\rm{CRIS}}_{\rm Int}^{(m)}(X/S)$}
		\end{array}
	\right)
	\end{equation*}	
	in a similar manner and denote it by $\mathbb G'$ by abuse of notations.

	\begin{lemm}\label{lemma20}
	The following diagram of categories commutes
	\[\xymatrix{
	{ \left(
		\begin{array}{c}
		\text{The category of} \\
		\text{log $0$-crystals on ${\rm{CRIS}}_{\rm Int}^{(0)}(X^{(m)}/S)$}
		\end{array}
	\right)
	} \ar[d] \ar[r]^{\mathbb G'} & 
	{\left(
		\begin{array}{c}
		\text{The category of} \\
		\text{log $m$-crystals on ${\rm{CRIS}}_{\rm Int}^{(m)}(X/S)$}
		\end{array}
	\right)} \ar[d] \\
	{ \left(
		\begin{array}{c}
		\text{The category of} \\
		\text{left ${\cal D}_{X^{(m)}/S}^{(0)}$-modules on $X^{(m)}$}
		\end{array}
	\right)
	} \ar[r]^{F^{*}} & 
	{ \left(
		\begin{array}{c}
		\text{The category of} \\
		\text{left ${\cal D}_{X/S}^{(m)}$-modules on $X$}
		\end{array}
	\right)
	}  .}
	\] 
	
	Here each of the vertical functors is defined by the composite of functors defined in Remark \ref{Remark8} and Remark \ref{Remark7}.	
	\end{lemm}

	\begin{proof}
	Let $\cal E$ be a log $0$-crystal on ${\rm{CRIS}}_{\rm Int}^{(0)}(X^{(m)}/S)$.
	Let us consider the following diagram:
	\[\xymatrix{
	X\ar[d] & X\ar[l]\ar[r] \ar[d] & (P^{n}_{X/S, (m)})_{0}\ar[ld]\ar@{^{(}->}[r]& P^{n}_{X/S, (m)} \ar@<1ex>[r]^{p_{0}}\ar[r]_{p_{1}}\ar[d]^{\Phi}& X\ar[d]\\
	X^{(m)}  & X^{(m)}\ar[l]\ar@{^{(}->}[rr] && P^{n}_{X^{(m)}/S, (0)} \ar@<1ex>[r]^{q_{0}}\ar[r]_{q_{1}}& X^{(m)},
	}
	\]
	where the first, second and the last vertical arrows are the $m$-th relative Frobenius morphisms, the slanting arrow is defined in Lemma \ref{lemma10}, $\Phi$ is explained in Subsection 4.1.2 and $p_{0}, p_{1}$ (resp. $q_{0}, q_{1}$) are the first and the second projection.
	Then the two squares in the diagram are clearly commutative.
	The triangle in the diagram is commutative by construction of the map $(P^{n}_{X/S, (m)})_{0}\to X^{(m)}$ (see Lemma \ref{lemma10}).
	For proving the commutativity of the middle trapezoid, we may work locally.
	Then the assertion follows from the local description of $\Phi^{*}$ (see (2) of Proposition \ref{proposition11}) and (1) of Proposition \ref{Proposition9}.
	Now, we see that $\Phi: \bigl((P^{n}_{X/S, (m)})_{0}\hookrightarrow P^{n}_{X/S, (m)}\bigr)\to \bigl(X^{(m)}\hookrightarrow P^{n}_{X^{(m)}/S, (0)}\bigr)$ is a morphism of log $0$-PD thickenings and, since $\cal E$ is a log $0$-crystal, we obtain 
	the canonical isomorphisms
		\begin{equation*}
		\Phi^{*}{\cal E}_{(X^{(m)}\hookrightarrow P^{n}_{X^{(m)}/S, (0)})}\xrightarrow{\cong} {\cal E}_{((P^{n}_{X/S, (m)})_{0}\hookrightarrow P^{n}_{X/S, (m)})} =:{\mathbb G'}({\cal E})_{(X\hookrightarrow P^{n}_{X/S, (m)})}.
		\end{equation*}
	Therefore the log $m$-PD stratification of ${\mathbb G'}({\cal E})$ is equal to that of $F^{*}({\cal E}_{X})$.
	\end{proof}
Now, we are ready to prove Theorem \ref{theorem11}.
By construction of the functor ${\mathbb G}_{{\cal I}^{gp}_{X}}$, Remark \ref{lemma25} and Remark \ref{lemma26} the following diagram of categories commutes:
	\[\xymatrix{
	{ \left(
		\begin{array}{c}
		\text{The category of} \\
		\text{left ${\cal D}_{X^{(m)}/S}^{(0)}$-modules on $X^{(m)}$}
		\end{array}
	\right)
	} \ar[d]_{\otimes{{\cal B}_{X/S}^{(m)}}} \ar[r]^{F^{*}} & 
	{\left(
		\begin{array}{c}
		\text{The category of} \\
		\text{left ${\cal D}_{X/S}^{(m)}$-modules on $X$}
		\end{array}
	\right)} \ar[d]^{\otimes{\cal A}_{X}^{gp}} \\
	{ \left(
		\begin{array}{c}
		\text{The category of} \\
		\text{left $\tilde{{\cal D}}_{X^{(m)}/S}^{(0)}$-modules on $X^{(m)}$}
		\end{array}
	\right)
	} \ar[r]^{{\mathbb G}_{{\cal I}_{X}^{gp}}} & 
	{ \left(
		\begin{array}{c}
		\text{The category of} \\
		\text{left $\tilde{{\cal D}}_{X/S}^{(m)}$-modules on $X$}
		\end{array}
	\right).
	}  }
	\] 
So it suffices to show that $F^{*}{\cal K}^{(0)}_{\cal X/S}$ is isomorphic to ${\cal K}^{(m)}_{\cal X/S}$.
Since the following diagram of categories commutes:
	\[\xymatrix{
	{ \left(
		\begin{array}{c}
		\text{The category of} \\
		\text{$p$-torsion log $0$-crystals}\\
		\text{ on ${\rm{CRIS}}_{{\rm Int},f}^{(0)}(X^{(m)}/\tilde{S})$}
		\end{array}
	\right)
	} \ar[d]^{\mathbb G'} \ar[r]^{A}& 
	{\left(
		\begin{array}{c}
		\text{The category of} \\
		\text{log $0$-crystals}\\
		\text{ on ${\rm{CRIS}}_{\rm Int}^{(0)}(X^{(m)}/S)$}		
		\end{array}
	\right)
	} \ar[d]^{\mathbb G'} \ar[r]&
	{ \left(
		\begin{array}{c}
		\text{The category of} \\
		\text{left ${\cal D}_{X^{(m)}/S}^{(0)}$-modules}\\
		\text{ on $X^{(m)}$}
		\end{array}
	\right)
	}\ar[d]^{F^{*}} \\	
	{ \left(
		\begin{array}{c}
		\text{The category of} \\
		\text{$p$-torsion log $m$-crystals}\\
		\text{ on ${\rm{CRIS}}_{{\rm Int},f}^{(m)}(X/\tilde{S})$}
		\end{array}
	\right)
	} \ar[r]^{A'} & 
	{ \left(
		\begin{array}{c}
		\text{The category of} \\
		\text{log $m$-crystals}\\
		\text{ on ${\rm{CRIS}}_{\rm Int}^{(m)}(X/S)$}		
		\end{array}
	\right)
	}  \ar[r] &
	{ \left(
		\begin{array}{c}
		\text{The category of} \\
		\text{left ${\cal D}_{X/S}^{(m)}$-modules}\\
		\text{ on $X$}
		\end{array}
	\right),
	}	
	}
	\] 
where $A$ and $A'$ are equivalences of categories in Lemma \ref{Lem7} and the right square is same as Lemma \ref{lemma20},
we shall show that the functor $\mathbb G'$ sends the $p$-torsion log $0$-crystal ${\cal K}^{(0)}_{{\cal X/S}}$ to the $p$-torsion log $m$-crystal ${\cal K}^{(m)}_{{\cal X/S}}$.
Let ($\tilde{U}, \tilde{T}, \tilde{J}, \tilde{\delta}$) be a $m$-PD thickening in ${\rm{CRIS}}_{{\rm Int},f}^{(m)}(X/\tilde{S})$.
Let $T$ denote the reduction of $\tilde{T}$ modulo $p$.
Let us consider the following commutative diagram:
	\[\xymatrix{
		{X} \ar[d]& {\tilde{U}} \ar[l]\ar@{^{(}->}[r]\ar[d] & {\tilde{T}_{0}} \ar@{^{(}->}[r]\ar[dd]\ar[ld]_{a} &\ar@{}[rd]|{\square}{T}\ar[ldd]_{c}\ar[r]^{\text{mod }p}\ar[lldd]_{b}\ar[d] & {\tilde{T}}\ar[d]\\
		{X^{(m)}}\ar[d] & {\tilde{U^{(m)}}}\ar[l]\ar[d] & & {S}\ar[r]^{\text{mod }p} & {\tilde{S}} & \\
		{X'} & {\tilde{U'}} \ar[l]_{e} & {\tilde{T}^{(1)}_{0}.}\ar@/^/[ll]^{d}
	}\]
Here the slanting arrow $a$ is defined in Lemma \ref{lemma10}, the slanting arrows $b$ and $c$ are defined in Lemma \ref{Lemma2}
and the arrow $d: \tilde{T}^{(1)}_{0}\to X'=\left(X^{(m)}\right)^{(1)}$ is induced from the composition $\tilde{T}_{0} \xrightarrow{a} \tilde{U^{(m)}}\to X^{(m)}$ via the functoriality of the first relative Frobenius morphism.
Then, by definition,
$\left({\mathbb G}'\left({\cal L}_{{\cal X/S}}^{(0)}\right)\right)_{(\tilde{U}\hookrightarrow\tilde{T})}=\left({\cal L}_{{\cal X/S}}^{(0)}\right)_{(\tilde{T}_{0}\hookrightarrow\tilde{T})}$
is the \'etale sheaf of sets on $\tilde{T}$ of local liftings of the composition $d \circ c$.
On the other hand, $\left({\cal L}_{{\cal X/S}}^{(m)}\right)_{(\tilde{U}\hookrightarrow\tilde{T})}$ is the \'etale sheaf of sets on $\tilde{T}$ of local liftings of the composition $f_{T/S}:=e\circ b$.
Therefore the equality $\left({\mathbb G}'\left({\cal L}_{{\cal X/S}}^{(0)}\right)\right)_{(\tilde{U}\hookrightarrow\tilde{T})}=\left({\cal L}_{{\cal X/S}}^{(m)}\right)_{(\tilde{U}\hookrightarrow\tilde{T})}$ follows from
the fact that the morphism $d \circ c$ is equals to the morphism $e\circ b$. 
Actually this is the identification as a torsor over $f^{*}_{T/S}{\cal T}_{X'/S}$.
So we have $\left({\mathbb G}'\left({\cal E}^{(0)}_{{\cal X/S}}\right)\right)_{(\tilde{U}\hookrightarrow\tilde{T})}=\left({\cal E}^{(m)}_{{\cal X/S}}\right)_{(\tilde{U}\hookrightarrow\tilde{T})}$ and
$\left({\mathbb G}'\left({\cal K}^{(0)}_{{\cal X/S}}\right)\right)_{(\tilde{U}\hookrightarrow\tilde{T})}=\left({\cal K}^{(m)}_{{\cal X/S}}\right)_{(\tilde{U}\hookrightarrow\tilde{T})}$.
Consequently we obtain the equality ${\mathbb G}'\left({\cal K}^{(0)}_{{\cal X/S}}\right)={\cal K}^{(m)}_{{\cal X/S}}$ as a $p$-torsion log $m$-crystal.
This finishes the proof.





\begin{thebibliography}{99}
	\bibitem[B1]{B1}Pierre Berthelot, {\it Berthelot letter to Illusie}, (1990).
	\bibitem[B2]{B2}Pierre Berthelot, {\it ${\cal D}$-modules arithm\'etiques I. Op\'erateurs diff\'erentiels de niveau fini.} Ann. Sci. \'Ecole Norm. Sup. (4), 29(2) (1996), 185-272.	
	\bibitem[B3]{B3}Pierre Berthelot, {\it $\cal D$-module arithm\'etiques II. Descente par Frobenius}, M\'em. Soc. Math. France 81 (2000).
	\bibitem[GLQ]{GLQ}Michel Gros, Bernard Le Stum and Adolfo Quir\'os, {\it A Simpson correspondance in positive characteristic}, Publ. Res. Inst. Math. Sci. 46 (2010), 1-35.		
	\bibitem[K]{K}Kazuya Kato, {\it Logarithmic structures of Fontaine-Illusie}, in Algebraic analysis, geometry, and number theory (Baltimore, MD, 1988), 191-224, Johns Hopkins Univ. Press, Baltimore, MD (1989).	
	\bibitem[LQ]{LQ}Bernard Le Stum and Adolfo Quir\'os. {\it Transversal crystals of finite level}. Ann. Inst. Fourier (Grenoble), 47(1) (1997), 69-100.	
	\bibitem[L]{L}Pierre Lorenzon, {\it Indexed algebras associated to a log structure and a theorem of p-descent on log schemes}, Manuscripta Mathematica 101 (2000), 271-299.
	\bibitem[Mi]{Mi}Kazuaki Miyatani, {\it On the finitude of logarithmic crystalline cohomology of higher level}, master thesis (2009).
	\bibitem[M]{M}Claude Montagnon, {\it G\'en\'eralisation de la th\'eorie arithm\'etique des d-modules $\grave{a}$ la g\'eom\'etrie logarithmique}, Ph.D. thesis, L'universit\'e de Rennes I (2002).
	\bibitem[OV] {OV}Arthur Ogus and Vladimir Vologodsky, {\it Nonabelian Hodge theory in characteristic} $p$. Publ. Math. Inst. Hautes \'Etudes Sci. 106 (2007), 1-138.			
	\bibitem[S]{S}Daniel Schepler, {\it Logarithmic nonabelian Hodge theory in characteristic p}, Ph.D. thesis (2005).	
	\bibitem[Si]{Si}Carlos T. Simpson, {\it Higgs bundles and local systems}. Inst. Hautes \'Etudes Sci. Publ.Math., 75 (1992), 5-95.				
\end{thebibliography}
\end{document}